\documentclass[11pt,leqno]{article}
\usepackage{amsmath, amscd, amsthm, amssymb, graphics, xypic, mathrsfs, setspace, fancyhdr, times, bm, enumitem}
\usepackage[letterpaper,top=1.05in, bottom=1.05in, left=1.05in, right=1.05in]{geometry}
\usepackage[colorlinks=true,pagebackref=true]{hyperref} % new hyperref package added by Jean
\hypersetup{backref}

\newcommand{\tensor}{\otimes}
\newcommand{\colim}{\operatorname{colim}}
\newcommand{\Spec}{\operatorname{Spec}}
\newcommand{\Aut}{{\mathbf{Aut}}}
\newcommand{\isomto}{{\stackrel{\sim}{\;\longrightarrow\;}}}
\newcommand{\isomt}{{\stackrel{{\scriptscriptstyle{\sim}}}{\;\rightarrow\;}}}

\newcommand{\sma}{{\scriptstyle{\wedge}}}

\newcommand{\Singaone}{\operatorname{Sing}^{\aone}\!\!}
\newcommand{\Map}{\operatorname{Map}}

\renewcommand{\hom}{\operatorname{Hom}}

\newcommand{\real}{{\mathbb R}}
\newcommand{\cplx}{{\mathbb C}}
\newcommand{\Q}{{\mathbb Q}}
\newcommand{\Z}{{\mathbb Z}}

\newcommand{\aone}{{\mathbb A}^1}
\newcommand{\pone}{{\mathbb P}^1}

\newcommand{\gm}[1]{{{\mathbf G}_{m}^{#1}}}

\newcommand{\MW}{\mathrm{MW}}

\newcommand{\et}{\text{\'et}}
\newcommand{\ho}[1]{\mathrm{Ho}_{#1}}

\newcommand{\DM}{{\mathrm{DM}}}

\newcommand{\bpi}{\bm{\pi}}

\newcommand{\eff}{\operatorname{eff}}
\newcommand{\D}{{\mathrm D}}
\newcommand{\Daone}[1]{{\mathrm D}_{\aone}^{\eff}(#1)}

\newcommand{\Nis}{\operatorname{Nis}}
\newcommand{\Laone}{\operatorname{L}_{\aone}}
\newcommand{\LR}{\operatorname{L}_R}

\newcommand{\Sm}{\mathrm{Sm}}
\newcommand{\Cor}{\mathrm{Cor}}
\newcommand{\Spc}{\mathrm{Spc}}
\newcommand{\Mod}{\mathrm{Mod}}
\newcommand{\Ab}{\mathrm{Ab}}
\newcommand{\Grp}{\mathrm{Grp}}

\newcommand{\K}{{{\mathbf K}}}

\renewcommand{\H}{{{\mathbf H}}}

\newcommand{\sPre}{\mathrm{sPre}}

\newcommand{\im}{\operatorname{im}}

\newcommand{\Addresses}{{% additional braces for segregating \footnotesize
 \bigskip
 \footnotesize

 A.~Asok, Department of Mathematics, University of Southern California, 3620 S. Vermont Ave.,
  Los Angeles, CA 90089-2532, United States; \textit{E-mail address:} \url{asok@usc.edu}

  \medskip

 J.~Fasel, Institut Fourier - UMR 5582, Universit\'e Grenoble Alpes, 100, rue des math\'ematiques, F-38402 Saint Martin d'H\`eres; France \textit{E-mail address:} \url{jean.fasel@gmail.com}

 \medskip

 M.J.~Hopkins, Department of Mathematics, Harvard University, One Oxford Street, Cambridge, MA 02138, United States \textit{E-mail address:} \url{mjh@math.harvard.edu}
}}

\newcounter{intro}
\theoremstyle{plain}
\newtheorem{thm}{Theorem}[subsection]

\newtheorem{lem}[thm]{Lemma}
\newtheorem{cor}[thm]{Corollary}
\newtheorem{prop}[thm]{Proposition}
\newtheorem*{claim*}{Claim}  %Claim

\newtheorem*{thm*}{Theorem}
\newtheorem*{problem*}{Problem}

\newtheorem{thmintro}{Theorem}

\theoremstyle{definition}
\newtheorem{defn}[thm]{Definition}

\newtheorem{notation}[thm]{Notation}

\theoremstyle{remark}
\newtheorem{rem}[thm]{Remark}
\newtheorem{remintro}[thmintro]{Remark}

\newtheorem{ex}[thm]{Example}

\numberwithin{equation}{section}

\begin{document}
\pagestyle{fancy}
\renewcommand{\sectionmark}[1]{\markright{\thesection\ #1}}
\fancyhead{}
\fancyhead[LO,R]{\bfseries\footnotesize\thepage}
\fancyhead[LE]{\bfseries\footnotesize\rightmark}
\fancyhead[RO]{\bfseries\footnotesize\rightmark}
\chead[]{}
\cfoot[]{}
\setlength{\headheight}{1cm}

\author{Aravind Asok\thanks{Aravind Asok was partially supported by National Science Foundation Awards DMS-1254892 and DMS-1802060.} \and Jean Fasel \and Michael J. Hopkins\thanks{Michael J. Hopkins was partially supported by National Science Foundation Awards DMS-0906194, DMS-1510417 and DMS-1810917.}}

\title{{\bf Localization and nilpotent spaces in $\aone$-homotopy theory}}
\date{}
\maketitle

\begin{abstract}
For a subring $R$ of the rational numbers, we study $R$-localization functors in the local homotopy theory of simplicial presheaves on a small site and then in $\aone$-homotopy theory.  To this end, we introduce and analyze two notions of nilpotence for spaces in $\aone$-homotopy theory paying attention to future applications for vector bundles.  We show that $R$-localization behaves in a controlled fashion for the nilpotent spaces we consider.  We show that the classifying space $BGL_n$ is $\aone$-nilpotent when $n$ is odd, and analyze the (more complicated) situation where $n$ is even as well.  We establish analogs of various classical results about rationalization in the context of $\aone$-homotopy theory: if $-1$ is a sum of squares in the base field, ${\mathbb A}^n \setminus 0$ is rationally equivalent to a suitable motivic Eilenberg--Mac Lane space, and the special linear group decomposes as a product of motivic spheres.
\end{abstract}

\begin{footnotesize}
\setcounter{tocdepth}{2}
\tableofcontents
\end{footnotesize}

\section{Introduction}
In this paper we build some foundations for localization with respect to a set of primes in the Morel--Voevodsky unstable $\aone$-homotopy theory \cite{MV} and study some applications.  Our approach proceeds largely by analogy with the classical topological story \cite{QuillenRHT, BK, Sullivan}.

The Morel--Voevodsky approach to the construction of the $\aone$-homotopy category is a two stage localization.  One begins with a category of spaces: we choose the model of simplicial presheaves on the category of smooth schemes.  At the first stage, one formally inverts ``Nisnevich local" weak equivalences to obtain a ``local homotopy category".  At the second stage, one formally inverts projection maps with one factor the affine line to obtain the unstable $\aone$-homotopy category.

If $R$ is a subring of the rational numbers, then our goal will be to further localize the $\aone$-homotopy category to obtain an $R$-local $\aone$-homotopy category.  While such localizations formally exist, as in the classical story, we want to isolate a class of spaces for which $R$-localization is ``well-behaved".  In the classical situation, for sufficiently nice spaces (e.g., nilpotent spaces), ``well-behaved" means that localization preserves fiber sequences, and has the effect of simply tensoring homotopy groups with $R$.

Following the two-stage construction of the $\aone$-homotopy category above, we begin in Section~\ref{s:nilpotencelocalhomotopy} by analyzing nilpotence in the more general context of the local homotopy theory of simplicial presheaves on a site (with enough points in the topos-theoretic sense).  In this context, there is a ``strong" notion of nilpotence, which is analogous to the classical notion in the sense that Postnikov towers for spaces can be refined to a tower of principal fibrations.  However, we also introduce broader ``local" notion of nilpotence, where local nilpotence means ``stalkwise" nilpotence, which is natural from the standpoint of local homotopy theory.  Section~\ref{s:nilpotencelocalhomotopy} closes with a quick discussion of $R$-localization in the local homotopy theory of simplicial presheaves.

In Section~\ref{s:nilpotenceinaonehomotopytheory}, we begin our analysis of nilpotence in $\aone$-homotopy theory.  After reviewing foundational structural facts about $\aone$-homotopy sheaves, the bulk of Sections~\ref{ss:aonelocalgrouptheory} and \ref{ss:aonenilpotentgroups} is devoted to an analysis of the relevant ``group theory" in the context of $\aone$-homotopy theory.  Then we analyze variants of the above notions of nilpotence in $\aone$-homotopy theory (see Definition~\ref{defn:aonenilpotentspace}).

A basic motivation for our analysis is that the analogs of spaces that are simple or even simply-connected in classical algebraic topology can fail to be so in $\aone$-algebraic topology.  For example, the projective line or, more generally, the variety parameterizing complete flags in a vector space are $\aone$-nilpotent (see Theorem~\ref{thm:flagmanifoldaonenilpotent}).  However, these examples typically fail to be $\aone$-simple because their $\aone$-fundamental sheaves of groups may fail to be abelian (see Remark~\ref{rem:nonabelianaonenilpotent} for further discussion).

In Section~\ref{s:nilpotenceaonehomology}, we then proceed to analyze $R$-localization in $\aone$-homotopy theory, keeping our motivating examples in mind.  Classically, $R$-local equivalences can be characterized as maps that induce isomorphisms on $R$-local homology.  The natural analog of ``homology" in the context of $\aone$-homotopy theory is the $\aone$-homology theory studied by F. Morel in \cite[\S 6.2]{MField}.  Thus, we analyze morphisms of spaces that induce isomorphisms on $R$-local $\aone$-homology.  

\begin{remintro}
For spaces that are ($\aone$-)simply connected, there is a tight connection between $\aone$-homotopy and $\aone$-homology.  In such sitations, our approach agrees with previous approaches to localization, e.g., \cite[\S 3]{WickelgrenWilliams}, which use equivalent alternative approaches circumventing discussion of $\aone$-homology.  On the other hand, more recently, G. Guzman \cite{Guzman} has studied localization of the unstable motivic homotopy category with respect to $\aone$-homology.  Her focus is rather different from ours, as her goal is to obtain a ``coalgebraic" description of the associated local homotopy category along the lines of \cite{QuillenRHT} or \cite{Goerss}.
\end{remintro}

A significant portion of Section~\ref{s:nilpotenceaonehomology} is devoted to analyzing the relationship between nilpotence and $\aone$-homology.  This section closes with a construction of a model for $R$-localization that behaves ``reasonably" for suitable nilpotent spaces (see Theorems~\ref{thm:aonelocalizationofnilpotentspaces} and \ref{thm:Raonelocalizationoffibersequences}): for example, the effect on higher $\aone$-homotopy sheaves of $R$-localization amounts to tensoring with $R$).

\begin{remintro}
Morel's $\aone$-homology theory does {\em not} coincide with motivic homology defined via Voevodsky's triangulated category of motives \cite{MVW}.  For example, if we consider the motivic Hopf map $\eta: {\mathbb A}^2 \setminus 0 \to \pone$, then the map $\eta_*$ in motivic homology is the zero map but, as Morel observed, non-zero in $\aone$-homology \cite[Remark 3.12]{MICM}.  Moreover, this theory is {\em not} $\pone$-stable, i.e., representable in the stable $\aone$-homotopy category of $\pone$-spectra (if it was, then the $\aone$-homology sheaves of a space in the sense of Morel could be equipped with transfers of a suitable form, but they cannot always be equipped with this structure; see \cite[\S 2]{LevineSandT} for an example).  
\end{remintro}

Section~\ref{s:applications} focuses on the main applications of our constructions: we establish analogs of some classical results of J.-P. Serre.  First, Serre showed that after inverting sufficiently many primes, compact Lie groups split as products of odd-dimensional spheres \cite{Serre}.  Serre obtained this decomposition by analyzing explicitly the computation of the cohomology of compact Lie groups. An algebro-geometric analog of a compact Lie group is a split reductive group; once-punctured affine spaces are algebro-geometric analogs of odd-dimensional spheres.  With these analogies in mind, we establish the following result.

\begin{thmintro}[See Theorem \ref{thm:suslinsplitting}]
\label{thmintro:suslinsplitting}
Suppose $k$ is a field that is not formally real and $n \geq 2$ is an integer.  After inverting $(n-1)!$, there is an $\aone$-weak equivalence of the form
\[
{\mathbb A}^2 \setminus 0 \times {\mathbb A}^3 \setminus 0 \cdots \times {\mathbb A}^n \setminus 0 \cong SL_n.
\]
\end{thmintro}

\begin{remintro}
\label{remintro:suslinsplitting}
Our proof of this fact is rather different from the classical proof as we make no direct cohomology computations.  While the rational Voevodsky motive (or motivic cohomology) of split reductive groups is understood \cite{Gross,Biglari}, the $\aone$-homology has only been analyzed in low-degrees \cite[Appendix A]{MFM}.  We are forced to construct the summands by means of explicit maps and we do this by studying ``Suslin matrices" (introduced in \cite{Suslinstablyfree}).  Appealing to (complex) realization, our proof thus yields a new proof of Serre's classical splitting in this case.
\end{remintro}

Serre also showed that the homotopy groups of $S^{2n-1}$ are finite except in degree equal $2n-1$.  Likewise, the homotopy groups of $S^{2n}$ are finite except in degrees $2n$ and $4n-1$.  The first result is a juxtaposition of two results: one shows that odd-dimensional spheres are rationally equivalent to Eilenberg--Mac Lane spaces by analysis of the homotopy fiber of the map $S^{2n-1} \to K(\Z,2n-1)$ induced by the fundamental class (equivalently, the first non-trivial stage of the Postnikov tower), and one establishes that Eilenberg--Mac Lane spaces have finitely generated cohomology by use of Serre's class theory.  The second result follows from the first by construction of a fiber sequence of the form
\[
S^{2n} \longrightarrow K(\Z,2n) \longrightarrow K(\Z,4n).
\]
We establish some motivic analogs of these results.  To phrase the results in a uniform geometric fashion, we use the split affine quadrics $Q_{2n-1}$ (defined by setting the even-dimensional hyperbolic form equal to $1$) and its corresponding even dimensional analog $Q_{2n}$ (this is not quite obtained by setting the odd-dimensional standard split hyperbolic form equal to $1$, but it is isomorphic to this if $k$ has characteristic not equal to $2$; see the beginning of Section~\ref{ss:Suslin} for more details).

\begin{thmintro}[See Theorems \ref{thm:rationalizedmotivicspheres} and \ref{thm:fibersequenceevenspheres}]
\label{thmintro:rationalizedspheres}
Assume $k$ is a field that is not formally real.
\begin{enumerate}[noitemsep,topsep=1pt]
\item The map
\[
Q_{2n-1} \longrightarrow K(\Z(n),2n-1)
\]
induced by the fundamental class in motivic cohomology is an $\aone$-equivalence for $n = 1$, and a rational $\aone$-weak equivalence if $n \geq 2$.
\item There is a rational $\aone$-fiber sequence of the form
\[
Q_{2n} \longrightarrow K(\Z(n),2n) \longrightarrow K(\Z(2n),4n),
\]
where the first map is the fundamental class in motivic cohomology, and second map is the squaring map.
\end{enumerate}
\end{thmintro}

\begin{remintro}
The classical proof of the first result uses the Serre spectral sequence to understand the cohomology of the $(2n-1)$-fold connective cover of $S^{2n-1}$.  In the motivic context, we do not have a Serre spectral sequence of the necessary form.  Our proof is by necessity quite different from the classical proof and proceeds by appeal to Theorem~\ref{thmintro:suslinsplitting} and rational degeneration of the motivic spectral sequence linking algebraic K-theory and motivic cohomology.  Moreover, our result is an unstable analog of Morel's equivalence between the rationalized stable $\aone$-homotopy category and Voevodsky's category of motives (assuming $-1$ is a sum of squares in $k$).  Similar comments apply to the even-dimensional case.
\end{remintro}

After Theorem~\ref{thmintro:rationalizedspheres}, the motivic analog of the finite-generation of the cohomology of classical Eilenberg--Mac Lane spaces is closely related to the Beilinson--Soul\'e vanishing conjecture and the unstable $\aone$-homotopy theory of Voevodsky's motivic Eilenberg--Mac Lane spaces.  We refer the reader to Corollary~\ref{cor:beilinsonsoule} and the surrounding material for further discussion of this relationship.

With the exception of the case $n = 1$ of Theorem~\ref{thmintro:rationalizedspheres}(2), all of the results stated above could be established with a more restrictive version of $R$-localization (i.e., for ``simple" spaces).  Nevertheless, allowing nilpotent spaces yields considerable additional flexibility, e.g., good control of fiber sequences under $R$-localization (see Remark~\ref{rem:homotopyfiberofmapofsimplespacescanbenilpotent} for further explanations).  As a consequence, the bulk of this paper is devoted to developing localization in generality suitable for future geometric applications.  For us, key among these observations is the fact that $BGL_n$ has a ``well-behaved" rationalization, either if $n$ is odd or if $n$ is even and $-1$ is a sum of squares in the base field (combine Theorems~\ref{thm:bglnoddaonesimple}, \ref{thm:bglevenlocallyaonenilpotent} and \ref{thm:aonelocalizationofnilpotentspaces}); these results use the full strength of the techniques developed here.  A different set of applications will be presented in \cite{AsokFaselHopkinsplocal}, where we will construct low rank vector bundles on smooth affine varieties that are $\aone$-weakly equivalent to projective spaces.  We refer the reader to the introduction to each section for a more detailed description of its contents.

\subsubsection*{Acknowledgements}
The authors would like to thank Marc Hoyois for various discussions around the notion of nilpotence in $\aone$-homotopy theory and Fr{\' e}d{\' e}ric D{\'e}glise for discussions around the Beilinson--Soul{\'e} vanishing conjecture (in particular, for providing us with Lemma~\ref{lem:BSsubfields}).  The authors owe an intellectual debt to Peter Hilton, Guido Mislin and Joseph Roitberg: while there is very little explicit reference to the book \cite{HMR} elsewhere in this work, our presentation of nilpotence is inspired by theirs.  Finally, the authors would like to thank the referees for numerous corrections, questions and helpful comments that clarified the exposition.

\subsubsection*{Preliminaries/Notation}
Throughout this paper $k$ will denote a fixed base field and $R$ will be a commutative unital ring (typically a subring of the field of rational numbers $\Q$).  We remind the reader that a field $k$ is called {\em formally real} if $-1$ is not a sum of squares in $k$ and {\em not formally real} otherwise; these notions will reappear later in the paper and we refer the reader to \cite[\S 95]{EKM} for a discussion of some properties of formally real fields.

If $\mathbf{C}$ is a small category, then we use the following notation:
\begin{itemize}[noitemsep,topsep=1pt]
\item $\Ab(\mathbf{C})$ - the category of presheaves of abelian groups on $\mathbf{C}$;
\item $\Mod_R(\mathbf{C})$ - the category of presheaves of $R$-modules on $\mathbf{C}$;
\item $\Grp(\mathbf{C})$ - the category of presheaves of groups on $\mathbf{C}$;
\item $\sPre(\mathbf{C})$ for the category of simplicial presheaves on $\mathbf{C}$.
\end{itemize}
If $\mathbf{C}$ is equipped with a Grothendieck topology $\tau$, then we use the following notation:
\begin{itemize}[noitemsep,topsep=1pt]
\item $\Ab_{\tau}(\mathbf{C})$ - the full subcategory of $\Ab(\mathbf{C})$ consisting of $\tau$-sheaves of abelian groups;
\item $\Mod_{R,\tau}(\mathbf{C})$ - the full subcategory of $\Mod_R(\mathbf{C})$ consisting of $\tau$-sheaves of $R$-modules.
\item $\Grp_{\tau}(\mathbf{C})$ - the full subcategory of $\Grp(\mathbf{C})$ consisting of $\tau$-sheaves of groups.
\end{itemize}

We will typically use script letters $\mathscr{X},\mathscr{Y}$ for objects in $\sPre(\mathbf{C})$.  If $\mathbf{C}$ is a pointed category with terminal object $\ast$, (i.e., $\mathbf{C}$ has an initial object and the canonical map from the initial to the terminal object is an isomorphism), then we write
\begin{itemize}
\item $\sPre(\mathbf{C})_{\ast}$ for the category of pointed objects in $\sPre(\mathbf{C})$,
\end{itemize}
i.e., pairs $(\mathscr{X},x)$ where $\mathscr{X}$ and $x: \ast \to \mathscr{X}$ is a morphism.

If $(\mathbf{C},\tau)$ is a site, $\sPre(\mathbf{C})$ can be equipped with a $\tau$-local model structure (see Section \ref{ss:nilpotenceinlocalhomotopy} for more details); we write $\ho{\tau}(\mathbf{C})$ for the associated homotopy category and $\ho{\tau,\ast}(\mathbf{C})$ for the pointed homotopy category.  We write $\mathrm{R}_{\tau}$ for a fixed functorial fibrant replacement functor for the $\tau$-local model structure and given $\mathscr{X},\mathscr{Y} \in \sPre(\mathbf{C})$, we write $\Map(\mathscr{X},\mathscr{Y})$ for the associated derived mapping space.

If $\mathbf{G} \in \Grp_{\tau}(\mathbf{C})$, we write $B\mathbf{G} \in \sPre(\mathbf{C})$ for the associated classifying space; in brief, $\mathbf{G}$-torsors on spaces may be described in terms of maps in $\ho{\tau}(\mathbf{C})$; we refer the reader to \cite[Theorem 9.8]{Jardine} in particular, and \cite[Chapter 9]{Jardine} more generally for further discussion of this point.  We also write $B_{\tau}\mathbf{G}$ for a $\tau$-local fibrant replacement for $B\mathbf{G}$; we refer the reader to \cite[\S 2.3]{AHWII} for more discussion of this notation.  

Similarly, if $\mathbf{A} \in \Ab_{\tau}(\mathbf{C})$, and $n \geq 0$ is an integer, we write $K(\mathbf{A},n) \in \sPre(\mathbf{C})$ for the associated Eilenberg--Mac Lane space \cite[p. 212]{Jardine}; such objects represent sheaf cohomology with coefficients in $\mathbf{A}$ in $\ho{\tau}(\mathbf{C})$ \cite[Theorem 8.26]{Jardine}.  We will also use the notation $K(\mathbf{A},n)$ when $\mathbf{A}$ is a chain complex of $\tau$-sheaves of abelian groups situated in positive (homological degrees).  In this case, $K(\mathbf{A},n)$ represents hypercohomology with coefficients in $\mathbf{A}$; see \cite[Chapter 8]{Jardine} for more details.

We write $\Sm_k$ for the category of schemes that are separated, smooth and have finite type over $\Spec k$.  When $\mathbf{C} = \Sm_k$, we set $\Spc_k := \sPre(\Sm_k)$.  Typically, we will take $\mathbf{C} = \Sm_k$ equipped with the Nisnevich topology, and if no topology is mentioned, the word ``sheaf" should be taken to mean Nisnevich sheaf.  We set $\Ab_k := \Ab_{\Nis}(\Sm_k)$ and $\Grp_k := \Grp_{\Nis}(\Sm_k)$.  When $(\mathbf{C},\tau) = (\Sm_k,\Nis)$, the category $\ho{\Nis}(\Sm_k)$ ($\ho{\Nis,\ast}(\Sm_k)$) will be called the {\em (pointed) simplicial homotopy category}.

The Morel--Voeovodsky $\aone$-homotopy category $\ho{k}$ is obtained from $\ho{\Nis}(\Sm_k)$ by a further Bousfield localization (see Section \ref{ss:aonelocalgrouptheory} for more details); the pointed $\aone$-homotopy category is obtained similarly, beginning with the pointed simplicial homotopy category $\ho{\Nis,\ast}(\Sm_k)$ instead.  If $\mathscr{X},\mathscr{Y} \in \Spc_k$, we set
\[
[\mathscr{X},\mathscr{Y}]_{\aone} := \hom_{\ho{k}}(\mathscr{X},\mathscr{Y}),
\]
and similar notation will be used for pointed $\aone$-homotopy classes of maps.

\section{Nilpotence and $R$-localization in local homotopy theory}
\label{s:nilpotencelocalhomotopy}
Suppose $R \subset \Q$ is a subring.  The main goal of this section is to study a notion of $R$-localization of simplicial presheaves, paying close attention to the behavior of localization for ``nilpotent spaces" in a sense we introduce here.

In Section \ref{ss:nilpotentandrlocalgroups}, we begin by formulating sheaf-theoretic notions of nilpotent and $R$-local sheaves of groups.  In Section \ref{ss:nilpotenceinlocalhomotopy}, we review basic constructions related to the local homotopy theory of simplicial presheaves on a small Grothendieck site \cite{Jardine}.  In the case where the site has ``enough points", we introduce various notions of nilpotent spaces (or maps) in this context.  In Section \ref{ss:rlocalizationofsimplicialpresheaves}, we analyze functorial $R$-localizations for nilpotent spaces in the situation under consideration.

Presumably, much of what we discuss in this section could be formulated in an arbitrary $\infty$-topos \cite{Lurie}, but we restrict our attention to simplicial presheaves on a small Grothendieck site both to simplify the discussion and because that is all we will use in our applications to $\aone$-homotopy theory.

Versions of some facts here about $R$-localization were worked out in \cite[\S 3]{WickelgrenWilliams}, though they mainly analyze localization functors for the analogs of simply connected or, more generally, simple spaces.  Moreover, it was observed there that various inequivalent versions of the theory exist, just as in the classical situation.  Indeed, for non-nilpotent spaces, there are a number of inequivalent $R$-localization functors (see Remark~\ref{rem:inequivalentversions} for further discussion and references).

In any case, upon choosing a model of functorial $R$-localization for simplicial sets, one may then define $R$-localization for pointed simplicial presheaves sectionwise.  Then, we check that $R$-localization behaves well with respect to local weak equivalences.  Nothing in this section should be surprising to experts; we collected the relevant facts here for lack of a better reference.

\subsection{Nilpotent (pre)sheaves of groups}
\label{ss:nilpotentandrlocalgroups}
Suppose throughout this section that $\mathbf{C}$ is a small category.  If $\mathbf{C}$ is equipped with a Grothendieck topology, then we formulate notions of nilpotence for sheaves of groups and actions of sheaves of groups (see Definitions~\ref{defn:nilpotentsheaf} and \ref{defn:localnilpotence}) and study the basic properties of such sheaves of groups.  We then study $R$-local sheaves of groups when $R \subset \Q$ is a subring and discuss the interplay of nilpotence and the property of being $R$-local; we close with a definition of $R$-nilpotent sheaf of groups (see Definition~\ref{defn:Rnilpotentsheafofgroups}).

\subsubsection*{Nilpotent sheaves of groups and actions}
If $\mathbf{H},\mathbf{G} \in \Grp(\mathbf{C})$, then an action of $\mathbf{G}$ on $\mathbf{H}$ is a homomorphism of presheaves of groups $\mathbf{G} \to \Aut(\mathbf{H})$.  If $(\mathbf{C},\tau)$ is a site (i.e., $\tau$ is a Grothendieck topology on $\mathbf{C}$), then given a sheaf of groups $\mathbf{H}$, $\Aut(\mathbf{H})$ is automatically a sheaf of groups.  Therefore, an action of a sheaf of groups $\mathbf{G}$ on a sheaf of groups $\mathbf{H}$ is just an action of the underlying presheaves.  We now define a notion of nilpotence for actions.

\begin{defn}
\label{defn:nilpotentsheaf}
Assume $(\mathbf{C},\tau)$ is a site.  Suppose $\mathbf{G},\mathbf{H} \in \Grp(\mathbf{C})$.
\begin{enumerate}[noitemsep,topsep=1pt]
\item Given an action of $\mathbf{G}$ on $\mathbf{H}$, a {\em $\mathbf{G}$-central series for the given action on $\mathbf{H}$} is a finite decreasing filtration
    \[
    \mathbf{H} = \mathbf{H}_0 \supset \cdots \supset \mathbf{H}_n = 1
    \]
    of $\mathbf{H}$ by $\mathbf{G}$-stable normal subgroup presheaves such that
    \begin{enumerate}[noitemsep,topsep=1pt]
    \item the successive subquotients $\mathbf{H}_i/\mathbf{H}_{i+1}$ are presheaves of abelian groups, and
    \item the induced action of $\mathbf{G}$ on each subquotient is trivial.
    \end{enumerate}
\item An action of $\mathbf{G}$ on $\mathbf{H}$ is called {\em nilpotent} if there exists a $\mathbf{G}$-central series for the action.
\item If $\mathbf{G},\mathbf{H} \in \Grp_{\tau}(\mathbf{C})$, then an action of $\mathbf{G}$ on $\mathbf{H}$ is called {\em nilpotent} if the underlying action as presheaves is nilpotent.
\item A presheaf of groups $\mathbf{G}$ will be called nilpotent if the conjugation action of $\mathbf{G}$ on itself is a nilpotent action.
\item If $\mathbf{G} \in \Grp_{\tau}(\mathbf{C})$ then $\mathbf{G}$ will be called nilpotent if it is nilpotent as a presheaf.
\item Given a nilpotent action of $\mathbf{G}$ on $\mathbf{H}$ and a $\mathbf{G}$-central series for the action, the {\em length} of the series is the smallest integer $c$ such that $\mathbf{H}_{c'} = 1$ for all $c' \geq c$.
\item The {\em nilpotence class} of a nilpotent action is the minimum length of a $\mathbf{G}$-central series for the action.
\end{enumerate}
\end{defn}

\begin{rem}
\label{rem:stalksandnilpotence}
Suppose $(\mathbf{C},\tau)$ is a site with enough points.  If a sheaf of groups $\mathbf{G}$ acts on a sheaf of groups $\mathbf{H}$, then the action is trivial if and only if the action is stalkwise trivial since one may check that $\mathbf{G} \to \mathbf{Aut}(\mathbf{H})$ is the trivial homomorphism stalkwise.  It follows that one can check whether a given filtration of $\mathbf{H}$ by $\mathbf{G}$-stable subsheaves is a $\mathbf{G}$-central series stalkwise.
\end{rem}

In the sheaf-theoretic context, we also make the following definition.

\begin{defn}
\label{defn:localnilpotence}
Suppose $(\mathbf{C},\tau)$ is a site with enough points.  An action of a sheaf of groups $\mathbf{G}$ on a sheaf of groups $\mathbf{H}$ is {\em locally nilpotent} if there exists a decreasing filtration
\[
\mathbf{H} = \mathbf{H}_0 \supset \mathbf{H}_1 \supset \cdots
\]
by $\mathbf{G}$-stable normal subgroup sheaves such that for every point $s$ of $(\mathbf{C},\tau)$ the induced action of the group $s^*\mathbf{G}$ on the group $s^*\mathbf{H}$ is nilpotent (in particular, the induced filtration on each stalk is finite).  A filtration as above will be called a {\em locally finite $\mathbf{G}$-central series for the action of $\mathbf{G}$ on $\mathbf{H}$}.
\end{defn}

One interesting difference between the classical story for nilpotent actions and the sheaf-theoretic situation we consider is that local nilpotence need not imply nilpotence in general, essentially because the nilpotence class of the action can vary with the stalks.

\begin{ex}
\label{ex:pointwisenilpotent}
Consider the unramified Grothendieck--Witt sheaf $\mathbf{GW}$; this is a Zariski sheaf of rings on $\Sm_k$ (in fact it is already a Nisnevich sheaf).  This sheaf has a natural action by the sheaf $\gm{}$ of units.  Since we will only be interested in phenomena at the level of stalks, we mention that separable, finitely generated extensions of the base-field $k$ are examples of stalks in the Nisnevich topology on $\Sm_k$.

Fix a separable, finitely generated extension $L/k$.  We now describe the action of $\gm{}$ on $\mathbf{GW}$ at the level of $L$-points.  The abelian group of sections $\mathbf{GW}(L)$ coincides with $GW(L)$, the usual Grothendieck--Witt group of (stable) isomorphism classes of symmetric bilinear forms over $L$.  There is an evident morphism $\gm{}(L) \to \mathbf{GW}(L)$ that sends a unit $u$ to the (invertible) $1$-dimensional form $\langle u \rangle$.  The action of $\gm{}$ is  given by tensoring a given form with the $1$-dimensional form $\langle u \rangle$.

Consider the filtration of $GW(L)$ by powers of the fundamental ideal $I^n(L)$.  Note that $I(L)$ is additively generated by $1$-fold Pfister forms:
\[
\langle \langle u \rangle \rangle := 1 - \langle u \rangle, u \in L^{\times}
\]
by appeal to \cite[Corollary 4.9]{EKM}. Since $I(L)$ is generated by $1$-fold Pfister forms, it follows that $I^n(L)$ is additively generated by ($n$-fold) Pfister forms.  A straightforward calculation shows that multiplication by $\langle u \rangle$ is trivial on the successive subquotients $I^n(L)/I^{n+1}(L)$ \cite[Lemma E.1.3]{FaselChowWitt}.  It follows from these facts that the filtration $I^n(L)$ is the $\gm{}(L)$-lower central series for this action (see, e.g., \cite[\S I.4]{HMR} for the definition of the latter notion).

Assume furthermore that $k$ is not formally real and that $k$ has finite \'etale $2$-cohomological dimension.  The Arason--Pfister Hauptsatz shows that $I^n(L)$ is then trivial for $n$ large enough (depending on the \'etale $2$-cohomological dimension of $L$; see \cite[Theorem 6.18]{EKM} and \cite[Proposition 5.1]{AsokFaselThreefolds} for details).  Thus, as $L$ varies through the generic points of smooth $k$-schemes, we see that the length of the $\gm{}(L)$-lower central series varies.  Later, we will see that variants of this example arise naturally in geometry (see Example~\ref{ex:pointwiseaonenilpotent}, Proposition~\ref{prop:p2npointwiseaonenilpotent} and Theorem~\ref{thm:bglevenlocallyaonenilpotent} for further discussion).
\end{ex}

\subsubsection*{Permanence properties for (locally) nilpotent actions}
Suppose $(\mathbf{C},\tau)$ is a site.  If $\varphi: \mathbf{G}' \to \mathbf{G}$ is a homomorphism of $\tau$-sheaves of groups, and $\mathbf{H}$ is a sheaf of groups with a $\mathbf{G}$-action, then $\mathbf{H}$ inherits a $\mathbf{G}'$-action via precomposition with $\varphi$.

\begin{lem}
\label{lem:functorialityofnilpotence}
If $\varphi: \mathbf{G}' \to \mathbf{G}$ is a homomorphism of sheaves of groups, and if $\mathbf{H}$ carries a (locally) nilpotent $\mathbf{G}$-action, then the $\mathbf{G}'$-action induced by $\varphi$ is (locally) nilpotent as well.
\end{lem}

\begin{proof}
A (finite) $\mathbf{G}$-central series for $\mathbf{H}$ gives a (finite) $\mathbf{G}'$-central series as well by restriction.
\end{proof}

The next result is a sheaf-theoretic analog of \cite[II.4.2]{BK}.

\begin{lem}
\label{lem:shortexactsequencesofnilpotentactions}
Suppose $\mathbf{G}$ is a sheaf of groups, $\mathbf{H},\mathbf{H}'$ and $\mathbf{H}''$ are sheaves of groups equipped with an action of $\mathbf{G}$ and we are given a short exact sequence of the form
\[
1 \longrightarrow \mathbf{H}' \longrightarrow \mathbf{H} \longrightarrow \mathbf{H}'' \longrightarrow 1,
\]
where the homomorphisms in the exact sequence are $\mathbf{G}$-equivariant.  The action of $\mathbf{G}$ on $\mathbf{H}$ is (locally) nilpotent if and only if the actions of $\mathbf{G}$ on $\mathbf{H}'$ and $\mathbf{H}''$ are so.
\end{lem}

\begin{proof}
Choose a (locally finite) $\mathbf{G}$-central series $\{\mathbf{H}_i\}$ for $\mathbf{H}$.  The restriction $\mathbf{H}_i \cap \mathbf{H}'$ yields a (locally finite) $\mathbf{G}$-central series for $\mathbf{H}'$.  Likewise, the quotient sheaves $\mathbf{H}_i/(\mathbf{H}' \cap \mathbf{H}_i)$ yield a (locally finite) $\mathbf{G}$-central series for $\mathbf{H}''$.

Conversely, take a (locally finite) $\mathbf{G}$-central series $\{\mathbf{H}''_i\}$ for $\mathbf{H}''$.  The fiber product sheaves $\mathbf{H}_i := \mathbf{H} \times_{\mathbf{H}''} \mathbf{H}''_i$ for all exist and are automatically $\mathbf{G}$-stable as all the morphisms defining the fiber product are $\mathbf{G}$-equivariant.  Moreover, all of these sheaves contain $\mathbf{H}'$ by construction. If we take a (locally finite) $\mathbf{G}$-central series of $\mathbf{H}'$, say $\{\mathbf{H}'_j\}$ then we may build a (locally finite) $\mathbf{G}$-central series of $\mathbf{H}$ by reindexing.
\end{proof}

\subsection{Nilpotence in local homotopy theory}
\label{ss:nilpotenceinlocalhomotopy}
The main goal of this section is to study nilpotent spaces in the local homotopy theory of simplicial presheaves.  We begin by reviewing some terminology from local homotopy theory following \cite{Jardine}.  We then define a notion of nilpotent morphism of simplicial presheaves using the definitions of Section \ref{ss:nilpotentandrlocalgroups}.  We then deduce analogs of a number of ``permanence" properties for nilpotent morphisms of spaces that will be useful later.

\subsubsection*{Review of local homotopy theory}
Suppose $\mathbf{C}$ is a small category and write $\sPre(\mathbf{C})$ for the category of simplicial presheaves on $\mathbf{C}$.  We view $\sPre({\mathbf C})$ as a model category with the injective model structure: the weak equivalences are objectwise weak equivalences, the cofibrations are the monomorphisms, and the fibrations are determined by the right lifting property.  We write $\Map(\mathscr{F},\mathscr{G})$ for the derived mapping space between two objects of $\sPre(\mathbf{C})$.

The model category $\sPre(\mathbf{C})$ is known to be simplicial, proper and combinatorial \cite[Proposition A.2.8.2]{Lurie} and we may appeal to the machinery of Bousfield localization (see \cite{Hirschhorn} for a detailed treatment of localization).  Recall that if $S$ is a set of morphisms in $\sPre(\mathbf{C})$, then a simplicial presheaf $\mathscr{F}$ is $S$-local if for every $f: \mathscr{G} \to \mathscr{H}$ in $S$, the induced map $f^*: \Map(\mathscr{H},\mathscr{F}) \to \Map(\mathscr{G},\mathscr{F})$ is a weak equivalence.  A morphism $f: \mathscr{G} \to \mathscr{H}$ is an $S$-local equivalence if, for every $S$-local $\mathscr{F}$, the map $f^*$ is a weak equivalence.

The $S$-local equivalences are the weak equivalences for a new model structure $\sPre(\mathbf{C})[S^{-1}]$ called the $S$-local model structure: the cofibrations in $\sPre(\mathbf{C})[S^{-1}]$ are still monomorphisms, and the fibrant objects are the fibrant objects in $\sPre(\mathbf{C})$ that are $S$-local.  This new model structure is again left proper, simplicial and combinatorial (\cite[Proposition A.3.7.3]{Lurie}).  The identity functors form a Quillen adjunction
\[
\xymatrix{
\sPre(\mathbf{C}) \ar@<.3ex>[r] & \ar@<.3ex>[l] \sPre(\mathbf{C})[S^{-1}],
}
\]
and the associated Quillen derived functor $\mathrm{Ho}(\sPre(\mathbf{C})[S^{-1}]) \to \mathrm{Ho}(\sPre(\mathbf{C}))$ is fully faithful with essential image the subcategory of $S$-local objects.

If $(\mathbf{C},\tau)$ is a small Grothendieck site, then one introduces the notion of {\em $\tau$-local weak equivalence} \cite[p. 64]{Jardine}.  The Bousfield localization of $\sPre(\mathbf{C})$ with respect to the class of $\tau$-local weak equivalences yields the (injective) $\tau$-local model structure on $\sPre(\mathbf{C})$, which is again a combinatorial, proper, simplicial model category.  We will write $\mathrm{R}_{\tau}$ for a fibrant replacement functor in this category.  If $(\mathbf{C},\tau)$ has enough points, then the functor $\mathrm{R}_{\tau}$ may be assumed to commute with formation of finite products.  Indeed, this follows from \cite[\S 2 Theorem 1.66]{MV}, which shows that one may use a ``Godement resolution" functor for $R_{\tau}$.  Note that Morel--Voevodsky state this assertion under an auxiliary ``finite type" hypothesis \cite[\S 2 Definition 1.31]{MV}, which ensures that Postnikov towers converge.  However, in modern terminology this hypothesis is equivalent to a hypercompleteness hypothesis on the $\infty$-category attached to the injective model structure on $\sPre(\mathbf{C})$ \cite[p. 666-669]{Lurie}.  With that in mind, the hypothesis that $(\mathbf{C},\tau)$ has enough points is sufficient to guarantee hypercompleteness \cite[Remark 6.5.4.7]{Lurie}.  In any case, we write $\operatorname{Ho}_{\tau}(\mathbf{C})$ for the associated homotopy category.

If $\mathscr{X} \in \sPre(\mathbf{C})$ is a simplicial presheaf, then a base-point for $\mathscr{X}$ is a morphism $x$ from the final object $\ast$ to $\mathscr{X}$.  The category of pointed simplicial presheaves $\sPre(\mathbf{C})_{\ast}$ can be equipped with a $\tau$-local model structure as well: a map of pointed simplicial presheaves is a cofibration, weak equivalence or fibration in the $\tau$-local model structure if it is so after forgetting the base-point.  We use the same notation as above for the fibrant replacement functor in this context.  We write $\operatorname{Ho}_{\tau,\ast}(\mathbf{C})$ for the associated homotopy category.

\subsubsection*{Nilpotent simplicial presheaves}
Suppose $(\mathbf{C},\tau)$ is a small Grothendieck site (which we will assume has enough points when we speak of ``local" notions below).  Given an object $\mathscr{X} \in \sPre(\mathbf{C})$, we write $\bpi_0(\mathscr{X})$ for the $\tau$-sheaf associated with the presheaf $U \mapsto \pi_0(\mathscr{X}(U))$.  We will say that $\mathscr{X}$ is $\tau$-connected (or just simplicially connected if $\tau$ is clear from context) if the structure map $\mathscr{X} \to \ast$ induces an isomorphism on $\bpi_0$.

If $(\mathscr{X},x)$ is a pointed simplicial presheaf, we define $\tau$-homotopy sheaves $\bpi_i(\mathscr{X},x)$ as the $\tau$-sheaves associated with the presheaves
\[
U \mapsto \pi_i(\mathscr{X}(U),x).
\]
As usual, these are sheaves of groups for $i = 1$ and sheaves of abelian groups for $i \geq 2$.  For any integer $n > 0$, say that a pointed space $(\mathscr{X},x)$ is simplicially $n$-connected if $\mathscr{X}$ is simplicially connected, and $\bpi_i(\mathscr{X},x) = 0$ for $1 \leq i \leq n$, .

The standard action of the group $\pi_1(\mathscr{X}(U),x)$ on $\pi_i(\mathscr{X}(U),x)$ is functorial in $U$ and sheafification yields an action of the sheaf $\bpi_1(\mathscr{X},x)$ on $\bpi_i(\mathscr{X},x)$.  (Frequently, to unburden the notation, we will suppress base-points in homotopy sheaves of pointed spaces.)

If $\mathscr{F} \to \mathscr{E} \to \mathscr{B}$ is a fiber sequence of pointed simplicial presheaves (henceforth we will refer to such sequences as simplicial fiber sequences), then for any object $U \in \sPre(\mathbf{C})_{\ast}$ applying $\hom_{\operatorname{Ho}_{\tau,\ast}(\mathbf{C})}(U,-)$ to the above fiber sequence yields a long exact sequence \cite[Proposition 6.5.3]{Hovey}.  Sheafifying this exact sequence then gives a long exact sequence in homotopy sheaves associated with any simplicial fiber sequence.

If $f: \mathscr{E} \to \mathscr{B}$ is any morphism of pointed connected spaces, then we may build the simplicial homotopy fiber of $f$ exactly as one does classically.  In more detail, consider the map $\mathscr{B} \to \mathrm{R}_{\tau}\mathscr{B}$ and functorially factor the composite $\mathscr{E} \to \mathrm{R}_{\tau}\mathscr{B}$ as an acyclic cofibration $\mathscr{E} \to \mathscr{E}'$ followed by a fibration $\mathscr{E}' \to \mathrm{R}_{\tau}\mathscr{B}$.  If $\mathscr{\ast}$ is the base-point of $\mathrm{R}_{\tau}\mathscr{B}$, then we define $\operatorname{hofib}(f)$ to be the ordinary fiber of $\mathscr{E}' \to \mathrm{R}_{\tau}\mathscr{B}$ over $\ast$.  There is a simplicial fiber sequence of the form
\[
\operatorname{hofib}(f) \longrightarrow \mathscr{E} \longrightarrow \mathscr{B}
\]
by construction.  If $U \in \mathbf{C}$, then sheafifying the objectwise action of $\pi_1(\mathscr{E}(U))$ on $\pi_1(\operatorname{hofib}(f)(U))$ induces an action of the sheaf of groups $\bpi_1(\mathscr{E})$ on the sheaf $\bpi_i(\operatorname{hofib}(f))$.  Granted these facts, one makes the following definition.

\begin{defn}
\label{defn:nilpotentspace}
Suppose $f: \mathscr{E} \to \mathscr{B}$ is a morphism in $\sPre(\mathbf{C})_{\ast}$, and write $\mathscr{F}$ for the homotopy fiber of $f$.
\begin{enumerate}[noitemsep,topsep=1pt]
\item The morphism $f$ is {\em (locally) nilpotent} if $\mathscr{F}$ is simplicially connected and the action of $\bpi_1(\mathscr{E})$ on $\bpi_i(\mathscr{F})$ is (locally) nilpotent for every $i \geq 0$ (in the sense of \textup{Definition~\ref{defn:nilpotentsheaf}} or \textup{Definition~\ref{defn:localnilpotence}}).
\item A pointed space $(\mathscr{X},x)$ is {\em (locally) nilpotent} if the structure morphism is (locally) nilpotent.
\end{enumerate}
\end{defn}

\subsubsection*{Permanence properties of nilpotence}
\begin{prop}
\label{prop:2outof3fornilpotentmorphisms}
Suppose
\[
\xymatrix{
\mathscr{E}_2 \ar[r]^{q} & \mathscr{E}_1 \ar[r]^{p} & \mathscr{E}_0
}
\]
is a composable pair of morphisms in $\sPre(\mathbf{C})_{\ast}$.  Assume that the simplicial homotopy fibers of $p$, $q$ and $pq$ are all simplicially connected.  If any two elements of $\{p,q,pq\}$ are nilpotent morphisms, then so is the third.
\end{prop}

\begin{proof}
Write $\mathscr{F}_p$ for the simplicial homotopy fiber of $p$, $\mathscr{F}_q$ for the simplicial homotopy fiber of $q$ and $\mathscr{F}_{pq}$ for the simplicial homotopy fiber of the composite.  In that case, by the ``octahedral axiom" in the pointed homotopy category (e.g., \cite[Proposition 6.3.6]{Hovey}) there is an associated simplicial fiber sequence of the form
\[
\mathscr{F}_q \longrightarrow \mathscr{F}_{pq} \longrightarrow \mathscr{F}_p,
\]
and we may consider the associated long exact sequence of homotopy sheaves, which takes the form:
\[
\cdots \longrightarrow \bpi_{n+1}(\mathscr{F}_p) \longrightarrow \bpi_n(\mathscr{F}_q) \longrightarrow \bpi_n(\mathscr{F}_{pq}) \longrightarrow \bpi_n(\mathscr{F}_q) \longrightarrow \cdots.
\]
By assumption, $\bpi_1(\mathscr{E}_2)$ acts on the homotopy sheaves of $\mathscr{F}_q$ and $\mathscr{F}_{pq}$ while $\bpi_1(\mathscr{E}_1)$ acts on the homotopy sheaves of $\mathscr{F}_p$.  Composing with the homomorphism $\bpi_1(\mathscr{E}_2) \to \bpi_1(\mathscr{E}_1)$ defines an action of $\bpi_1(\mathscr{E}_2)$ on the higher homotopy sheaves of $\mathscr{F}_p$ as well.  The long exact sequence of homotopy sheaves is then $\bpi_1(\mathscr{E}_2)$-equivariant with respect to these actions.

We may break the long exact sequence in homotopy sheaves above $\bpi_1(\mathscr{E}_2)$-equivariantly into short exact sequences.  The result then follows from repeated application of Lemma \ref{lem:shortexactsequencesofnilpotentactions}.
\end{proof}

Next, we observe that, under suitable hypotheses, (locally) nilpotent morphisms are stable under homotopy base change.

\begin{prop}
\label{prop:nilpotencebasechange}
Consider a homotopy Cartesian diagram
\[
\xymatrix{
\mathscr{E}' \ar[r]^-{g'}\ar[d]^-{f'} & \mathscr{E} \ar[d]^-{f} \\
\mathscr{B}' \ar[r]^-{g} & \mathscr{B}.
}
\]
in $\sPre(\mathbf{C})_{\ast}$.  If all spaces in the diagram are simplicially connected, and $f$ is (locally) nilpotent, then so is $f'$.
\end{prop}

%Suppose $f: \mathscr{E} \to \mathscr{B}$ and $g: \mathscr{B}' \to \mathscr{B}$ are morphisms of pointed connected spaces.  Set $\mathscr{E}' := \mathscr{E} \times^h_{\mathscr{B}} \mathscr{B}'$ and write $f': \mathscr{E}' \to \mathscr{B}'$ for the homotopy base-change of $f$.  If $\mathscr{E}'$ is connected, and $f$ is (locally) nilpotent, then so is $f'$.

\begin{proof}
This follows immediately from Lemma~\ref{lem:functorialityofnilpotence}: the action of $\bpi_1(\mathscr{E}')$ on the homotopy fiber of $f'$, which coincides with the homotopy fiber of $f$ by assumption, is induced by the homomorphism $\bpi_1(\mathscr{E}') \to \bpi_1(\mathscr{E})$.
\end{proof}

%Under the hypotheses, the map of vertical homotopy fibers is an weak equivalence; write $\mathscr{F}$ for this (common) vertical homotopy fiber.  The morphism $g'$ yields a homomorphism $\bpi_1(\mathscr{E}') \to \bpi_1(\mathscr{E})$.  By assumption, $\mathscr{F}$ is connected and the action of $\bpi_1(\mathscr{E})$ on $\bpi_1(\mathscr{F})$ is (locally) nilpotent.  Therefore, the induced action of $\bpi_1(\mathscr{E}')$ is again (locally) nilpotent by appeal to Lemma \ref{lem:functorialityofnilpotence}.

\begin{prop}
\label{prop:enilpotentimpliesfnilpotent}
Suppose $g: \mathscr{E} \to \mathscr{B}$ is a morphism in $\sPre(\mathbf{C})_{\ast}$ with homotopy fiber $\mathscr{F}$, and assume $\mathscr{F},\mathscr{E}$ and $\mathscr{B}$ are all simplicially connected.  If $\mathscr{E}$ is (locally) nilpotent, then $\mathscr{F}$ is (locally) nilpotent.
\end{prop}
%Suppose $g: \mathscr{E}_2 \to \mathscr{E}_1$ and $f: \mathscr{E}_1 \to \mathscr{E}_0$ and $\varphi: \mathscr{E}'_1 \to \mathscr{E}_1$ are morphisms in $\sPre(\mathbf{C})_{\ast}$ where all spaces are simplicially connected.  Assume the homotopy fiber product $\mathscr{E}'_2 \times^h_{\mathscr{E}_1} \mathscr{E}_2$ is simplicially connected and write $g': \mathscr{E}'_2 \to \mathscr{E}'_1$ for the induced morphism.  If the composite $fg$ is nilpotent, then $g'$ is nilpotent.

\begin{proof}
Consider the homotopy commutative diagram:
\[
\xymatrix{
\mathscr{F} \ar[r]\ar[d] & \mathscr{E} \ar[d]^g \\
\ast \ar[r] & \mathscr{B}.
}
\]
By appeal to Proposition~\ref{prop:nilpotencebasechange} to check $\mathscr{F} \to \ast$ is nilpotent, it suffices to show that $g$ is (locally) nilpotent, i.e., that the $\bpi_1(\mathscr{E})$-action on $\bpi_i(\mathscr{F})$ is (locally) nilpotent.

To this end, observe that there is a $\bpi_1(\mathscr{E})$-equivariant long exact sequence of the form
\[
\bpi_{i+1}(\mathscr{B}) \longrightarrow \bpi_i(\mathscr{F}) \longrightarrow \bpi_i(\mathscr{E}) \longrightarrow \bpi_i(\mathscr{B}).
\]
While we do not know that the action of $\bpi_1(\mathscr{E})$ on $\bpi_{i+1}(\mathscr{B})$ is (locally) nilpotent, we do know that $\im \pi_{i+1}(\mathscr{B})$ in $\bpi_i(\mathscr{F})$ is central for $i \geq 1$ and carries a trivial action of $\bpi_1(\mathscr{E})$.
In other words, there are $\bpi_1(\mathscr{E})$-equivariant short exact sequences of the form:
\[
1 \longrightarrow \mathbf{M}_{i+1} \longrightarrow \bpi_i(\mathscr{F}) \longrightarrow \ker(\bpi_i(\mathscr{E}) \longrightarrow \bpi_i(\mathscr{B})) \longrightarrow 1,
\]
where $\mathbf{M}_{i+1}$ is the image of $\bpi_{i+1}(\mathscr{B})$ in $\bpi_i(\mathscr{F})$.  In fact, the action of $\bpi_1(\mathscr{E})$ on $\mathbf{M}_{i+1}$ is trivial (the classical proof of this result, e.g., \cite[Proof of Theorem 2.2]{HMR} discusses this further).  As a $\bpi_1(\mathscr{E})$-stable subsheaf of $\bpi_i(\mathscr{E})$, the kernel in the above statement carries a (locally) nilpotent $\bpi_1(\mathscr{E})$-action.  The nilpotence of the $\bpi_1(\mathscr{E})$-action on $\bpi_i(\mathscr{F})$ then follows by appeal to Lemma~\ref{lem:shortexactsequencesofnilpotentactions}.
\end{proof}

\subsection{$R$-localization of simplicial presheaves}
\label{ss:rlocalizationofsimplicialpresheaves}
We now discuss $R$-localization of simplicial presheaves.  Suppose $(\mathbf{C},\tau)$ is a small Grothendieck site and consider $\sPre(\mathbf{C})$ with the $\tau$-local model structure described above.  We use the following notation.

\begin{notation}
\label{notation:fractions}
Suppose $R$ is a subring of $\Q$.  The elements of $R$ are fractions with denominator divisible by elements of some (possibly empty) set $P$ of prime numbers, i.e., $R = \Z[P^{-1}]$.  Henceforth, given $R \subset \Q$, we will always write $P$ for this associated set of primes.
\end{notation}

After reviewing some aspects of $R$-locality for sheaves of groups, we then introduce a notion of $R$-local weak equivalence.  We choose definitions that are well-adapted to stalkwise analysis, following \cite{WickelgrenWilliams}.  We then analyze the interaction between nilpotence properties and $R$-localization for simplicial presheaves

\subsubsection*{$R$-localization for sheaves of groups}
If $\mathbf{G}$ is any presheaf of groups on $\mathbf{C}$, and $n$ is an integer, then we define the $n$-th power map $\mathbf{G} \to \mathbf{G}$ to be the sectionwise $n$-th power map of groups.  Of course, if $\mathbf{G}$ is not a presheaf of abelian groups, this needs not be a group homomorphism.

Assume $R \subset \Q$ is a subring, notation following \ref{notation:fractions}.  If $\mathbf{G} \in \Grp(\mathbf{C})$, we will say that $\mathbf{G}$ is $R$-local (or $P$-local) if the $p$-th power map is an isomorphism of presheaves for every $p \in P$.  Similarly, we will say that a sheaf of groups $\mathbf{G} \in \Grp_{\tau}(\mathbf{C})$ is $R$-local if the $p$-th power map is an isomorphism of sheaves for every $p \in P$.  The following result is immediate from the definitions.

\begin{lem}
Suppose $(\mathbf{C},\tau)$ is a site with enough points. A sheaf of groups $\mathbf{G} \in \Grp_{\tau}(\mathbf{C})$ is $R$-local if and only if for every point $s$ of $(\mathbf{C},\tau)$, the group $s^*\mathbf{G}$ is $R$-local. \qed
\end{lem}

One knows that an abelian group is $R$-local if and only if it is an $R$-module.  It follows immediately from the definitions that a (pre)sheaf of abelian groups is $R$-local if and only if it is a (pre)sheaf of $R$-modules.  More generally, a nilpotent group is $R$-local if and only if it has a finite central series $G = G_0 \supset \cdots \supset G_n = e$ such that each subquotient $G_j/G_{j+1}$ admits a (necessarily unique) $R$-module structure \cite[V.2.7]{BK}.  The next result follows from this fact by passing to stalks.

\begin{lem}
\label{lem:Rnilpotenceequivconditions}
Suppose $(\mathbf{C},\tau)$ is a site with enough points. The following conditions on a (locally) nilpotent sheaf of groups $\mathbf{G} \in \Grp_{\tau}(\mathbf{C})$ are equivalent.
\begin{enumerate}[noitemsep,topsep=1pt]
\item The sheaf $\mathbf{G}$ is $R$-local
\item The sheaf $\mathbf{G}$ has a (locally) finite central series with successive subquotients that are sheaves of $R$-modules.\qed
\end{enumerate}
\end{lem}

\begin{defn}
\label{defn:Rnilpotentsheafofgroups}
Suppose $(\mathbf{C},\tau)$ is a site with enough points.  A sheaf of groups $\mathbf{G} \in \Grp_{\tau}(\mathbf{C})$ will be called (locally) $R$-nilpotent if it satisfies either of the equivalent conditions of \textup{Lemma~\ref{lem:Rnilpotenceequivconditions}}.
\end{defn}

\subsubsection*{$R$-local weak equivalences}
We continue with notation as in \ref{notation:fractions}.  Set $S^1_{rig} := B\Z$; this is a Kan complex (in fact a simplicial abelian group) weakly equivalent to  $\Delta^1/\partial \Delta^1$.  The multiplication by $n$ map induces a self-map $B\Z \to B\Z$ that we will call $\rho^1_n$.  For $r\geq 2$, let $S^{r}_{rig} = S^1_{rig} \sma (\Delta^{r-1}/\partial\Delta^{r})_+$.  Define $\rho^r_n: S^r_{rig} \to S^r_{rig}$ for $r \geq 2$ by the formula $\rho^r_n = \rho^1_n \sma id$.

For $(\mathbf{C},\tau)$ any small Grothendieck site, we set
\[
T_R := \{ \rho_n^r \times id_U | r \geq 1, n \in P, U \in \mathbf{C} \},
\]
where $\rho^n_r$ is now viewed as a morphism of constant simplicial presheaves.  The left Bousfield localization of the injective $\tau$-local model structure on $\sPre(\mathbf{C})$ with respect to the set of morphisms $T_R$ will be called the {\em $R$-local model structure} on $\sPre(\mathbf{C})$; since the site $(\mathbf{C},\tau)$ will be fixed, we hope that suppressing $\tau$ from the notation causes no confusion.  We write $\LR$ for the fibrant replacement functor on $\sPre(\mathbf{C})$ for the $R$-local model structure so $\LR \mathscr{X}$ is both $R$-local and $\tau$-fibrant.

\begin{rem}
\label{rem:inequivalentversions}
In the situation where $\mathbf{C}$ is the $1$-point site, this definition yields the form of $R$-localization on the model category of simplicial sets studied in \cite{CasacubertaPeschke}.  The {\em a priori} strange looking form of $T_R$ is necessary to insure that localization is compatible with the action of the fundamental group on higher homotopy groups, already in this classical situation.  Various different notions of $R$-localization are compared in \cite[Section 8]{CasacubertaPeschke}: these notions may disagree on non-nilpotent spaces.  More generally, suppose $(\mathbf{C},\tau)$ is a small Grothendieck site and $s$ is a point of this site.  In Proposition~\ref{prop:LRproperties}(1) below, we investigate the commutation of $\LR$ with $s^*$; here, we abuse notation and implicitly write $\LR$ both for localization of simplicial presheaves and of simplicial sets.
\end{rem}

The next result collects formal properties of $\LR$, which will be essential in the sequel; this result restates \cite[Lemma 3.3 and Propositions 3.4,3.8]{WickelgrenWilliams}, but our proof slightly corrects that of \cite[Lemma 3.3]{WickelgrenWilliams}.

\begin{prop}
\label{prop:LRproperties}
Suppose $(\mathbf{C},\tau)$ is a small Grothendieck site with enough points, and $\mathscr{X} \in \sPre(\mathbf{C})$.
\begin{enumerate}[noitemsep,topsep=1pt]
\item The $R$-localization functor $\LR$ commutes with taking stalks, i.e., if $s$ is a point of $(\mathbf{C},\tau)$, then $s^* \LR \mathscr{X} \cong \LR s^* \mathscr{X}$.
\item If $\mathscr{X}$ is fibrant, then $\mathscr{X}$ is $R$-local if and only if its stalks are all $R$-local for $s$ ranging through a conservative family of points for $(\mathbf{C},\tau)$.
\item The $R$-localization functor commutes with formation of finite products.
\end{enumerate}
\end{prop}

\begin{proof}
For Point (1), we first show that $s^* \LR \mathscr{X}$ is $R$-local.  Since $\LR \mathscr{X}$ is fibrant, so is $s^* \LR \mathscr{X}$.  Therefore, it suffices to show that if $\rho^k_n$ is an element of $T_R$, then the induced map
\[
\operatorname{Map}(S^k_{\tau},s^*\LR \mathscr{X}) \longrightarrow \operatorname{Map}(S^k_{\tau},s^*\LR \mathscr{X})
\]
is a weak equivalence of simplicial sets.  By definition, any point of the site $(\mathbf{C},\tau)$ may be realized as a filtered colimit over neighborhoods.  Explicitly, there is an isomorphism of simplicial sets of the form:
\[
\operatorname{Map}(S^k_{\tau},s^*\LR \mathscr{X}) \cong \operatorname{Map}(S^k_{\tau},\operatorname{colim}_{U \in Neib(s)}\LR\mathscr{X}(U)).
\]
The simplicial set $S^k_{\tau}$ is not a compact simplicial set (since $B\Z$ is not finite) so we cannot simply commute the colimit past the mapping space.  Instead, $S^k_{\tau}$ is ``homotopically small" in the sense that it is weakly equivalent to a compact object (this notion is related to, but different from, that of \cite[\S 2.2]{DwyerKan}), which allows us to proceed as follows.

Set $S^1 := \Delta^1/\partial \Delta^1$ and consider the map $S^1 \to S^1_{rig}$; this map is an acyclic cofibration.  There is an induced cofibration $S^1 \sma (\Delta^k/\partial \Delta^k_+) \to S^{k+1}_{rig}$.  By the universal property of colimits, this map induces a commutative square of the form
\[
\xymatrix{
\operatorname{colim}_{U \in Neib(s)} \operatorname{Map}(S^k_{rig},\LR\mathscr{X}(U)) \ar[r]\ar[d] & \operatorname{colim}_{U \in Neib(s)}\operatorname{Map}(S^1 \sma (\Delta^k/\partial \Delta^k_+)),\LR\mathscr{X}(U) \ar[d] \\
\operatorname{Map}(S^k_{rig},\operatorname{colim}_{U \in Neib(s)}\LR\mathscr{X}(U)) \ar[r]& \operatorname{Map}(S^1 \sma (\Delta^k/\partial \Delta^k_+),\operatorname{colim}_{U \in Neib(s)}\LR\mathscr{X}(U))
}
\]
Since $s^*\LR \mathscr{X}$ and $\LR \mathscr{X}(U)$ are fibrant simplicial set, it follows that the maps $\operatorname{Map}(S^k_{\tau},s^*\LR \mathscr{X}) \to \operatorname{Map}(S^1 \sma (\Delta^k/\partial \Delta^k_+),s^*\LR \mathscr{X})$ and 
\[\operatorname{Map}(S^k_{rig},\LR\mathscr{X}(U)) \to \operatorname{Map}(S^1 \sma (\Delta^k/\partial \Delta^k_+),\LR\mathscr{X}(U))
\]
are acyclic fibrations.  In particular, the bottom horizontal map in the square is a weak equivalence.  Since filtered colimits of weak equivalences are again weak equivalences, it follows that the top horizontal map is a weak equivalence as well.  Since $S^1 \sma (\Delta^k/\partial \Delta^k_+)$ is compact, the right hand vertical map is an isomorphism of simplicial sets.  Altogether, we conclude that the left hand vertical map is a weak equivalence of simplicial sets.

There is an isomorphism of simplicial sets of the form
\[
\operatorname{colim}_{U \in Neib(s)}\operatorname{Map}(S^k_{rig},\LR \mathscr{X}(U)) \cong \operatorname{colim}_{U \in Neib(s)}\operatorname{Map}(S^k_{rig} \times U,\LR \mathscr{X})
\]
As $\LR \mathscr{X}$ is $R$-local by assumption, $\rho^k_n$ induces a weak equivalence on $\operatorname{Map}(S^k_{rig} \times U,\LR \mathscr{X})$.

The functor $s^*$ preserves acyclic cofibrations and therefore the map $s^* \mathscr{X} \to s^* \LR \mathscr{X}$ is an acyclic cofibration with target an object that is fibrant in the $R$-local model structure on simplicial sets.  It follows that $s^* \LR \mathscr{X}$ is equivalent in the $R$-local model structure to any other $R$-local fibrant replacement of $s^* \mathscr{X}$, e.g., $\LR s^* \mathscr{X}$.

For Point (2), observe that the proof of Point (1) implies that if $\mathscr{X}$ is $R$-local, then so is $s^*\mathscr{X}$.  Conversely, $\mathscr{X}$ is $R$-local if and only if $\mathscr{X}$ is fibrant and the map $\mathscr{X} \to \LR\mathscr{X}$ is a simplicial weak equivalence.  However, the map $\mathscr{X} \to \LR\mathscr{X}$ is a simplicial weak equivalence if and only if it is after taking stalks.  By the proof of Point (1), this is the case if and only if $s^*\mathscr{X} \to \LR s^* \mathscr{X}$ is a weak equivalence for every point $s$.  Since $s^*\mathscr{X}$ is fibrant, this amounts to the assertion that $s^*\mathscr{X}$ is an $R$-local simplicial set.

For Point (3), it suffices to observe that if $\mathscr{X} \in \sPre(\mathbf{C})$, then $\mathscr{X} \times (-)$ preserves weak equivalences.  Indeed, this latter fact follows from Ken Brown's lemma \cite[Lemma 1.1.12]{Hovey} because all objects of $\sPre(\mathbf{C})$ are cofibrant and $\mathscr{X} \times (-)$ preserves acyclic cofibrations.  To conclude, one observes that $\LR \mathscr{X} \times \LR \mathscr{Y}$ is $R$-local and $\tau$-fibrant and $R$-locally weakly equivalent to $\mathscr{X} \times \mathscr{Y}$ and therefore also to $\LR (\mathscr{X} \times \mathscr{Y})$.
\end{proof}

\begin{lem}
\label{lem:homotopysheavesofRlocalspacesareRlocal}
Suppose $(\mathbf{C},\tau)$ is a small Grothendieck site with enough points, and $(\mathscr{X},x) \in \sPre(\mathbf{C})_*$ is a pointed $R$-local space.  For any integer $i \geq 0$, the homotopy sheaves $\bpi_i(\mathscr{X},x)$ are $R$-local sheaves of groups.
\end{lem}

\begin{proof}
We show that the homotopy presheaves are $R$-local, and the statement about homotopy sheaves follows by sheafifying.  Let $U$ be an object of $\mathbf{C}$.  Write $R = \Z[P^{-1}]$ for some set of primes $P$.  We will show that the $p$-th power map on $\pi_i(\LR\mathscr{X}(U))$ is a bijection for any $p \in P$.  To this end, observe that the $p$-th power map is the map induced by $\rho_p^i \times id_{U}$ on $\pi_0(\Map_{*}(S^i_{rig} \times U,\LR \mathscr{X}))$.  Since $\LR\mathscr{X}$ is $R$-local and $\rho^i_p \times id_{U}$ lies in $T_R$, this map is a bijection.
\end{proof}

\begin{rem}
\label{rem:conversetolemhomotopysheavesofRlocalspacesareRlocal}
Lemma~\ref{lem:homotopysheavesofRlocalspacesareRlocal} admits a converse in the following sense.  If $(\mathscr{X},x) \in \sPre(\mathbf{C})_*$ is fibrant in the $\tau$-local model structure and its homotopy sheaves are $R$-local, then $(\mathscr{X},x)$ is $R$-local as well.
\end{rem}

\begin{lem}
\label{lem:equivalentcharacterizationsofbeinganRequivalence}
Assume $(\mathbf{C},\tau)$ is a site.  If $f: \mathscr{X} \to \mathscr{Y}$ is a morphism in $\sPre(\mathbf{C})$, then the following conditions are equivalent:
\begin{enumerate}[noitemsep,topsep=1pt]
\item the morphism $f$ is an $R$-local weak equivalence;
\item for every $R$-local space $\mathscr{W}$, the map
    \[
    f^*: \hom_{\operatorname{Ho}_{\tau}(\mathbf{C})}(\mathscr{Y},\mathscr{W}) \longrightarrow \hom_{\operatorname{Ho}_{\tau}(\mathbf{C})}(\mathscr{X},\mathscr{W})
    \]
    is a bijection.
\end{enumerate}
\end{lem}

\begin{proof}
The first statement implies the second by taking $\pi_0$ of the simplicial mapping space.  The second statement implies the first by observing that if $\mathscr{W}$ is $R$-local, then so is $\Omega^i \mathscr{W}$ for every $i \geq 0$.
\end{proof}

\begin{lem}
\label{lem:Rlocalizationnilpotentfiberlemma}
Suppose $p: \mathscr{E} \to \mathscr{B}$ is a (locally) nilpotent morphism of pointed simplicially connected spaces.  The morphism
\[
\LR p: \LR \mathscr{E} \longrightarrow \LR \mathscr{B}
\]
is also a (locally) nilpotent morphism, and the canonical map $\LR(\operatorname{hofib}(p)) \to \operatorname{hofib}(\LR p)$ is a $\tau$-local weak equivalence.
\end{lem}

\begin{proof}
Consider the morphism $\mathscr{B} \to \mathrm{R}_{\tau}\mathscr{B}$ and factor the composite $\mathscr{E} \to \mathrm{R}_\tau{\mathscr{B}}$ as an acyclic cofibration $\mathscr{E} \to \mathscr{E}'$ followed by a $\tau$-fibration $\mathscr{E}' \to \mathrm{R}_{\mathscr{B}}$ so that the homotopy fiber of $p$ is the actual fiber of $p': \mathscr{E}' \to \mathrm{R}_\tau{\mathscr{B}}$.  Both conditions in the statement may be checked stalkwise, and stalkwise the hypotheses imply that $p'$ is a nilpotent fibration.  

Next, recall that Bousfield--Kan construct the functor $R_{\infty}$ (\cite[I.4.2]{BK}) and show that $R_{\infty}$ models $R$-localization \cite[V.4.3]{BK}.  By \cite[Proposition 8.1]{CasacubertaPeschke} on the category of simplicial sets, $\LR$ models localization \`a la Bousfield--Kan.  Therefore, the results about preservation of nilpotent fibrations by $R_{\infty}$ also hold for $\LR$.  Granted these observations, the result follows directly from \cite[II.4.8]{BK}.
\end{proof}

\subsubsection*{Functorial $R$-localization for nilpotent sheaves of groups and spaces}
We now define a functorial $R$-localization for sheaves of groups; our goal is to show that this functor is well-behaved on the category of (locally) nilpotent sheaves of groups.

\begin{defn}
\label{defn:functorialRlocalization}
If $\mathbf{G}$ is a sheaf of groups, then we define
\[
\mathbf{G}_R := \bpi_1(\LR B\mathbf{G});
\]
the sheaf of groups $\mathbf{G}_R$ will be called the {\em $R$-localization} of $\mathbf{G}$.
\end{defn}

By definition, the assignment $\mathbf{G} \mapsto \mathbf{G}_R$ is functorial.  The properties of this functor upon restriction to the category of (locally) nilpotent sheaves of groups are recorded in the following result.

\begin{prop}
\label{prop:Rlocalizationofgroups}
Suppose $R \subset \Q$ is a ring.  The assignment $\mathbf{G} \mapsto \mathbf{G}_R$ enjoys the following properties:
\begin{enumerate}[noitemsep,topsep=1pt]
\item the functor $(-)_R$ is left adjoint to the forgetful functor from the category of (locally) $R$-local nilpotent sheaves of groups to the category of (locally) nilpotent sheaves of groups;
\item the functor $(-)_R$ preserves exact sequences of (locally) nilpotent sheaves of group; and
\item if $R' \subset \Q$ is another ring, then the natural transformation $(-)_{R \tensor_{\Z} R'} \to ((-)_R)_{R'}$ is an isomorphism of functors from the category of (locally) nilpotent sheaves of groups to the category of $R \tensor_{\Z} R'$-local (locally) nilpotent sheaves of groups.
\end{enumerate}
\end{prop}

\begin{proof}
For the first point, $(-)_R$ is functorial by construction and it follows from Lemma~\ref{lem:homotopysheavesofRlocalspacesareRlocal} that $\mathbf{G}_R$ is an $R$-local sheaf of groups.  If $\mathbf{G}$ is (locally) nilpotent, then we claim $\mathbf{G}_R$ is (locally) nilpotent and $R$-local.  Since $\LR$ commutes with taking stalks, we may check this stalkwise, in which case it follows immediately from \cite[V.2.2]{BK}.

To complete the proof of (1), it suffices to prove that $\mathbf{G} \to \mathbf{G}_R$ is initial among maps from $\mathbf{G}$ to (locally) nilpotent $R$-local sheaves of groups.  If $\mathbf{H}$ is a (locally) nilpotent $R$-local sheaf of groups equipped with a homomorphism $\mathbf{G} \to \mathbf{H}$, then there is an evident map $\mathrm{R}_{\tau}B\mathbf{G} \to \mathrm{R}_{\tau}B\mathbf{H}$.  The space $\mathrm{R}_{\tau}B\mathbf{H}$ is $R$-local by Remark~\ref{rem:conversetolemhomotopysheavesofRlocalspacesareRlocal}, and therefore, it follows that $\mathrm{R}_{\tau}B\mathbf{G} \to \mathrm{R}_{\tau}B\mathbf{H}$ factors through $\LR B\mathbf{G}$.  Applying $\bpi_1$ establishes Point (1).

For Point (2), suppose
\[
1 \longrightarrow \mathbf{G}' \longrightarrow \mathbf{G} \longrightarrow \mathbf{G}'' \longrightarrow 1
\]
is a short exact sequence of (locally) nilpotent sheaves of groups.  In that case, we get a simplicial fiber sequence of the form
\[
B\mathbf{G}' \longrightarrow B\mathbf{G} \longrightarrow B\mathbf{G}''
\]
Applying $\LR$ to this fiber sequence yields a simplicial fiber sequence by appeal to Lemma~\ref{lem:Rlocalizationnilpotentfiberlemma}.  Exactness then follows by appeal to \cite[V.2.4]{BK}, since exactness may be checked stalkwise.

If $\mathbf{A}$ is an abelian sheaf of groups, then there is an evident isomorphism of functors $(R \tensor_{\Z} R') \tensor \mathbf{A} \isomto R \tensor (R' \tensor \mathbf{A})$.  Following \cite[V.2.9]{BK}, point (3) is immediate from (2) and a straightforward induction argument.
\end{proof}

\begin{cor}
\label{cor:Rlocalizationofnilpotentisnilpotent}
Suppose $(\mathscr{X},x)$ is a pointed locally nilpotent space.
\begin{enumerate}[noitemsep,topsep=1pt]
\item The space $\LR\mathscr{X}$ is an $R$-local locally nilpotent space;
and
\item For every integer $i \geq 1$, the canonical map $\bpi_i(\mathscr{X})_R \to \bpi_i(\LR\mathscr{X})$ is an isomorphism..
\end{enumerate}
\end{cor}

\begin{proof}
The first point is a special case of Lemma~\ref{lem:Rlocalizationnilpotentfiberlemma}.  The second point follows immediately from \cite[V.3.1]{BK} because the relevant isomorphism may be checked stalkwise.
\end{proof}

\section{Nilpotence in $\aone$-homotopy theory}
\label{s:nilpotenceinaonehomotopytheory}
In this section, we formulate and study a notion of nilpotence in $\aone$-homotopy theory.  Section~\ref{ss:aonelocalgrouptheory} recalls various results from Morel's foundational work \cite{MField}, in particular the notions of strongly and strictly $\aone$-invariant sheaves of groups (see Definition~\ref{defn:stronglyaoneinvariant}).  We then develop the basic properties of $\aone$-local group theory necessary to formulate suitable notions of nilpotence.  The main technical complication that arises is due to failure of existence of cokernels (see Remark~\ref{rem:nocokernels} for a precise statement). Section~\ref{ss:aonenilpotentgroups} then studies nilpotence in the context of $\aone$-local group theory.

Section~\ref{ss:aonenilpotentspaces} introduces various flavors of nilpotence for spaces and morphisms in $\aone$-homotopy theory (see Definition~\ref{defn:aonenilpotentspace}); we encourage the reader to pay attention to the notion of {\em weakly $\aone$-nilpotent} spaces which will play a prominent role later.  Moreover, we establish a collection of basic properties for such spaces (e.g., a characterization in terms of existence of Moore--Postnikov factorizations, i.e., Theorem~\ref{thm:moorepostnikovcharacterization}).  Finally, Section~\ref{ss:examplesofaonenilpotentspaces} constructs a host of spaces of geometric interest that exemplify the variant definitions we make.

\subsection{$\aone$-local group theory}
\label{ss:aonelocalgrouptheory}
F. Morel showed that the $\aone$-homotopy sheaves of a pointed space in degrees greater than $1$ have the fundamental property that their cohomology presheaves are themselves $\aone$-invariant.  After reviewing the basic definitions in $\aone$-homotopy theory, we study what one might call $\aone$-local group theory, building on the foundational results of F. Morel \cite{MField}.  In particular, we analyze functorial $\aone$-localization of groups and various categorical properties of the resulting category.

\subsubsection*{Preliminaries on $\aone$-homotopy theory}
Fix a base field $k$, and write $\Sm_k$ for the category of schemes that are separated, smooth and have finite type over $\Spec k$.  We view $\Sm_k$ as a site by equipping it with the Nisnevich topology. We set $\Spc_k := \sPre(\Sm_k)$ and view this as a site equipped with the injective Nisnevich-local model structure.  Similarly, we write $\Spc_{k,*}$ for the category of pointed simplicial presheaves on $\Sm_k$.

The (pointed) $\aone$-homotopy category $\ho{k}$ can be obtained as a left Bousfield localization of the injective Nisnevich-local model structure on $\Spc_k$ (resp. $\Spc_{k,*}$) with respect to the morphisms $\mathscr{X} \times \aone \to \mathscr{X}$, $\mathscr{X} \in \Spc_k$.

We write $\Laone: \Spc_k \to \Spc_k$ for the $\aone$-localization functor. This functor comes equipped with a natural transformation $\theta: id \to \Laone$ such that for any space $\mathscr{X}$ the following properties hold:
\begin{enumerate}[noitemsep,topsep=1pt]
\item the space $\Laone \mathscr{X}$ is fibrant in the injective Nisnevich local model structure and the map $\mathscr{X} \to \Laone \mathscr{X}$ is an $\aone$-weak equivalence;
\item if $\mathscr{Y}$ is any simplicially fibrant and $\aone$-local space, and $f: \mathscr{X} \to \mathscr{Y}$, then $f$ factors as $\mathscr{X} \to \Laone \mathscr{X} \to \mathscr{Y}$;
\item the functor $\Laone$ preserves finite limits.
\end{enumerate}

\begin{rem}
The properties of the $\aone$-localization construction described here are all explicitly contained in \cite{MV}, except the last one.  That the $\aone$-localization functor can be assumed to preserve finite limits follows from a slight modification of the construction described before \cite[\S 2 Lemma 3.20]{MV}.  There, the $\aone$-localization functor is constructed by iterated application of a fibrant resolution functor for the injective Nisnevich local model structure and the singular construction.  The singular construction preserves limits \cite[\S 2 p. 87]{MV}.  Taking the Godement resolution \cite[p. 66-70]{MV} as fibrant replacement functor for the injective Nisnevich local model structure it follows by appeal to \cite[\S 2 Theorem 1.66]{MV} that the fibrant resolution functor can be assumed to preserve finite limits as well.  
\end{rem}

\begin{rem}
In \cite{MV} the motivic homotopy category $\ho{k}$ is obtained by Bousfield localizing simplicial Nisnevich {\em sheaves} on $\Sm_k$.  By contrast, the model we used builds $\ho{k}$ as a localization of simplicial {\em presheaves} on $\Sm_k$.  The two constructions yield Quillen equivalent models of $\ho{k}$ by \cite[Theorem B.4 and B.6]{JardineSymmetric}.
\end{rem}

It will be important to remember that the category of $\aone$-local objects is closed under formation of homotopy limits and filtered homotopy colimits.  We now recall some results about how $\aone$-local objects behave in fiber sequences.

\begin{lem}
\label{lem:homotopyfiberaonelocal}
If $\mathscr{F} \to \mathscr{E} \to \mathscr{B}$ is a simplicial fiber sequence where $\mathscr{E}$ and $\mathscr{B}$ are $\aone$-local spaces, then $\mathscr{F}$ is $\aone$-local as well. \qed
\end{lem}

\begin{lem}[{\cite[Lemma 2.2.10]{AWW}}]
\label{lem:aww2210}
Suppose $\mathscr{F} \to \mathscr{E} \to \mathscr{B}$ is a simplicial fiber sequence of pointed spaces.  If $\mathscr{B}$ and $\mathscr{F}$ are both $\aone$-local, and $\mathscr{B}$ is simplicially connected, then $\mathscr{E}$ is $\aone$-local as well. \qed
\end{lem}

\subsubsection*{Consequences of Morel's unstable connectivity theorem}
If $\mathscr{X} \in \Spc_k$, then we set $\bpi_0^{\aone}(\mathscr{X}) := \bpi_0(\Laone \mathscr{X})$.  Similarly, if $(\mathscr{X},x) \in \Spc_{k,*}$, then we write $\bpi_i^{\aone}(\mathscr{X},x) := \bpi_i(\Laone \mathscr{X},x)$; these sheaves are called the $\aone$-homotopy sheaves of $\mathscr{X}$.   Morel established a number of key structural results for $\aone$-homotopy sheaves.  To state these results, we first recall the following definition.

\begin{defn}
\label{defn:stronglyaoneinvariant}
Assume $k$ is a base field.
\begin{enumerate}[noitemsep,topsep=1pt]
\item A Nisnevich sheaf of groups $\mathbf{G}$ on $\Sm_k$ is called {\em strongly $\aone$-invariant} if $B\mathbf{G}$ is $\aone$-local.
\item A Nisnevich sheaf of abelian groups $\mathbf{A}$ on $\Sm_k$ is called strictly $\aone$-invariant if $K(\mathbf{A},n)$ is $\aone$-local for every $n \geq 0$.
\end{enumerate}
\end{defn}

\begin{notation}
We write $\Grp^{\aone}_k$ for the full subcategory of Nisnevich sheaves of groups consisting of strongly $\aone$-invariant sheaves of groups.    We write $\Ab^{\aone}_k$ for the full subcategory of Nisnevich sheaves of abelian groups consisting of strictly $\aone$-invariant sheaves.
\end{notation}

\begin{ex}
\label{ex:finiteetalegroupschemes}
If $G$ is a finite \'etale group scheme, then the simplicial classifying space $BG$ is $\aone$-local by \cite[\S 4 Proposition 3.5]{MV}, i.e., $G$ is strongly $\aone$-invariant.  In particular, if $G$ is a finite group, the associated constant group scheme is strongly $\aone$-invariant.  Thus, sending a finite group to its associated constant group scheme defines a functor from the category of finite groups to the category of strongly $\aone$-invariant sheaves of groups.
\end{ex}

We recall some key results of Morel on the structure of the categories of strongly and strictly $\aone$-invariant sheaves of groups.

\begin{thm}[Morel]
\label{thm:unstableconnectivity}
Assume $(\mathscr{X},x) \in \Spc_{k,*}$.
\begin{enumerate}[noitemsep,topsep=1pt]
\item The sheaf $\bpi_1^{\aone}(\mathscr{X},x)$ is a strongly $\aone$-invariant sheaf of groups.
\item Any strongly $\aone$-invariant sheaf of abelian groups that is pulled back from a perfect subfield is strictly $\aone$-invariant.
 \end{enumerate}
\end{thm}

\begin{proof}[Comments on the proof.]
The first statement in Theorem~\ref{thm:unstableconnectivity} is \cite[Theorem 6.1]{MField}.  For the second statement, we proceed as follows.  Morel proves in \cite[Theorem 5.46]{MField} that if $k$ is a perfect field, then strongly $\aone$-invariant sheaves of abelian groups are strictly $\aone$-invariant.  Our claim follows from the Morel's assertion by standard base-change results (see, e.g., \cite[Lemmas A.2 and A.4]{HoyoisHM}).  A common application of the above result is that if $(\mathscr{X},x)$ is a space over a field $k$ that is pulled back from the prime field, then for any integer $i \geq 2$, the sheaf $\bpi_i^{\aone}(\mathscr{X},x)$ is strictly $\aone$-invariant.
\end{proof}

\begin{rem}
\label{rem:stronglyaoneinvariantimpliesunramified}
The condition that a sheaf of groups be strongly or strictly $\aone$-invariant implies many useful properties that will be used in the sequel.  For example, if $\mathbf{G}$ is a strongly $\aone$-invariant sheaf, then $\mathbf{G}$ is {\em unramified} in the sense of \cite[Definition 2.1]{MField}; this follows from \cite[Corollary 6.9(2)]{MField}.  In particular, this means that if $X$ is any irreducible smooth $k$-scheme with function field $k(X)$, then the restriction map
\[
\mathbf{G}(X) \longrightarrow \mathbf{G}(k(X))
\]
is injective.  In particular, Theorem~\ref{thm:unstableconnectivity} implies that all the higher homotopy sheaves of a pointed space $(\mathscr{X},x) \in \Spc_k$ that are pulled back from a perfect subfield are unramified.
\end{rem}

Theorem \ref{thm:unstableconnectivity} allows one to construct a left adjoint to the inclusion $\Grp_k^{\aone} \hookrightarrow \Grp_k$.

\begin{defn}
If $\mathbf{G}$ is a Nisnevich sheaf of groups on $\Sm_k$, then its $\aone$-localization is defined by
\[
\mathbf{G}_{\aone} := \bpi_1^{\aone}(B\mathbf{G}).
\]
\end{defn}

\begin{ex}
The evident homomorphism $\mathbf{G} \to \mathbf{G}_{\aone}$ need not be injective.  Take $\mathbf{G} = SL_n$ (or any $\aone$-connected sheaf of groups); in that case one can show that $\bpi_1^{\aone}(BSL_n) = 1$.
\end{ex}

The next result is a consequence of \cite[Theorem 6.1.8]{MStable} and \cite[Lemma 6.2.13]{MStable}.

\begin{thm}[F. Morel]
\label{thm:stableconnectivity}
The category $\Ab^{\aone}_k$ is abelian, and the forgetful functor $\Ab^{\aone}_k \to \Ab_k$ is an exact full embedding. \qed
\end{thm}

\begin{rem}
As observed in \cite[Remark 6.2.14]{MStable}, the above result means, in particular, that any kernel or cokernel of a morphism of strictly $\aone$-invariant sheaves is again strictly $\aone$-invariant.  In fact, this consequence will be used repeatedly in the sequel.  In what follows, we will study corresponding questions for strongly $\aone$-invariant sheaves of groups where the situation is more complicated; see, e.g., Remark~\ref{rem:nocokernels} for further discussion of this point.
\end{rem}

\subsubsection*{Strongly $\aone$-invariant sheaves of groups}
We first examine how various group-theoretic constructions interact with the property of being $\aone$-local.

\begin{lem}
\label{lem:stronglyaoneinvariantsheavesofgroups}
Assume $k$ is a field.
\begin{enumerate}[noitemsep,topsep=1pt]
\item The kernel of a morphism in $\mathrm{Grp}^{\aone}_k$ is again a strongly $\aone$-invariant sheaf of groups.
\item Any extension of strongly $\aone$-invariant sheaves of groups is again strongly $\aone$-invariant.
\item If $\varphi: \mathbf{A} \to \mathbf{G}$ is any injective morphism of strongly $\aone$-invariant sheaves of groups, and $\mathbf{A}$ is i) pulled back from a perfect subfield of $k$ and ii) contained in the center of $\mathbf{G}$, then $\mathbf{G}/\mathbf{A}$ is strongly $\aone$-invariant as well.
\end{enumerate}
\end{lem}

\begin{proof}
Suppose $\varphi: \mathbf{H} \to \mathbf{G}$ is a morphism of strongly $\aone$-invariant sheaves of groups.  There is an induced morphism $B\mathbf{H} \to B\mathbf{G}$.  By assumption, both $B\mathbf{H}$ and $B\mathbf{G}$ are $\aone$-local.  Write $\mathscr{F}$ for the simplicial homotopy fiber of $\varphi$.  Since $\bpi_2(B\mathbf{G})$ vanishes by assumption, the associated long exact sequence in homotopy sheaves takes the form
\[
0 \longrightarrow \bpi_1(\mathscr{F}) \longrightarrow \mathbf{H} \longrightarrow \mathbf{G} \longrightarrow \bpi_0(\mathscr{F}).
\]
Appealing to Lemma~\ref{lem:homotopyfiberaonelocal} allows us to conclude that $\mathscr{F}$ is already $\aone$-local, which means that $\bpi_1(\mathscr{F}) = \bpi_1^{\aone}(\mathscr{F})$ is strongly $\aone$-invariant by Theorem~\ref{thm:unstableconnectivity}.

For the second point, suppose $1 \to \mathbf{G}_1 \to \mathbf{G}_2 \to \mathbf{G}_3 \to 1$ is an exact sequence of Nisnevich sheaves of groups and $\mathbf{G}_3$ is strongly $\aone$-invariant.  In that case, there is a simplicial fiber sequence
\[
B\mathbf{G}_1 \longrightarrow B \mathbf{G}_2 \longrightarrow B \mathbf{G}_3.
\]
Note that $B\mathbf{G}_3$ is simplicially connected by construction.  If $\mathbf{G}_1$ and $\mathbf{G}_3$ are strongly $\aone$-invariant, then $B\mathbf{G}_1$ and $B\mathbf{G}_3$ are $\aone$-local by assumption.  In that case, it follows from Lemma~\ref{lem:aww2210} that $B\mathbf{G}_2$ must be $\aone$-local as well, i.e., $\mathbf{G}_2$ is strongly $\aone$-invariant.

For the third point, by assumption there is a central extension of the form
\[
1 \longrightarrow \mathbf{A} \longrightarrow \mathbf{G} \longrightarrow \mathbf{G}/\mathbf{A} \longrightarrow 1.
\]
This central extension is classified by a simplicial fiber sequence of the form:
\[
B\mathbf{G} \longrightarrow B(\mathbf{G}/\mathbf{A}) \longrightarrow K(\mathbf{A},2).
\]
Since $\mathbf{G}$ is strongly $\aone$-invariant, $B\mathbf{G}$ is $\aone$-local.  Likewise, since $\mathbf{A}$ is strongly $\aone$-invariant, abelian and pulled back from a perfect subfield of $k$ it is strictly $\aone$-invariant by appeal to Theorem~\ref{thm:unstableconnectivity}.  Therefore, $K(\mathbf{A},2)$ is $\aone$-local.  Since $K(\mathbf{A},2)$ is simplicially connected by construction, we conclude that $B\mathbf{G}$ is $\aone$-local by appeal to Lemma~\ref{lem:aww2210}.
\end{proof}

\begin{rem}
\label{rem:nocokernels}
In recent work, Choudhury and Hogadi have established that the kernel and image of a morphism of strongly $\aone$-invariant sheaves are again strongly $\aone$-invariant \cite[Theorem 1.5]{ChoudhuryHogadi}.
%In contrast to the situation for strictly $\aone$-invariant sheaves (see Theorem~\ref{thm:stableconnectivity}), we do not know whether the cokernel of a morphism of strongly $\aone$-invariant sheaves is again strongly $\aone$-invariant in general.
\end{rem}

We now analyze existence of various limits in the category $\Grp^{\aone}_k$.

\begin{lem}
\label{lem:fiberproducts}
The category $\Grp^{\aone}_k$ has all finite limits, and the forgetful functor $\Grp^{\aone}_k \to \Grp_k$ creates limits.  More precisely,
\begin{enumerate}[noitemsep,topsep=1pt]
\item the trivial group is strongly $\aone$-invariant and is a terminal object; and
\item the pullback of a diagram of strongly $\aone$-invariant sheaves is again strongly $\aone$-invariant (in particular, the intersection of strongly $\aone$-invariant sheaves of groups is again strongly $\aone$-invariant).
\end{enumerate}
\end{lem}

\begin{proof}
For the first point, observe that the trivial group $1$ is evidently strongly $\aone$-invariant; that it is a terminal object in $\Grp^{\aone}_k$ follows from the fact that it is a terminal object in $\Grp_k$.

For the second point, suppose $\mathbf{G}$, $\mathbf{H}_1$ and $\mathbf{H}_2$ are sheaves of groups, and assume we are given morphisms $\varphi_1: \mathbf{H}_1 \to \mathbf{G}$ and $\varphi_2: \mathbf{H}_2 \to \mathbf{G}$.  In that case, the simplicial homotopy fiber product $B\mathbf{H}_1 \times^{h}_{B{\mathbf G}} B\mathbf{H}_2$ fits into a fiber sequence of the form
\[
\mathbf{G} \longrightarrow B\mathbf{H}_1 \times^{h}_{B{\mathbf G}} B\mathbf{H}_2 \longrightarrow B \mathbf{H}_1 \times B\mathbf{H}_2,
\]
By, e.g., \cite[Corollary 2.2.3]{MayPonto}, this description implies that $\bpi_1^s(B\mathbf{H}_1 \times^{h}_{B{\mathbf G}} B\mathbf{H}_2)$ coincides with the pullback of the diagram $\mathbf{H}_1 \to \mathbf{G} \leftarrow \mathbf{H}_2$ (because $B \mathbf{H}_1 \times B\mathbf{H}_2$ is $1$-truncated).

If $\mathbf{G}$,$\mathbf{H}_1$ and $\mathbf{H}_2$ are all strongly $\aone$-invariant, then the simplicial homotopy fiber product is $\aone$-local, again by Lemma~\ref{lem:aww2210} applied to the simplicial fiber sequence of the previous paragraph ($B \mathbf{H}_1 \times B\mathbf{H}_2$ is simplicially connected as before).  Thus,
\[
\mathbf{H}_1 \times_{\mathbf{G}} \mathbf{H}_2 = \bpi_1^s(B\mathbf{H}_1 \times^{h}_{B{\mathbf G}} B\mathbf{H}_2) = \bpi_1^{\aone}(B\mathbf{H}_1 \times^{h}_{B{\mathbf G}} B\mathbf{H}_2),
\]
which is precisely what we wanted to show.

Any category that has pullbacks and a terminal object necessarily possesses all finite limits \cite[Corollary V.2.1]{MacLane}.  Finally, the results above establish that the inclusion functor preserves the terminal object and pullbacks and therefore, creates finite limits as well.
\end{proof}

\subsubsection*{Contractions}
\begin{defn}
\label{defn:contraction}
If $\mathbf{G}$ is a strongly $\aone$-invariant sheaf of groups, the {\em contraction} of $\mathbf{G}$, denoted $\mathbf{G}_{-1}$, is defined by
\[
\mathbf{G}_{-1}(U) := \ker(\mathbf{G}(\gm{} \times U) \stackrel{ev_1}{\longrightarrow} \mathbf{G}(U)).
\]
\end{defn}

The next result summarizes the important formal properties of contraction.

\begin{prop}[Morel]
\label{prop:contractionisexact}
Assume $k$ is a field.
\begin{enumerate}[noitemsep,topsep=1pt]
\item The functor $\mathbf{G} \mapsto (\mathbf{G})_{-1}$ defines an endo-functor of $\Grp_k^{\aone}$;
\item if $k$ is furthermore perfect, then contraction restricts to an endo-functor of $\Ab^{\aone}_k$ to $\Ab^{\aone}_k$.
\item The contraction functor preserves short exact sequences.
\end{enumerate}
\end{prop}

\begin{proof}
The assignment $\mathbf{G} \mapsto \mathbf{G}_{-1}$ is evidently functorial.  That $\mathbf{G}_{-1}$ is again strongly $\aone$-invariant is \cite[Lemma 2.32]{MField}.  The second statement follows from the first and Theorem~\ref{thm:unstableconnectivity}.  The third statement is \cite[Lemma 7.33]{MField}.
\end{proof}

The next result explains the geometric importance of the contraction construction.

\begin{thm}[{\cite[Theorem 6.13]{MField}}]
\label{thm:contractionsandgmloops}
Suppose $(\mathscr{X},x)$ is a pointed $\aone$-connected space. The following statements hold:
\begin{enumerate}[noitemsep,topsep=1pt]
\item the space $\Omega_{\gm{}}\mathscr{X}$ is again pointed and $\aone$-connected, and
\item $\bpi_i^{\aone}(\Omega_{\gm{}}\mathscr{X}) = \bpi_i^{\aone}(\mathscr{X})_{-1}$.
\end{enumerate}
\end{thm}

\subsubsection*{Strong $\aone$-invariance of group-theoretic constructions}
If $\mathbf{G}$ is a sheaf of groups, then we write $Z(\mathbf{G})$ for the kernel of the morphism $\mathbf{G} \to \mathbf{Aut}(\mathbf{G})$ (sectionwise, sending an element to its corresponding inner automorphism).  There is a natural homomorphism
\[
Z(\mathbf{G}) \times \mathbf{G} \longrightarrow \mathbf{G}
\]
given by the product.

More generally, suppose we are given a homomorphism of sheaves of groups $\varphi: \mathbf{H} \to \mathbf{G}$.  Consider the sheaf of pointed sets $\hom(\mathbf{H},\mathbf{G})$.  This sheaf of pointed sets comes equipped with an action of the sheaf $\mathbf{G}$ by conjugation on the target.  Viewing $\varphi$ as a global section of $\hom(\mathbf{H},\mathbf{G})$, we define $C_{\mathbf{G}}(\varphi)$ to be the sheaf-theoretic stabilizer of the section $\varphi \in \hom(\mathbf{H},\mathbf{G})$ under the conjugation action of $\mathbf{G}$.

%we define the normalizer $N_{{\mathbf G}}(\varphi)$ to be the sheafification of the presheaf assigning to $U \in \Sm_k$, the set of $g \in \mathbf{G}(U)$ such that $g \im(\varphi)(U) g^{-1} = \im(\varphi)(U)$.  There is an induced homomorphism $N_{{\mathbf G}}(\varphi) \longrightarrow \mathbf{Aut}(im(\varphi))$, and we define $C_{{\mathbf G}}(\varphi)$ to be the kernel of this homomorphism of sheaves of groups.  

The homomorphism $\varphi$ induces an evaluation homomorphism
\[
C_{\mathbf{G}}(\varphi) \times \mathbf{H} \longrightarrow \mathbf{G};
\]
when $\varphi$ is the identity homomorphism on $\mathbf{G}$, this homomorphism coincides with the one defined in the preceding paragraph.

With $\varphi: \mathbf{H} \to \mathbf{G}$ as above, write $\operatorname{Map}(B\mathbf{H},B\mathbf{G})_\varphi$ for the component of the internal mapping space containing $B\varphi$.  The homomorphism of the preceding paragraph induces a morphism
\[
BC_{{\mathbf G}}(\varphi) \longrightarrow \operatorname{Map}(B\mathbf{H},B\mathbf{G})_{\varphi},
\]
and applying $\bpi_1$ defines a homomorphism
\[
C_{{\mathbf G}}(\varphi) \longrightarrow \bpi_1(\operatorname{Map}(B\mathbf{H},B\mathbf{G})_{\varphi}).
\]
Again, in the special case where $\varphi$ is the identity map on $\mathbf{G}$, this reduces to a homomorphism $Z(\mathbf{G}) \to \bpi_1(\operatorname{Map}(B\mathbf{G},B\mathbf{G})_1)$.

\begin{prop}
\label{prop:strongaoneinvarianceofsomecentralizers}
Assume $k$ is a field, and $\varphi: \mathbf{H} \to \mathbf{G}$ is a morphism in $\Grp_k$.
\begin{enumerate}[noitemsep,topsep=1pt]
\item The morphism
\[
C_{{\mathbf G}}(\varphi) \longrightarrow \bpi_1(\operatorname{Map}(B\mathbf{H},B\mathbf{G})_{\varphi})
\]
is an isomorphism of sheaves of groups.
\item If $\mathbf{G}$ is strongly $\aone$-invariant, then so is $C_{{\mathbf G}}(\varphi)$.
\item If $k$ is a perfect field, then $Z({\mathbf{G}})$ is strictly $\aone$-invariant.
\end{enumerate}
\end{prop}

\begin{proof}
The first statement is a sheaf-theoretic variant of a classical statement about discrete groups, which is essentially an exercise in covering space theory.  Indeed, we claim that there is an evaluation morphism $\operatorname{Map}(B\mathbf{H},B\mathbf{G})_{\varphi} \to B\mathbf{G}$ that fits into a simplicial fiber sequence of the form:
\[
\hom(\mathbf{H},\mathbf{G}) \longrightarrow \operatorname{Map}(B\mathbf{H},B\mathbf{G}) \longrightarrow B\mathbf{G},
\]
where $\hom(\mathbf{H},\mathbf{G})$ is viewed as a simplicially constant presheaf, i.e., all face and degeneracy maps are reduced to the identity.  Granting this, taking homotopy sheaves of this fiber sequence with base-point $\varphi$ yields a long exact sequence in homotopy sheaves takes the form
\[
1 \longrightarrow \bpi_1(\operatorname{Map}(B\mathbf{H},B\mathbf{G})_{\varphi}) \longrightarrow \mathbf{G} \longrightarrow \hom(\mathbf{H},\mathbf{G}) \longrightarrow \cdots
\]
and the first statement follows immediately from the fact that the action of $\mathbf{G}$ on $\hom(\mathbf{H},\mathbf{G})$ is induced by conjugation on the target, in conjunction with the definition of $C_{{\mathbf G}}(\varphi)$.

To produce this fiber sequence, recall that if $H$ and $G$ are (discrete) groups, we may form the internal hom-groupoid $\hom(BH,BG)$: objects are group homomorphisms $H \to G$ and given two objects $f: H \to G$ and $f': H \to G$, a morphism from $f$ to $f'$ is an element $g \in G$ conjugating $f$ to $f'$.  Equivalently, $\hom(BH,BG)$ is the ``action groupoid" attached to the set $\hom(H,G)$ viewed as a $G$-set with $G$ acting by conjugation on the target.  There is an associated morphism of groupoids $\hom(BH,BG) \to BG$: send an object $f: H \to G$ to the unique object of $BG$ and an element $g \in G$ conjugating $f$ to $f'$ to the element $g$.  In terms of action groupoids, this is functor corresponding to the morphism of $G$-sets $\hom(H,G) \to \ast$.  Covering space theory implies that the nerve of the groupoid $\hom(BH,BG)$ coincides with the mapping space $\operatorname{Map}(BH,BG)$ at which point one immediately deduces the existence of a fiber sequence as above in the category of simplicial sets. 

The general case reduces to the case just mentioned using the homotopy-theoretic interpretation of stacks in terms of $1$-truncated simplicial presheaves (i.e., having no homotopy in degrees $\geq 1$; see \cite{Hollander,JardineStacks}).  Indeed, consider the action of $\mathbf{G}$ on $\hom(\mathbf{H},\mathbf{G})$ by conjugation; the homotopy quotient of this action fits into a simplicial fiber sequence of the form
\[
\hom(\mathbf{H},\mathbf{G}) \longrightarrow \hom(\mathbf{H},\mathbf{G})_{h\mathbf{G}} \longrightarrow B\mathbf{G},
\]
$\hom(\mathbf{H},\mathbf{G})_{h\mathbf{G}}$ is modelled by the usual simplicial Borel construction.

To establish the original claim, it suffices to show that the middle term is Nisnevich locally weakly equivalent to $\operatorname{Map}(B\mathbf{H},B\mathbf{G})$, in which case the original statement follows from the discussion above by taking the fiber over the base-point corresponding to $\varphi$.  To see this, we use the following model for the mapping space.  The space $B\mathbf{G}$ is canonically simplicially weakly equivalent to the simplicial presheaf $B\mathrm{Tors}(\mathbf{G})$ that assigns to $U \in \Sm_k$ the nerve of the groupoid $B\mathrm{Tors}(\mathbf{G})(U)$ of $\mathbf{G}$-torsors on $U$ by \cite[Lemma 2.3.2(1)]{AHWII}.   

There is a map $\hom(\mathbf{H},\mathbf{G})_{h\mathbf{G}} \to \operatorname{Map}(B\mathbf{H},B\mathrm{Tors}(\mathbf{G}))$ induced by the classifying space functor: at the level of $0$-simplices: a homomorphism $\mathbf{H} \to \mathbf{G}$ defines a morphism $B\mathbf{H} \to B\mathbf{G}$ and conjugate homomorphisms induce the same map.  That this map is a Nisnevich local weak equivalence can be checked sectionwise; we treat the case of $0$-simplices as the general case is similar.  A section of $\operatorname{Map}(B\mathbf{H},B\mathrm{Tors}(\mathbf{G}))(U)$ is a $\mathbf{G}$-torsor over $B\mathbf{H} \times U$.  By picking a refinement $U' \to U$, the composite map $E\mathbf{H} \times U' \to B\mathbf{G}$ admits a trivialization.  In that case, a choice of such a trivialization then corresponds to a section of $\hom(\mathbf{H},\mathbf{G})$ over $U'$.  Unwinding the definitions, two local sections determine isomorphic torsors if they differ by conjugation.  This construction is evidently inverse to the one given by the ``classifying space" map above, which is thus a Nisnevich local weak equivalence.

%The associated long-exact sequence in homotopy takes the form
%\[
%1 \longrightarrow \pi_1(\operatorname{Map}(BH,BG)_{\varphi}) \longrightarrow G \longrightarrow \hom(H,G) \longrightarrow \pi_0(\operatorname{Map}(BH,BG)).
%\]
%Since $G$ acts on $\hom(H,G)$ by conjugation, exactness of the sequence together with the definition of $C_G(\varphi)$ yields the identification $\pi_1(\operatorname{Map}(BH,BG)_{\varphi}) = C_G(\varphi)$.

%For the first statement, we observe that is suffices to check the required isomorphism stalkwise.  Since $B\mathbf{G}$ is stalkwise a Kan complex, we immediately reduce to proving the same statement where $G$ is a discrete group.  The case of discrete groups is ``classical"; see, e.g., the discussion just after \cite[Theorem 1.1']{DwyerZabrodsky}.

For the second statement, it suffices to observe that $\operatorname{Map}(B\mathbf{H},B\mathbf{G})_\varphi$ is $\aone$-local.  To see this, observe that since $B\mathbf{G}$ is $\aone$-local, $\operatorname{Map}(\mathscr{Y},B\mathbf{G})$ is $\aone$-local for any $\mathscr{Y}$.  Likewise, any component of an $\aone$-local space is $\aone$-local, so it follows that $\operatorname{Map}(B\mathbf{H},B\mathbf{G})_\varphi$ is $\aone$-local is well.  In that case, $\bpi_1(\operatorname{Map}(B\mathbf{H},B\mathbf{G})_\varphi)$ is strongly $\aone$-invariant, and the strong $\aone$-invariance of $C_{{\mathbf G}}(\varphi)$ then follows immediately from Point (1).

The third statement follows from the second  by taking $\varphi$ to be the identity map on $\mathbf{G}$ and using the fact that $Z({\mathbf{G}})$ is a sheaf of abelian groups.
\end{proof}

\begin{rem}
	The assertion that the center of a strongly $\aone$-invariant sheaf is again strongly $\aone$-invariant has been recently established by different means by Choudhury and Hogadi \cite[Theorem 3.1]{ChoudhuryHogadi}.
\end{rem}

If $\mathbf{G}$ is a sheaf of groups, define a sequence of quotients of $\mathbf{G}$ inductively as follows.  Sets $\mathbf{G}_0 = \mathbf{G}$ and define $\mathbf{G}_i$, $i \geq 1$ inductively by setting $\mathbf{G}_i = \mathbf{G}_{i-1}/Z(\mathbf{G}_{i-1})$.  In that case, we define the $i$-th higher center $Z_i(\mathbf{G})$ as the kernel of the map $\mathbf{G} \to \mathbf{G}_i$.

\begin{prop}
\label{prop:uppercentralseries}
Suppose $\mathbf{G}$ is a strongly $\aone$-invariant sheaf of groups.  The following statements hold.
\begin{enumerate}[noitemsep,topsep=1pt]
\item The higher centers $Z_i(\mathbf{G})$ of $\mathbf{G}$ are all strongly $\aone$-invariant normal subsheaves of groups of $\mathbf{G}$;
\item each of these sheaves is {\em characteristic}, i.e., it is stable under the natural action of $\mathbf{Aut}(\mathbf{G})$; and
\item the quotients $\mathbf{G}/Z_i(\mathbf{G})$ are strongly $\aone$-invariant.
\end{enumerate}
\end{prop}

\begin{proof}
The sheaf $\mathbf{G}_0 := \mathbf{G}$ is strongly $\aone$-invariant by assumption.  Since $\mathbf{G}_i = \mathbf{G}_{i-1}/Z(\mathbf{G}_{i-1})$, it follows inductively by combining Proposition~\ref{prop:strongaoneinvarianceofsomecentralizers}(3) and Lemma~\ref{lem:stronglyaoneinvariantsheavesofgroups}(3) that $\mathbf{G}_i$ is strongly $\aone$-invariant.  The fact that $Z_i(\mathbf{G})$ is strongly $\aone$-invariant follows from Lemma~\ref{lem:stronglyaoneinvariantsheavesofgroups}(1) since it is the kernel of a map of strongly $\aone$-invariant sheaves.  The final statement is immediate since $\mathbf{G}/Z_i(\mathbf{G}) = \mathbf{G}_i$ by definition.  The statement about being characteristic subgroup sheaves may be checked stalkwise, in which case it reduces to the corresponding fact in group theory.
\end{proof}

We also make one further remark, unrelated to the above discussion, for future use.

\begin{rem}
\label{rem:freegroups}
If $\mathscr{S}$ is a pointed sheaf of sets, then the free strongly $\aone$-invariant sheaf of groups on $\mathscr{S}$ is defined to be $\bpi_1^{\aone}(\Sigma \mathscr{S})$.  One may check that this construction defines a functor from the category of sheaves of sets to the category of strongly $\aone$-invariant sheaves of groups that is left adjoint to the evident forgetful functor \cite[Lemma 7.23]{MField}.   Morel writes $F_{\aone}(\mathscr{S})$ for this construction in \cite[p. 189]{MField}.  For our purposes, it will be important to remember that $\bpi_1^{\aone}(\pone) = F_{\aone}(\gm{})$; we will come back to this observation in Example~\ref{ex:pi1poneaonenilpotent}.
\end{rem}

\subsection{$\aone$-nilpotent groups}
\label{ss:aonenilpotentgroups}
We now define $\aone$-nilpotent actions and $\aone$-nilpotent sheaves of groups in parallel with the classical story: the latter are defined as iterated central extensions of strictly $\aone$-invariant sheaves of groups.

\begin{defn}
\label{defn:aonenilpotent}
Suppose $\mathbf{G}$ and $\mathbf{H}$ are strongly $\aone$-invariant sheaves of groups.
\begin{enumerate}[noitemsep,topsep=1pt]
\item An action $\varphi$ of $\mathbf{G}$ on $\mathbf{H}$ is said to be {\em $\aone$-nilpotent} (resp. {\em locally $\aone$-nilpotent}) if $\mathbf{H}$ has a finite (resp. locally finite) $\mathbf{G}$-central series (see \textup{Definitions~\ref{defn:nilpotentsheaf} and~\ref{defn:localnilpotence}})
    \[
    \mathbf{H} = \mathbf{H}_0 \supset \mathbf{H}_1 \cdots
    \]
    such that the successive subquotients $\mathbf{H}_i/\mathbf{H}_{i+1}$ are strongly $\aone$-invariant.
\item A $\mathbf{G}$-central series for $\mathbf{H}$ as in the preceding point will be called an {\em $\aone$-$\mathbf{G}$-central series}; the smallest integer $i$ such that $\mathbf{H}_j = 1$ for all $j \geq i$ will be called the length of the series.
\item We will say that $\mathbf{G}$ is {\em (locally) $\aone$-nilpotent} if the conjugation action of $\mathbf{G}$ on itself is (locally) $\aone$-nilpotent, i.e., $\mathbf{G}$ admits a decreasing filtration by normal subgroup sheaves with abelian strongly $\aone$-invariant subquotients and such that the induced action of $\mathbf{G}$ on successive subquotients is trivial; in this case the filtration will be called an {\em $\aone$-central series for $\mathbf{G}$}.
\end{enumerate}
\end{defn}

\begin{rem}
Given the above definition, a number of remarks are in order.
\begin{enumerate}[noitemsep,topsep=1pt]
\item Lemma~\ref{lem:stronglyaoneinvariantsheavesofgroups} states that kernels of morphisms of strongly $\aone$-invariant sheaves of groups are again strongly $\aone$-invariant.  Since $\mathbf{H}_{i+1}$ is the kernel of the morphism $\mathbf{H}_i \to \mathbf{H}_i/\mathbf{H}_{i+1}$, and the latter is strongly $\aone$-invariant by assumption, we conclude inductively that each $\mathbf{H}_i$ is a strongly $\aone$-invariant subsheaf of $\mathbf{H}$.
\item By Theorem \ref{thm:unstableconnectivity}, assuming $k$ is perfect, the condition $\mathbf{H}_i/\mathbf{H}_{i+1}$ is strongly $\aone$-invariant implies $\mathbf{H}_i/\mathbf{H}_{i+1}$ is strictly $\aone$-invariant: indeed, $\mathbf{H}_i/\mathbf{H}_{i+1}$ is abelian since it arises as a subquotient in a $\mathbf{G}$-central series.
\item If $\mathbf{G}$ is an $\aone$-nilpotent sheaf of groups, then the condition that $\mathbf{G}_i/\mathbf{G}_{i+1}$ is abelian is implied by the assumption that the conjugation action of $\mathbf{G}$ on $\mathbf{G}_i/\mathbf{G}_{i+1}$ is trivial.  As usual, if $\mathbf{G}$ is $\aone$-nilpotent, the extensions
    \[
    1 \longrightarrow \mathbf{G}_i/\mathbf{G}_{i+1} \longrightarrow \mathbf{G}/\mathbf{G}_{i+1} \longrightarrow \mathbf{G}/\mathbf{G}_i \longrightarrow 1
    \]
    are all central extensions of strongly $\aone$-invariant sheaves.
\end{enumerate}
\end{rem}

\begin{prop}
\label{prop:characterizationofaonenilpotent}
Assume $k$ is a perfect field, and $\mathbf{G}$ is a sheaf of groups.
\begin{enumerate}[noitemsep,topsep=1pt]
\item The sheaf $\mathbf{G}$ is $\aone$-nilpotent if and only if it is strongly $\aone$-invariant and the upper central series terminates in finitely many steps.
\item In particular, a sheaf of groups $\mathbf{G}$ is $\aone$-nilpotent if and only if $\mathbf{G}$ is strongly $\aone$-invariant and nilpotent as a sheaf of groups.
\end{enumerate}
\end{prop}

\begin{proof}
The standard proof from group theory that a group is nilpotent if and only if its upper central series terminates in finitely many steps works for sheaves of groups.  Next, observe that Proposition~\ref{prop:uppercentralseries} guarantees that if $\mathbf{G}$ is strongly $\aone$-invariant that the upper central series has terms that are strongly $\aone$-invariant with strictly $\aone$-invariant subquotients.  Therefore, if $\mathbf{G}$ has a finite upper central series, it has a finite $\aone$-central series.  Conversely, if $\mathbf{G}$ has a finite $\aone$-central series, then the upper central series must terminate in finitely many steps as well.  The second statement follows immediately from the first.
\end{proof}

\begin{rem}
If $\mathbf{G}$ is a strongly $\aone$-invariant sheaf of groups, then one may define its ``$\aone$-abelianization" $\mathbf{G}^{ab} := \mathbf{H}_1^{\aone}(B\mathbf{G})$ (see  Section~\ref{ss:sheafcohomology}); this sheaf is the initial strictly $\aone$-invariant sheaf of abelian groups admitting a morphism from $\mathbf{G}$.  Choudhury and Hogadi prove that the map $\mathbf{G} \to \mathbf{G}^{ab}$ is an epimorphism \cite[Theorem 1.1]{ChoudhuryHogadi}.  One may then inductively construct a lower $\aone$-central series for $\mathbf{G}$.
\end{rem}
	
%If $\mathbf{G}$ is a strongly $\aone$-invariant sheaf of groups, then we do not know if the commutator subgroup sheaf of $\mathbf{G}$ or the abelianization $\mathbf{G}^{ab}$ of $\mathbf{G}$ are strongly $\aone$-invariant.  As a consequence, it is unclear whether one may define a ``lower $\aone$-central series" for $\mathbf{G}$.  One may define a candidate ``$\aone$-abelianization" for $\mathbf{G}$ using the first $\aone$-homology sheaf of $B\mathbf{G}$ as defined in Section~\ref{ss:sheafcohomology}: this is the initial strictly $\aone$-invariant sheaf of abelian groups admitting a morphism from $\mathbf{G}$.  However, one does not know if the resulting morphism from $\mathbf{G}$ to its $\aone$-abelianization is a surjective morphism of sheaves in general.

\begin{ex}
\label{ex:pointwiseaonenilpotent}
The key example of a locally $\aone$-nilpotent sheaf of groups arises from the scaling action of $\gm{}$ on the first non-vanishing $\aone$-homotopy sheaf of ${\mathbb A}^{n} \setminus 0$, $n \geq 2$, at least if we work over a base field that is not formally real.  (Note: the action of $\gm{}$ on ${\mathbb A}^n \setminus 0$ does not preserve base-points, but this is irrelevant for the action on the {\em first} $\aone$-homotopy sheaf because it can be interpreted in terms of $\aone$-homology via the Hurewicz theorem \ref{thm:relativeHurewicz}.)  In this case $\bpi_{n-1}^{\aone}({\mathbb A}^n \setminus 0) = \K^{MW}_n$, which has a filtration by the subsheaves $\mathbf{I}^{j}$ for $j \geq n+1$, all of which are stable under multiplication by units.  The successive subquotients are isomorphic to either $\K^M_n$ or $\K^M_{j}/2$, $j \geq n+1$, with a trivial action of $\gm{}$ in either.  As in Example~\ref{ex:pointwisenilpotent}, the induced filtration on sections over any extension of the base field will be finite.
\end{ex}

\begin{lem}
\label{lem:aonenilpotentactionsarefunctorial}
Suppose $\varphi: \mathbf{G}' \to \mathbf{G}$ is a morphism of strongly $\aone$-invariant sheaves of groups.  If $\mathbf{H}$ is a strongly $\aone$-invariant sheaf of groups carrying a (locally) $\aone$-nilpotent action of $\mathbf{G}$, then the action of $\mathbf{G}'$ on $\mathbf{H}$ induced by precomposition with $\varphi$ is also (locally) $\aone$-nilpotent.
\end{lem}

\begin{proof}
Any (locally finite) $\mathbf{G}$-central series for $\mathbf{H}$ defines, by precomposition with $\varphi$, a (locally finite) $\mathbf{G}'$-central series for $\mathbf{H}$, i.e., this is a special case of Lemma \ref{lem:functorialityofnilpotence}.
\end{proof}

\begin{lem}
\label{lem:subsquotientsandcentralextensions}
Suppose $\mathbf{G} \in \mathrm{Grp}^{\aone}_k$.
\begin{enumerate}[noitemsep,topsep=1pt]
\item If $\mathbf{G}$ is a (locally) $\aone$-nilpotent group, then any strongly $\aone$-invariant subsheaf of groups of $\mathbf{G}$ is again (locally) $\aone$-nilpotent.
\item If $k$ is a perfect field, then any strongly $\aone$-invariant quotient of $\mathbf{G}$ is again $\aone$-nilpotent.
\item If $\mathbf{A}$ is a strongly $\aone$-invariant sheaf of abelian groups, and $\mathbf{G}$ is $\aone$-nilpotent, then a central extension of $\mathbf{G}$ by $\mathbf{A}$ is again $\aone$-nilpotent.
\item If $\mathbf{H}$ and $\mathbf{H}'$ are $\aone$-nilpotent sheaves of groups, and we are given morphisms $\mathbf{H} \to \mathbf{G}$ and $\mathbf{H}' \to \mathbf{G}$, then $\mathbf{H} \times_{\mathbf{G}} \mathbf{H}'$ is again $\aone$-nilpotent.
\end{enumerate}
In particular, the category of $\aone$-nilpotent sheaves of groups is a subcategory of $\Grp^{\aone}_k$ that is stable under formation of finite limits.
\end{lem}

\begin{proof}
Assume $\mathbf{G}$ is an $\aone$-nilpotent sheaf of groups.  For the first statement, suppose $\mathbf{H}$ is a strongly $\aone$-invariant subsheaf of groups of $\mathbf{G}$.  In that case, Lemma~\ref{lem:fiberproducts} guarantees that the intersection of the sheaves in an $\aone$-central series for $\mathbf{G}$ with $\mathbf{H}$ provides an $\aone$-central series for $\mathbf{H}$.  One argues in an identical way to establish local $\aone$-nilpotence of $\mathbf{H}$ given the corresponding statement for $\mathbf{G}$.

For the second statement, observe that any quotient of a nilpotent sheaf of groups is a nilpotent sheaf of groups.  Since the quotient is assumed strongly $\aone$-invariant, the result follows from Proposition~\ref{prop:characterizationofaonenilpotent}.

For the third statement, suppose we have a central extension $\mathbf{G}'$ of $\mathbf{G}$ by $\mathbf{A}$.  If $\mathbf{G}_i$ is a filtration of $\mathbf{G}$ witnessing $\aone$-nilpotence with $\mathbf{G}_n = 1$, then Lemma \ref{lem:fiberproducts} implies that $\mathbf{G}'_i := \mathbf{G}_i \times_{\mathbf{G}} \mathbf{G}'$ yields a sequence of strongly $\aone$-invariant normal subsheaves containing $\mathbf{A}$.  Setting $\mathbf{G}'_{n+1} = 1$, one checks that $\mathbf{G}'_i$ provides an $\aone$-central series for $\mathbf{G}'$ as in the corresponding statement in classical group theory.

For the final assertion, observe that if $\mathbf{H}$ and $\mathbf{H}'$ are $\aone$-nilpotent, then $\mathbf{H} \times \mathbf{H}'$ is also $\aone$-nilpotent: if $\mathbf{H}_i$ and $\mathbf{H}_i'$ are $\aone$-central series (of possibly different lengths) for each group, then $\mathbf{H}_i \times \mathbf{H}_i'$ is an $\aone$-central series for the product.  As the fiber product is strongly $\aone$-invariant by Lemma \ref{lem:fiberproducts}, the fact that $\mathbf{H} \times_{\mathbf{G}} \mathbf{H}'$ is $\aone$-nilpotent follows from Point (1).

That the category of $\aone$-nilpotent sheaves of groups is stable under formation of finite limits now follows from the fact that the terminal object in the category of groups (i.e., the trivial sheaf of groups) is $\aone$-nilpotent, and the fact that this category is closed under pullbacks.
\end{proof}

\begin{rem}
After Lemma \ref{lem:subsquotientsandcentralextensions}, the class of $\aone$-nilpotent groups can be described as the smallest subcategory of $\Grp^{\aone}_k$ containing all strongly $\aone$-invariant sheaves of abelian groups and closed under central extensions.  Arbitrary extensions of $\aone$-nilpotent sheaves of groups need not be $\aone$-nilpotent; this follows from Example~\ref{ex:finiteetalegroupschemes} since constant group schemes attached to finite groups are strongly $\aone$-invariant.
\end{rem}

\begin{defn}
\label{defn:aonenilpotenceclass}
If $\mathbf{G}$ is $\aone$-nilpotent, we will say that $\mathbf{G}$ has {\em $\aone$-nilpotence class $c$} if the minimum length of an $\aone$-central series is $c$.   Similarly, if $\mathbf{H}$ is a strongly $\aone$-invariant group that has an $\aone$-nilpotent action of a strongly $\aone$-invariant group $\mathbf{G}$, the {\em $\aone$-$\mathbf{G}$-nilpotence class} of $\mathbf{H}$ is the minimum length of an $\aone$-$\mathbf{G}$-central series for $\mathbf{H}$.
\end{defn}

\begin{lem}
If $k$ is a perfect field, and $\mathbf{G} \in \mathrm{Grp}^{\aone}_k$ is an $\aone$-nilpotent sheaf of groups, then the $\aone$-nilpotence class of $\mathbf{G}$ is the length of the upper central series for $\mathbf{G}$.
\end{lem}

\begin{proof}
This fact follows immediately from the characterization of $\aone$-nilpotent sheaves of groups given in Proposition~\ref{prop:characterizationofaonenilpotent}, and the proof of the corresponding fact in group theory.
\end{proof}

\begin{ex}
\label{ex:pi1poneaonenilpotent}
The most basic example of an $\aone$-nilpotent sheaf of groups is $\bpi_1^{\aone}({\mathbb P}^1)$; this is a non-abelian group which Morel shows is a central extension of $\gm{}$ by $\K^{\MW}_2$ \cite[Theorem 7.29 and Remark 7.31]{MField}.  This example presents the interesting situation that the free strongly $\aone$-invariant sheaf of groups on a (non-trivial) sheaf of sets $\mathscr{S}$ may be $\aone$-nilpotent (see Remark~\ref{rem:freegroups} for discussion of the free strongly $\aone$-invariant sheaf of groups on a sheaf of sets; in this case $\mathscr{S} = \gm{}$).  This example will be discussed in greater detail in Remark~\ref{rem:nonabelianaonenilpotent}.
\end{ex}

\subsubsection*{More on $\aone$-nilpotent actions}
Momentarily, we will define $\aone$-nilpotent spaces.  We consider in more detail $\aone$-nilpotent actions of a strongly $\aone$-invariant sheaf of groups.  The following result will be a key technical tool; it is the analog for $\aone$-nilpotent groups of Lemma~\ref{lem:shortexactsequencesofnilpotentactions}, though we note that the statement differs slightly from that one because of the behavior of limits and colimits in the category of strongly $\aone$-invariant sheaves.

\begin{prop}
\label{prop:extensionsofactions}
Assume $k$ is a perfect field, $\mathbf{G} \in \Grp^{\aone}_k$,  $\mathbf{H},\mathbf{H}'$ and $\mathbf{H}''$ are strongly $\aone$-invariant sheaves with an action of $\mathbf{G}$, and
\[
1 \longrightarrow \mathbf{H}' \longrightarrow \mathbf{H} \longrightarrow \mathbf{H}'' \longrightarrow 1
\]
is a $\mathbf{G}$-equivariant short exact sequence of strongly $\aone$-invariant sheaves.
\begin{enumerate}[noitemsep,topsep=1pt]
\item If the actions of $\mathbf{G}$ on $\mathbf{H}'$ and $\mathbf{H}''$ are (locally) $\aone$-nilpotent, then the action on $\mathbf{H}$ is (locally) $\aone$-nilpotent.
\item If the action of $\mathbf{G}$ on $\mathbf{H}$ is (locally) $\aone$-nilpotent, then the action of $\mathbf{G}$ on $\mathbf{H}'$ is (locally) $\aone$-nilpotent, and, if the sequence is furthermore a central extension, then the action of $\mathbf{G}$ on $\mathbf{H}''$ is also (locally) $\aone$-nilpotent.
%\item If the action of $\mathbf{G}$ on $\mathbf{H}'$ is locally $\aone$-nilpotent and the action of $\mathbf{G}$ on $\mathbf{H}''$ is $\aone$-nilpotent, then the action on $\mathbf{H}$ is locally $\aone$-nilpotent.
\end{enumerate}
\end{prop}

\begin{proof}
For the first statement, it suffices to observe that Lemma~\ref{lem:stronglyaoneinvariantsheavesofgroups}(2) implies that strongly $\aone$-invariant sheaves of groups are stable by extensions; the result then follows from the corresponding sheaf-theoretic statement, i.e., Lemma~\ref{lem:shortexactsequencesofnilpotentactions}.

For the second statement, we proceed as follows.  We prove the statement for locally $\aone$-nilpotent actions; the proof in the $\aone$-nilpotent case follows by simply dropping the word locally.  Suppose we have a locally finite $\aone$-$\mathbf{G}$-central series for $\mathbf{H}$.  Lemma~\ref{lem:fiberproducts} guarantees that the restriction of the corresponding filtration to $\mathbf{H}'$ provides a locally finite $\mathbf{G}$-central series of $\mathbf{H}'$.

We claim that, if the exact sequence is, furthermore, a central extension, then the image of the locally finite $\aone$-$\mathbf{G}$-central series for $\mathbf{H}$ in $\mathbf{H}''$ is also a locally finite $\aone$-$\mathbf{G}$-central series for the target.  To see this, it suffices to show that the image of each term in $\mathbf{H}''$ is again strongly $\aone$-invariant.

To this end, let $\mathbf{H}_i$ be an arbitrary term in the locally finite $\aone$-$\mathbf{G}$-central series for $\mathbf{H}$.  We already know that $\mathbf{H}_i \cap \mathbf{H}' =: \mathbf{H}'_i$ is strongly $\aone$-invariant.  However, since the exact sequence in the statement is a central extension, we conclude that $\mathbf{H}'_i$ is contained in the center of $\mathbf{H}_i$ in which case the quotient is strongly $\aone$-invariant subsheaf $\mathbf{H}''_i \subset \mathbf{H}''$ by appeal to Lemma~\ref{lem:stronglyaoneinvariantsheavesofgroups}(3).
\end{proof}

\begin{cor}
\label{cor:fiberproducts}
Suppose $\mathbf{G}$ is a strongly $\aone$-invariant sheaf of groups acting on strongly $\aone$-invariant groups $\mathbf{H}$, $\mathbf{H}'$ and $\mathbf{H}''$ and suppose we are given $\mathbf{G}$-equivariant morphisms $\mathbf{H}' \to \mathbf{H}$ and $\mathbf{H}'' \to \mathbf{H}$.  If the $\mathbf{G}$-actions on $\mathbf{H}'$ and $\mathbf{H}''$ are (locally) $\aone$-nilpotent, then the induced $\mathbf{G}$-action on $\mathbf{H}' \times_{\mathbf{H}} \mathbf{H}''$ is also (locally) $\aone$-nilpotent.
\end{cor}

\begin{proof}
As before, the fiber product is strongly $\aone$-invariant by Lemma \ref{lem:fiberproducts}.  Moreover, the fiber product is equipped with a $\mathbf{G}$-action: it is a $\mathbf{G}$-stable subsheaf of the product $\mathbf{H} \times \mathbf{H}'$ equipped with the diagonal $\mathbf{G}$-action.  If $\mathbf{H}'_i$ and $\mathbf{H}''_i$ are (locally finite) $\aone$-$\mathbf{G}$-central series, then $\mathbf{H}'_i \times_{\mathbf{H}} \mathbf{H}''_i$ yields a $\mathbf{G}$-central series for the fiber product by appeal to Proposition \ref{prop:extensionsofactions}(1).
\end{proof}

\subsection{$\aone$-nilpotent spaces and $\aone$-fiber sequences}
\label{ss:aonenilpotentspaces}
We now define various notions of $\aone$-nilpotence for morphisms of spaces by analogy with the situation in classical topology; we defer the explicit construction of examples to Section~\ref{ss:examplesofaonenilpotentspaces}.  Here, we discuss some of the basic formal properties of such morphisms and give some useful criteria for checking nilpotence.  We also analyze the interaction between nilpotence and $\aone$-fiber sequences.

\begin{defn}
\label{defn:aonenilpotentspace}
Suppose $f: (\mathscr{E},e) \to (\mathscr{B},b)$ is a morphism of spaces with $\aone$-homotopy fiber $\mathscr{F}$.
\begin{enumerate}[noitemsep,topsep=1pt]
\item The morphism $f$ will be called a {\em (locally) $\aone$-nilpotent morphism} if $\mathscr{F}$ is $\aone$-connected, and for any choice of base-point $\ast$ in $\mathscr{F}$ the action of $\bpi_1^{\aone}(\mathscr{E},e)$ on $\bpi_i^{\aone}(\mathscr{F},\ast)$ is (locally) $\aone$-nilpotent.
\item A pointed space $(\mathscr{X},x)$ will be called {\em (locally) $\aone$-nilpotent} if the structure morphism $\mathscr{X} \to \ast$ is (locally) $\aone$-nilpotent,
\item {\em weakly $\aone$-nilpotent} if $\mathscr{X}$ is locally $\aone$-nilpotent and $\bpi_1^{\aone}(\mathscr{X})$ is $\aone$-nilpotent; and
\item {\em $\aone$-simple} if $\bpi_1^{\aone}(\mathscr{X},x)$ is abelian and acts trivially on the higher $\aone$-homotopy sheaves of $\mathscr{X}$.
\end{enumerate}
\end{defn}

\begin{prop}
\label{prop:2outof3foraonenilpotence}
Assume $k$ is a perfect field, and suppose $q: \mathscr{E}_2 \to \mathscr{E}_1$ and $p: \mathscr{E}_1 \to \mathscr{E}_0$ are pointed morphisms in $\Spc_k$ such that the $\aone$-homotopy fibers of $p$, $q$ and $pq$ are all connected.  %If any two morphisms in the set $\{ p,q,pq\}$ are $\aone$-nilpotent, then so is the third.
\begin{enumerate}[noitemsep,topsep=1pt]
\item If $p$ and $q$ are (locally) $\aone$-nilpotent, then so is $pq$; if $pq$ and $p$ are (locally) $\aone$-nilpotent, so is $q$.
\item If $pq$ and $q$ are (locally) $\aone$-nilpotent, and $p$ is $1$-connected, then $p$ is (locally) $\aone$-nilpotent as well.
\end{enumerate}
\end{prop}

\begin{proof}
The proof parallels that of Proposition~\ref{prop:2outof3fornilpotentmorphisms}.  In particular, by the ``octahedral axiom" in the $\aone$-homotopy category, there is an $\aone$-fiber sequence of the form
\[
\mathscr{F}_q \longrightarrow \mathscr{F}_{pq} \longrightarrow \mathscr{F}_p,
\]
where each term is the $\aone$-homotopy fiber of the map in the subscript.  All the spaces in this $\aone$-fiber sequence are $\aone$-connected by assumption.  There is an induced action of $\bpi_1^{\aone}(\mathscr{E}_2)$ on all of the homotopy sheaves of each of these spaces, and a $\bpi_1^{\aone}(\mathscr{E}_2)$-equivariant long exact sequence of $\aone$-homotopy sheaves which ends as follows:
\[
\cdots \longrightarrow \bpi_2^{\aone}(\mathscr{F}_p) \longrightarrow \bpi_1^{\aone}(\mathscr{F}_q) \longrightarrow \bpi_1^{\aone}(\mathscr{F}_{pq}) \longrightarrow \bpi_1^{\aone}(\mathscr{F}_p) \longrightarrow 1.
\]

The kernel $\mathbf{K}$ of $\bpi_1^{\aone}(\mathscr{F}_q) \to \bpi_1^{\aone}(\mathscr{F}_{pq})$, which coincides with the image of $\bpi_2^{\aone}(\mathscr{F}_p)$ in $\bpi_1^{\aone}(\mathscr{F}_q)$ by exactness, is strongly $\aone$-invariant by Lemma~\ref{lem:stronglyaoneinvariantsheavesofgroups} and $\bpi_1^{\aone}(\mathscr{E}_2)$-stable by construction.  Moreover, $\mathbf{K}$ is contained in the center of $\bpi_1^{\aone}(\mathscr{F}_q)$.

We may then break the above long exact sequence of sheaves $\bpi_1^{\aone}(\mathscr{E}_2)$-equivariant short exact sequences of strongly $\aone$-invariant sheaves.  In particular, we obtain the two $\bpi_1(\mathscr{E}_2)$-equivariant exact sequences:
\[
\begin{split}
1 \longrightarrow \mathbf{K} \longrightarrow \bpi_1^{\aone}(\mathscr{F}_q) \longrightarrow &\ker(\bpi_1^{\aone}(\mathscr{F}_{pq}) \longrightarrow \bpi_1^{\aone}(\mathscr{F}_p)) \longrightarrow 1 \\
1 \longrightarrow \ker(\bpi_1^{\aone}(\mathscr{F}_{pq}) \longrightarrow \bpi_1^{\aone}(\mathscr{F}_p)) \longrightarrow &\bpi_1^{\aone}(\mathscr{F}_{pq}) \longrightarrow \bpi_1^{\aone}(\mathscr{F}_p) \longrightarrow 1,
\end{split}
\]
which we now analyze.

If $pq$ and $p$ are (locally) $\aone$-nilpotent, then we conclude by appeal to Proposition~\ref{prop:extensionsofactions} that the action of $\bpi_1^{\aone}(\mathscr{E}_2)$ on $\bpi_1(\mathscr{F}_q)$ is (locally) $\aone$-nilpotent.  If $p$ and $q$ are (locally) $\aone$-nilpotent, then observe that since the first exact sequence above is a $\bpi_1^{\aone}(\mathscr{E}_2)$-equviariant {\em central} extension we may conclude that the action of $\bpi_1^{\aone}(\mathscr{E}_2)$ on $\ker(\bpi_1^{\aone}(\mathscr{F}_{pq}) \to \bpi_1^{\aone}(\mathscr{F}_p))$ is (locally) $\aone$-nilpotent, again by appeal to Proposition~\ref{prop:extensionsofactions}.  Another appeal to this result then implies that the $\bpi_1^{\aone}(\mathscr{E}_2)$-action on $\bpi_1^{\aone}(\mathscr{F}_{pq})$ is $\aone$-nilpotent.
The statements about higher homotopy sheaves are established similarly, by repeated appeal to Proposition~\ref{prop:extensionsofactions} (since all sheaves are strictly $\aone$-invariant; this also treats the case where $pq$ and $q$ are (locally) $\aone$-nilpotent and $p$ is $1$-connected.
\end{proof}

If $\mathscr{X}$ is any pointed $\aone$-connected space, then we write $\widetilde{\mathscr{X}}$ for the $\aone$-universal cover of $\mathscr{X}$ in the sense of \cite[\S 7.1]{MField}.  If $\bpi := \bpi_1^{\aone}(\mathscr{X})$, then there is an $\aone$-fiber sequence of the form
\[
\widetilde{\mathscr{X}} \longrightarrow \mathscr{X} \longrightarrow B\bpi
\]
and $\widetilde{\mathscr{X}}$ is $\aone$-$1$-connected.  The action of $\bpi$ on $\widetilde{\mathscr{X}}$ is by ``deck transformations", which do not preserve base-points.  Nevertheless, analysis of this action in conjunction with Proposition~\ref{prop:2outof3foraonenilpotence} can be used to produce a  computationally useful characterization of $\aone$-nilpotence.

\begin{cor}
\label{cor:checkingaonenilpotence}
Assume $(\mathscr{X},x)$ is a pointed $\aone$-connected space with $\bpi := \bpi_1^{\aone}(\mathscr{X},x)$ an $\aone$-nilpotent sheaf of groups.  The following conditions are equivalent:
\begin{enumerate}[noitemsep,topsep=1pt]
\item the space $\mathscr{X}$ is (locally) $\aone$-nilpotent;
\item the action of $\bpi$ on $\bpi_i^{\aone}(\widetilde{\mathscr{X}})$ is (locally) $\aone$-nilpotent for every $i > 1$.
\end{enumerate}
\end{cor}

\begin{proof}
Since $B\bpi$ is $\aone$-nilpotent by assumption, and $\widetilde{\mathscr{X}}$ is $\aone$-simply connected, the statements involving the $\aone$-nilpotence hypothesis follow immediately by applying Proposition~\ref{prop:2outof3foraonenilpotence} to the composite morphism $\mathscr{X} \to B\bpi \to \ast$.  In fact, essentially the same proof works assuming the hypothesis of local $\aone$-nilpotence because $B\bpi$ is $1$-truncated.  Indeed, both the forward and reverse implications follow from this observation because the $\bpi$-action on the higher homotopy sheaves of $\widetilde{\mathscr{X}}$ coincides with the $\bpi$-action on the higher homotopy sheaves of $\mathscr{X}$.
\end{proof}

\begin{rem}
\label{rem:trivialityofactionofpi1}
Just as in classical topology, if the action of $\bpi$ on $\widetilde{\mathscr{X}}$ is null-$\aone$-homotopic, then the induced action on higher $\aone$-homotopy sheaves will be trivial.  Because the homotopy sheaves $\bpi_i^{\aone}(\mathscr{X})$ are all strictly $\aone$-invariant, they are unramified (see Remark~\ref{rem:stronglyaoneinvariantimpliesunramified}), and thus triviality of the necessary actions may be checked on sections over finitely generated extensions of the base field.
\end{rem}

\begin{prop}
\label{prop:basechangeaonenilpotence}
Assume we are given a homotopy Cartesian diagram of pointed $\aone$-connected spaces of the form
\[
\xymatrix{
\mathscr{E}' \ar[r]^{g'}\ar[d]^{f'} & \mathscr{E} \ar[d]^{f} \\
\mathscr{B}' \ar[r]^{g} & \mathscr{B}.
}
\]
If $f$ is (locally) $\aone$-nilpotent, then so is $f'$.
\end{prop}

\begin{proof}
Up to $\aone$-weak equivalence, we may replace the diagram in question by $\Laone f$ and $\Laone g$.  In that case, $\Laone \mathscr{E} \times^h_{\Laone \mathscr{B}} \Laone \mathscr{B}'$ is $\aone$-local since there is a simplicial homotopy fiber sequence of the form:
\[
\Omega \Laone \mathscr{B} \longrightarrow \Laone \mathscr{E} \times^h_{\Laone \mathscr{B}} \Laone \mathscr{B}' \longrightarrow \Laone \mathscr{E} \times \Laone \mathscr{B}'.
\]
Since $\mathscr{E}$ and $\mathscr{B}'$ are $\aone$-connected by assumption, $\Laone \mathscr{E} \times \Laone \mathscr{B}'$ is simplicially connected as well, so appeal to Lemma~\ref{lem:aww2210} allows us to conclude the middle term is $\aone$-local.  It also follows that the induced map $\mathscr{E}' \to \Laone \mathscr{E} \times^h_{\Laone \mathscr{B}} \Laone \mathscr{B}'$ is an $\aone$-weak equivalence since all spaces are connected.  In that case, the result follows from the corresponding result for simplicial presheaves, i.e., by appeal to Proposition \ref{prop:nilpotencebasechange} (replacing appeal to Lemma~\ref{lem:functorialityofnilpotence} by appeal to Lemma~\ref{lem:aonenilpotentactionsarefunctorial}).
\end{proof}

\begin{thm}
\label{thm:aonenilpotenthomotopyfiber}
Assume $k$ is a perfect field.  Suppose $\mathscr{F} \to \mathscr{E} \to \mathscr{B}$ is an $\aone$-fiber sequence of pointed $\aone$-connected spaces.  If $\mathscr{E}$ is (locally) $\aone$-nilpotent, then $\mathscr{F}$ is (locally) $\aone$-nilpotent.
\end{thm}

\begin{proof}
The proof is essentially identical to that of Proposition~\ref{prop:enilpotentimpliesfnilpotent}.  By assumption, there is an $\aone$-homotopy Cartesian square of the form
\[
\xymatrix{
\mathscr{F} \ar[r] \ar[d] & \mathscr{E} \ar[d]^{f} \\
\ast \ar[r] & \mathscr{B}.
}
\]
By Proposition~\ref{prop:basechangeaonenilpotence} to check that $\mathscr{F} \to \ast$ is (locally) $\aone$-nilpotent, it suffices to show that $f$ is (locally) $\aone$-nilpotent, i.e., the action of $\bpi_1^{\aone}(\mathscr{E})$ on $\bpi_i(\mathscr{F})$ is (locally) $\aone$-nilpotent.

There is a $\bpi_1^{\aone}(\mathscr{E})$-equivariant long exact sequence in $\aone$-homotopy sheaves associated with the above fibration, which may be $\bpi_1^{\aone}(\mathscr{E})$-equivariantly broken into short exact sequences of the form:
\[
1 \longrightarrow \mathbf{K}_{i+1} \longrightarrow \bpi_i^{\aone}(\mathscr{F}) \longrightarrow \ker(\bpi_i^{\aone}(\mathscr{E}) \longrightarrow \bpi_i^{\aone}(\mathscr{B})) \longrightarrow 1,
\]
In this case, by exactness of the sequence, both $\mathbf{K}_{i+1}$ and $\ker(\bpi_i^{\aone}(\mathscr{E}) \longrightarrow \bpi_i^{\aone}(\mathscr{B}))$ are strongly $\aone$-invariant by appeal to Lemma~\ref{lem:stronglyaoneinvariantsheavesofgroups}.  Moreover, as a homotopy exact sequence, the above short exact sequence is actually a central extension of strongly $\aone$-invariant sheaves.  Now, by assumption, $\bpi_1^{\aone}(\mathscr{E})$ acts in a (locally) $\aone$-nilpotent fashion on $\bpi_i^{\aone}(\mathscr{E})$ and therefore in a (locally) $\aone$-nilpotent fashion on the kernel in the above short exact sequence by appeal to Lemma~\ref{lem:subsquotientsandcentralextensions}.  Likewise, the $\bpi_1^{\aone}(\mathscr{E})$-action on $\mathbf{K}_{i+1}$ is trivial and it follows again by appeal to Lemma~\ref{lem:subsquotientsandcentralextensions} that the extension is (locally) $\bpi_1^{\aone}(\mathscr{E})$-nilpotent.
\end{proof}

We record the following special case of the above for later use.

\begin{cor}
\label{cor:weaklyaonenilpotentfibersequences}
If $\mathscr{F} \to \mathscr{E} \to \mathscr{B}$ is an $\aone$-fiber sequence of connected spaces, and $\mathscr{E}$ and $\mathscr{B}$ are weakly $\aone$-nilpotent, then $\mathscr{F}$ is weakly $\aone$-nilpotent.
\end{cor}

\begin{prop}
\label{prop:totalspaceaonenilpotent}
Assume $k$ is a perfect field.  Suppose $\mathscr{F} \to \mathscr{E} \to \mathscr{B}$ is an $\aone$-fiber sequence of pointed $\aone$-connected spaces.  If $\mathscr{B}$ is $\aone$-simply connected, and $\mathscr{F}$ is (locally) $\aone$-nilpotent (resp. $\aone$-simple), then $\mathscr{E}$ is (locally) $\aone$-nilpotent (resp. $\aone$-simple) as well.
\end{prop}

\begin{proof}
Under the hypotheses, $\bpi_1^{\aone}(\mathscr{E})$ is a quotient of $\bpi_1^{\aone}(\mathscr{F})$, and since all spaces have a $\bpi_1^{\aone}(\mathscr{E})$-action, they inherit an action of $\bpi_1^{\aone}(\mathscr{F})$ by restriction.  Now, the image of $\bpi_2^{\aone}(\mathscr{B})$ in $\bpi_1^{\aone}(\mathscr{F})$ is a central subsheaf.  Moreover, by exactness of the long exact sequence, the image of this subsheaf is again strongly $\aone$-invariant as it is identified with the kernel of the map $\bpi_1^{\aone}(\mathscr{F}) \to \mathscr{E}$.  Then, since $\bpi_1^{\aone}(\mathscr{F})$ is (locally) $\aone$-nilpotent, it follows that (locally) $\bpi_1^{\aone}(\mathscr{E})$ is also $\aone$-nilpotent by appeal to Lemma~\ref{lem:subsquotientsandcentralextensions}(2).

Next, consider the portion of the long exact sequence in $\aone$-homotopy sheaves:
\[
\cdots \longrightarrow \bpi_n^{\aone}(\mathscr{F}) \longrightarrow \bpi_n^{\aone}(\mathscr{E}) \longrightarrow \bpi_n^{\aone}(\mathscr{B}) \longrightarrow \cdots.
\]
Since $k$ is perfect and $n \geq 2$, the homotopy sheaves in question are strictly $\aone$-invariant (so all images and cokernels are strictly $\aone$-invariant as well by appeal to Theorem~\ref{thm:stableconnectivity}).  Moreover, by assumption $\bpi_n^{\aone}(\mathscr{B})$ carries a trivial action of $\bpi_1^{\aone}(\mathscr{B})$, so the quotient $\bpi_n^{\aone}(\mathscr{E})/\im(\bpi_n^{\aone}(\mathscr{F}))$ carries a trivial action of $\bpi_1^{\aone}(\mathscr{E})$, which is in particular (locally) $\aone$-nilpotent.

Thus, to conclude, it suffices by Proposition \ref{prop:extensionsofactions} to show that $\im(\bpi_n^{\aone}(\mathscr{F}))$ carries a (locally) $\aone$-nilpotent action of $\bpi_n^{\aone}(\mathscr{E})$.  However, $\im(\bpi_n^{\aone}(\mathscr{F}))$ carries a (locally) $\aone$-nilpotent action of $\bpi_1^{\aone}(\mathscr{F})$ by assumption.  If we pick a (locally finite) $\aone$-$\bpi_1^{\aone}(\mathscr{F})$-central series for $\bpi_n^{\aone}(\mathscr{F})$, it induces a (locally finite) $\aone$-$\bpi_1^{\aone}(\mathscr{F})$-central series $\im(\bpi_n^{\aone}(\mathscr{F}))$; this is also a (locally finite) $\aone$-$\bpi_1^{\aone}(\mathscr{E})$-central series by restriction, which is what we wanted to show.

The statement about $\aone$-simplicity is established similarly.  If $\bpi_1^{\aone}(\mathscr{F})$ is abelian, then $\bpi_1^{\aone}(\mathscr{E})$ is necessarily abelian as well.  If $\bpi_1^{\aone}(\mathscr{F})$ acts trivially on the higher $\aone$-homotopy sheaves of $\mathscr{F}$, then the argument of the preceding paragraph also shows that $\bpi_1^{\aone}(\mathscr{E})$ acts trivially on the higher $\aone$-homotopy sheaves of $\mathscr{E}$.
\end{proof}

\subsubsection*{$\aone$-nilpotence and Moore--Postnikov factorizations}
A pointed morphism $f:(\mathscr{E},e)\to (\mathscr{B},b)$ is an $\aone$-principal fibration if there exists a pointed space $(\mathscr{C},c)$ and a pointed ``classifying" map $w: \mathscr{B} \to \mathscr{C}$ such that $\mathscr{E}$ is the homotopy pullback of the path-loop fibration along $w$.  We may consider factorizations of a pointed morphism as a tower of principal fibrations in a sense we now make precise; we largely follow the discussion of \cite[II.2]{HMR}.

\begin{defn}
\label{defn:factorizationasatower}
Suppose $f: (\mathscr{E},e) \to (\mathscr{B},b)$ is a morphism of pointed $\aone$-connected spaces.   A {\em factorization of $f$ as a tower of $\aone$-fibrations} consists of a sequence of pointed spaces $\tau_{\leq i}f$, $i \geq 0$, morphisms $p_i: \tau_{\leq i}f \to \tau_{\leq i-1}f$, $\mathscr{E} \to \tau_{\leq i}f$ and $\tau_{\leq i}f \to \mathscr{B}$ fitting into a commutative diagram of the form:
\[
\xymatrix{
&& \ar[dl]\mathscr{E}\ar[d]\ar[dr] && \\
\cdots \ar[r]& \tau_{\leq i+1}f \ar[r]^{p_{i+1}}\ar[dr]& \tau_{\leq _i}f\ar[r]^{p_i} \ar[d]& \tau_{\leq i-1} f \ar[dl]\ar[r]^{p_{i-1}}& \cdots \\
&& \mathscr{B} &&
}
\]
having the following properties:
\begin{enumerate}[noitemsep,topsep=1pt]
\item $\tau_{\leq 0}f = \mathscr{B}$ (thus, the composite morphisms $\mathscr{E} \to \tau_{\leq i}f \to \mathscr{B}$ coincide with $f$);
\item the induced morphism ${\mathscr E} \to \operatorname{holim}_i \tau_{\leq i}f$ is an $\aone$-weak equivalence.
\end{enumerate}
If, furthermore,
\begin{enumerate}[noitemsep,topsep=1pt]
\item[3.] for each integer $i \geq 0$, the morphism $p_i$ is an $\aone$-principal fibration in the sense mentioned above, then
\end{enumerate}
we will say that the given factorization of $f$ is a {\em factorization of $f$ as a tower of $\aone$-principal fibrations}.
\end{defn}

\begin{ex}
\label{ex:aonemoorepostnikovfactorization}
Suppose $f: (\mathscr{E},e) \to (\mathscr{B},b)$ is a morphism of pointed $\aone$-connected spaces.  If the morphism $\bpi_1^{\aone}(\mathscr{E}) \to \bpi_1^{\aone}(\mathscr{B})$ is surjective, then $f$ admits a standard factorization as a tower of $\aone$-fibrations:  the $\aone$-Moore--Postnikov factorization.  Indeed, this is a factorization of $f$ as a tower of $\aone$-fibrations where, in addition to the statements in Definition~\ref{defn:factorizationasatower}, the following statements hold:
\begin{enumerate}[noitemsep,topsep=1pt]
\item the morphisms $\mathscr{E} \to \tau_{\leq i}f$ induce epimorphisms on $\aone$-homotopy sheaves in degrees $\leq i+1$ that are furthermore isomorphisms in degrees $ \leq i$;
\item the map on $\aone$-homotopy sheaves induced by the morphisms $\tau_{\leq i}f \to \mathscr{B}$ induce are isomorphisms in degrees $>i+1$, and a monomorphism in degree $i+1$.
\end{enumerate}
\end{ex}

\begin{defn}
\label{defn:principalrefinement}
Suppose $f: (\mathscr{E},e) \to (\mathscr{B},b)$ is a morphism of pointed $\aone$-connected spaces equipped with a factorization as a tower of $\aone$-fibrations as in \textup{Definition~\ref{defn:factorizationasatower}}.  We will say that this factorization admits an {\em $\aone$-principal refinement} if for each $n \geq 1$, there exists an integer $c$ (which will depend on $n$) such that
\begin{enumerate}[noitemsep,topsep=1pt]
\item the morphism $p_n: \tau_{\leq n}f \to \tau_{\leq n-1}f$ factors as follows:
    \[
    \tau_{\leq n} f = \tau_{\leq n,c} f \stackrel{p_{n,c}}{\longrightarrow} \tau_{\leq n,c-1} f \longrightarrow \cdots \longrightarrow \tau_{\leq n,1} f \stackrel{p_{n,0}}{\longrightarrow} \tau_{\leq n,0}f = \tau_{\leq n-1} f
    \]
    and,
\item the new tower $\tau_{\leq n,j}f$ is a factorization of $f$ by $\aone$-principal fibrations.
\end{enumerate}
\end{defn}

\begin{rem}
The statement that the $\aone$-Moore--Postnikov factorization \ref{ex:aonemoorepostnikovfactorization} admits an $\aone$-principal refinement amounts to asking that for each $n,i$ as in Definition~\ref{defn:principalrefinement}, the morphism $p_{n,i}$ is an $\aone$-principal fibration in the sense that there exist a strictly $\aone$-invariant sheaf $\mathbf{G}_{n,i}$ and an $\aone$-fiber sequence of the form
\[
\tau_{\leq n,i}f \stackrel{p_{n,i}}{\longrightarrow} \tau_{\leq n,i-1}f \longrightarrow K(\mathbf{G}_{n,i},n+1).
\]
\end{rem}

The following result is the $\aone$-homotopic analog of one of the classical characterizations of nilpotent spaces.

\begin{thm}
\label{thm:moorepostnikovcharacterization}
Suppose $f: \mathscr{E} \to \mathscr{B}$ is a morphism of pointed $\aone$-connected spaces such that the induced morphism $\bpi_1^{\aone}(\mathscr{E}) \to \bpi_1^{\aone}(\mathscr{B})$ is an epimorphism.  The morphism $f$ is $\aone$-nilpotent if and only if the $\aone$-Moore--Postnikov tower of $f$ admits an $\aone$-principal refinement.
\end{thm}

We give the proof of the forward implication below.  We will defer the proof of the reverse implication to Corollary~\ref{cor:MoorePostnikovfactorizationconclusion}.  Note: Theorem~\ref{thm:gmloopsaonenilpotent} below depends on this characterization of $\aone$-nilpotence, but is placed in this section (as opposed to after Corollary~\ref{cor:MoorePostnikovfactorizationconclusion}) for stylistic reasons.

\begin{proof}
See \cite[Theorem II.2.14]{HMR} for the corresponding statement in classical topology, which we essentially follow.  Observe that the hypotheses on $f$ imply that the $\aone$-homotopy fiber $\mathscr{F}$ of $f$ is $\aone$-connected.  We show that if the $\aone$-Moore--Postnikov filtration admits an $\aone$-principal refinement, then $f$ is $\aone$-nilpotent.

We want to show that the action of $\bpi_1^{\aone}(\mathscr{E})$ on all the higher homotopy sheaves of $\mathscr{F}$ is $\aone$-nilpotent.  Assuming $p_n$ admits a factorization as in the statement, there is an induced action of $\bpi_1^{\aone}(\mathscr{E})$ on $\bpi_i^{\aone}(\tau_{\leq n}f)$ for every $n$ and therefore an induced action on $\tau_{\leq n,i}f$ for every $i$ making
\[
\tau_{\leq n,i}f \stackrel{p_{n,i}}{\longrightarrow} \tau_{\leq n,i-1}f \longrightarrow K(\mathbf{G}_{n,i},n+1)
\]
into a fiber sequence where all the morphisms are $\bpi_1^{\aone}(\mathscr{E})$-equivariant.  Repeated application of Proposition~\ref{prop:extensionsofactions} then yields the result.
\end{proof}

\subsubsection*{$\aone$-nilpotence and mapping spaces}
\begin{defn}
\label{defn:cohomologicalfinitenessassumptions}
Assume $k$ is a field.  We will say that $\mathscr{X} \in \Spc_k$ has {\em $\aone$-cohomological dimension $\leq d$} if for every strictly $\aone$-invariant sheaf $\mathbf{A}$, and every integer $i > d$, the group $H^i_{\Nis}(\mathscr{X},\mathbf{A}) = 0$, and {\em $\aone$-cohomological dimension $d$} if $\mathscr{X}$ has $\aone$-cohomological dimension $\leq d$ but not $\leq d-1$.  We will say that $\mathscr{X}$ has {\em finite $\aone$-cohomological dimension}, if it has $\aone$-cohomological dimension $\leq d$ for some integer $d$, and {\em strongly finite $\aone$-cohomological dimension} if for every extension $L/k$, the base change $\mathscr{X}_L \in \Spc_L$ has $\aone$-cohomological dimension $\leq d$ for some integer $d$, independent of $L$.
\end{defn}

\begin{ex}
\label{ex:stronglyfinitecohomologicaldimension}
If $X$ is a smooth $k$-scheme of dimension $d$, then $X$ has strongly finite $\aone$-cohomological dimension $\leq d$.  From this and the suspension isomorphism in Nisnevich cohomology, one deduces that $\Sigma^j X_+$ also has strongly finite $\aone$-cohomological dimension.  More generally, if $\mathscr{X}$ is any pointed space for which there exist integers $i,j$ such that $\Sigma^i \mathscr{X}$ has the $\aone$-homotopy type of $\Sigma^j X_+$ for some smooth scheme $X_+$, then $\mathscr{X}$ has strongly finite $\aone$-cohomological dimension.  In particular, this last statement holds for any motivic sphere $S^i \sma \gm{\sma j}$ since it holds for $\gm{\sma j}$ (use ${\mathbb A}^j \setminus 0$).
\end{ex}

\begin{thm}
\label{thm:mappingspaces}
Assume $k$ is a perfect field.  If $\mathscr{X}$ is any pointed $\aone$-nilpotent space, and $\mathscr{Y}$ is a pointed space having strongly finite $\aone$-cohomological dimension, then for any fixed pointed map $f: \mathscr{Y} \to \mathscr{X}$, the $\aone$-connected component of the mapping space $\Map(\mathscr{Y},\mathscr{X})_f$ containing $f$ is again $\aone$-nilpotent.
\end{thm}

\begin{proof}
First, suppose $\mathscr{X} = K(\mathbf{A},n)$ for some integer $n \geq 1$ and a strictly $\aone$-invariant sheaf of groups $\mathbf{A}$.  In that case, consider the space
\[
\Map(\mathscr{Y},K(\mathbf{A},n))_f,
\]
for some map $f \in H^n(\mathscr{Y},\mathbf{A})$.  Observe that $\Map(\mathscr{Y},K(\mathbf{A},n))$ is itself a (possibly $\aone$-disconnected) $\aone$-$h$-space.  Indeed, the space $K(\mathbf{A},n)$ is $\aone$-local by assumption, so the mapping space $\Map(\mathscr{Y},K(\mathbf{A},n))$ is $\aone$-local as well.  Now, any $\aone$-connected component of an $\aone$-local space is again $\aone$-local, so $\Map(\mathscr{Y},K(\mathbf{A},n))_f$ is $\aone$-local also.  The identity component  $\Map(\mathscr{Y},K(\mathbf{A},n))_0$ is itself an $\aone$-$h$-space and therefore $\aone$-simple.  On the other hand, multiplication by $f$ induces an $\aone$-weak equivalence between $\Map(\mathscr{Y},K(\mathbf{A},n))_0$ and $\Map(\mathscr{Y},K(\mathbf{A},n))_f$, so we conclude the latter is $\aone$-simple as well.  Explicitly, $\bpi_i^{\aone}(\Map(\mathscr{Y},K(\mathbf{A},n))_f)$ coincides with $\bpi_i((\Map(\mathscr{Y},K(\mathbf{A},n))_f))$; for any $i > 0$, this latter sheaf, which is necessarily strictly $\aone$-invariant, has sections $\tilde{H}^{n-i}_{\Nis}(\mathscr{Y}_L,\mathbf{A}_L)$ over a finitely generated extension $L/k$.  

%The assumption that $\mathscr{Y}$ has strongly finite $\aone$-cohomological dimension then implies that $\Map(\mathscr{Y},K(\mathbf{A},n))_f$ is $\aone$-weakly equivalent to a finite product of Eilenberg--Mac Lane spaces based on these strictly $\aone$-invariant sheaves.  Indeed, consider the product $\prod_{i} K(\tilde{\mathbf{H}}^{n-i}(\mathscr{Y},\mathbf{A}),i)$.  In particular, such spaces are $\aone$-simple as they are $\aone$-$h$-spaces.

Next, assume that the $\mathscr{X}$ is $\aone$-$n$-truncated and $\aone$-nilpotent for some integer $n$, i.e., there exists an integer $n$ such that $\bpi_i^{\aone}(\mathscr{X},x)$ vanishes for $i > n$.  In that case, by appeal to Theorem~\ref{thm:moorepostnikovcharacterization} we know that the $\aone$--Postnikov factorization of the structure morphism admits an $\aone$-principal refinement, and furthermore that the resulting factorization is a finite tower; in this case, we will say that $\mathscr{X}$ has a finite $\aone$-principal series.  We then proceed by induction.  Since applying mapping spaces preserves $\aone$-fiber sequences, we reduce inductively to the following situation: there are $\aone$-fiber sequences of the form
\[
\Map(\mathscr{Y},\mathscr{X})_f \longrightarrow \Map(\mathscr{Y},\mathscr{X}')_f \longrightarrow \Map(\mathscr{Y},K(\mathbf{A},n))_f.
\]
where $\mathscr{X}$ admits a finite $\aone$-principal series and $\mathscr{X}'$ admits a finite $\aone$-principal series of shorter length.  By the discussion of the preceding paragraph $\Map(\mathscr{Y},K(\mathbf{A},n))_f$ is $\aone$-simple, and one assumes that $\Map(\mathscr{Y},\mathscr{X}')_f$ is $\aone$-nilpotent by induction.  In that case, $\Map(\mathscr{Y},\mathscr{X})_f$ is necessarily also $\aone$-nilpotent by appeal to Corollary~\ref{cor:weaklyaonenilpotentfibersequences}.  The discussion of the preceding paragraph also implies, by appeal to the long exact sequence in $\aone$-homotopy sheaves of a fibration, that $\bpi_i^{\aone}(\Map(\mathscr{Y},\mathscr{X})_f)$ only depends on $\bpi_i^{\aone}(\mathscr{X})$ for $n \leq d+i+1$, if $\mathscr{Y}$ has strongly finite $\aone$-cohomological dimension $\leq d$.

Now, assume that $\mathscr{X}$ is a general $\aone$-nilpotent space.  To check that $\Map(\mathscr{Y},\mathscr{X})_f$ is $\aone$-nilpotent, we need to know that $\bpi_1^{\aone}(\Map(\mathscr{Y},\mathscr{X})_f)$ is $\aone$-nilpotent and that this sheaf of groups acts $\aone$-nilpotently on $\bpi_i^{\aone}(\Map(\mathscr{Y},\mathscr{X})_f)$ for all $i > 0$.  Thus, $\Map(\mathscr{Y},\mathscr{X})_f$ is $\aone$-nilpotent if and only if all its finite Postnikov sections are $\aone$-nilpotent.  We know that $\mathscr{X}$ is the homotopy limit of its finite $\aone$-Postnikov sections.  The map $\mathscr{X} \to \tau_{\leq n} \mathscr{X}$ induces a map
\[
\tau_{\leq m}(\Map(\mathscr{Y},\mathscr{X})_f) \longrightarrow \tau_{\leq m}(\Map(\mathscr{Y},\tau_{\leq n}\mathscr{X})).
\]
Because $\mathscr{Y}$ has strongly finite $\aone$-cohomological dimension, it follows that for any $m$, the map of the preceding display induces an $\aone$-weak equivalence as $n \to \infty$ by the final comment of the preceding paragraph.  Indeed, the map on homotopy sheaves is an isomorphism for $n$ sufficiently large depending on $m$ and $d$.
\end{proof}

The above results have consequences for the contraction construction from Definition~\ref{defn:contraction}.  For example, we may analyze $\aone$-nilpotence of iterated $\gm{}$-loop spaces of $\aone$-nilpotent spaces.  As a corollary of Theorem~\ref{thm:mappingspaces}, we deduce the following result, which shows that $\gm{}$-loop spaces of $\aone$-nilpotent spaces are $\aone$-nilpotent.  Unlike the situation in classical homotopy theory, $\gm{}$-loop spaces need not automatically have an $h$-space structure since $\gm{}$ has no co-group structure in $\ho{k}$.

\begin{thm}
\label{thm:gmloopsaonenilpotent}
Assume $(\mathscr{X},x)$ is a pointed $\aone$-nilpotent space, then $\Omega_{\gm{\sma n}}\mathscr{X}$ is $\aone$-nilpotent for every $n \geq 0$.  In particular, if $\mathbf{G}$ is an $\aone$-nilpotent sheaf of groups, then $\mathbf{G}_{-n}$ is again an $\aone$-nilpotent sheaf of groups for every $n \geq 0$.
\end{thm}

\begin{proof}
By appealing to Theorem~\ref{thm:contractionsandgmloops}(1) we conclude that $\Omega_{\gm{\sma n}} \mathscr{X}$ is always $\aone$-connected.  Furthermore, in Example~\ref{ex:stronglyfinitecohomologicaldimension} we observed that $\gm{\sma n}$ has strongly finite $\aone$-cohomological dimension.  Granted those two facts, the first statement is a consequence of Theorem~\ref{thm:mappingspaces}: take $\mathscr{Y} = \gm{\sma n}$.  The second statement follows from the first by repeated application of Theorem~\ref{thm:contractionsandgmloops}: $\bpi_1^{\aone}(\Omega_{\gm{\sma n}}B\mathbf{G}) = \mathbf{G}_{-n}$.
\end{proof}

\subsection{Explicit examples}
\label{ss:examplesofaonenilpotentspaces}
In this section, we collect a number of examples of spaces as in Definition~\ref{defn:aonenilpotentspace}.  After reviewing some basic examples (see Example~\ref{ex:basicexamples}), we show that ${\mathbb P}^{2n+1}_k$ and $BGL_{2n+1}$ are $\aone$-simple (see Proposition~\ref{prop:p2n+1aonesimple}, and Theorem~\ref{thm:bglnoddaonesimple}).  Then, we show that generalized flag varieties are typically $\aone$-nilpotent but not $\aone$-simple (see Theorem~\ref{thm:flagmanifoldaonenilpotent} and Remark~\ref{rem:nonabelianaonenilpotent}).  Finally, we observe that ${\mathbb P}^{2n}_k$ and $BGL_{2n}$ are not locally $\aone$-nilpotent over general fields $k$, but are weakly $\aone$-nilpotent if $k$ is a field that is not formally real (see Proposition~\ref{prop:p2npointwiseaonenilpotent}, Theorem~\ref{thm:bglevenlocallyaonenilpotent} and the discussion of Remarks~\ref{rem:p2nnotaonenilpotent} and \ref{rem:bgl2nnotlocallyaonenilpotent}).

\subsubsection*{$\aone$-simple spaces}
If $(\mathscr{X},e)$ is an $h$-space in $\sPre(\Sm_k)$, then $\Laone \mathscr{X}$ is also an $h$-space since $\Laone$ preserves finite products.  If $(\mathscr{X},e)$ is an $\aone$-connected $h$-space, then it follows from the usual Eckmann--Hilton argument that $\bpi_1^{\aone}(\mathscr{X},e)$ is abelian and that this sheaf of groups furthermore acts trivially on the higher $\aone$-homotopy sheaves since this is true stalkwise.  Thus, if $\mathscr{X}$ is an $\aone$-connected $h$-space, it follows that $\mathscr{X}$ is $\aone$-simple.

\begin{ex}
\label{ex:basicexamples}
The following are well-known examples of $\aone$-simple spaces (see Definition~\ref{defn:aonenilpotentspace}).
\begin{enumerate}[noitemsep,topsep=1pt]
\item split, simply connected, semi-simple algebraic groups (the hypotheses guarantee that the resulting space is $\aone$-connected and the result is evident for any $\aone$-connected sheaf of groups);
\item the space $B_{\Nis}A$ for any abelian sheaf of groups $A$;
\item the stable group $BGL$ (this is an $\aone$-connected $h$-space);
\item the motivic Eilenberg--Mac Lane space $K(\Z(n),2n)$ or $K(\Z/m(n),2n)$ for any integers $m, n > 0$ (the hypotheses guarantee that the resulting space is $\aone$-connected; see Section~\ref{ss:rationalmotivicspheres} for further discussion of this example).
\end{enumerate}
\end{ex}

\begin{rem}
In classical algebraic topology, the odd-dimensional real projective spaces $\mathbb{RP}^{2n+1}$ are all {\em simple} spaces (i.e., have abelian fundamental group and the action on higher homotopy groups is trivial), not just nilpotent.  In contrast, $\mathbb{P}^1_k$ is not $\aone$-simple because $\bpi_1^{\aone}(\pone_k)$ is not even a sheaf of abelian groups.  Nevertheless, the $\aone$-fundamental sheaves of groups of ${\mathbb P}^n_k$ for $n \geq 2$ are isomorphic to $\gm{}$, which is abelian.
\end{rem}

The $\aone$-fiber sequences produced by the following result will be essential in our construction of (locally) $\aone$-nilpotent spaces to follow.

\begin{thm}
\label{thm:buildingfibersequences}
Assume $k$ is a field and $G$ is a split, simply connected, semi-simple $k$-group scheme.  If $H \subset G$ is a closed sub-group scheme, such that the $H$-torsor $G \to G/H$ is Nisnevich locally trivial, then there is an $\aone$-fiber sequence of the form
\[
G/H \longrightarrow BH \longrightarrow BG.
\]
\end{thm}

\begin{proof}
Under these hypotheses, there is a simplicial fiber sequence of the form
\[
G/H \longrightarrow B_{\Nis}H \longrightarrow B_{\Nis}G.
\]
by \cite[Lemma 2.4.1]{AHWII}.  The space $B_{\Nis}G$ is Nisnevich local by definition and it satisfies Nisnevich excision by \cite[Theorem 3.2.5]{AHW}.  Furthermore, $\pi_0(B_{\Nis}G)$ is $\aone$-invariant on affines by \cite[Theorem 2.4]{AHWIII}.  Therefore, by appeal to \cite[Theorem 2.2.5]{AHWII}, we conclude that the above simplicial fiber sequence is an $\aone$-fiber sequence.
\end{proof}

\begin{prop}
\label{prop:p2n+1aonesimple}
If $k$ is a field, then for every integer $n > 0$, the space ${\mathbb P}^{2n+1}_k$ is $\aone$-simple.
\end{prop}

\begin{proof}
We suppress $k$ from the notation.  We proceed by analogy to the argument establishing nilpotence of odd-dimensional real projective spaces in classical topology.  In other words, we will establish the equivalent conditions of Corollary \ref{cor:checkingaonenilpotence}.  To this end, let us recall that there is an $\aone$-fiber sequence of the form
\[
{\mathbb A}^{2n+2}\setminus 0 \longrightarrow {\mathbb P}^{2n+1} \longrightarrow B\gm{}
\]
classifying the tautological $\gm{}$-torsor over ${\mathbb P}^{2n+1}$ \cite[Lemma 7.5]{MField}.  In particular the first map is the quotient of the standard scaling action of $\gm{}$ on ${\mathbb A}^{2n+2} \setminus 0$.

For any $n > 0$, ${\mathbb A}^{2n+2} \setminus 0$ is $\aone$-$1$-connected, and thus is a model for the $\aone$-universal covering space of ${\mathbb P}^{2n+1}$ \cite[Theorem 7.8]{MField}.  Thus, it suffices to show that the action of $\gm{}$ on the higher $\aone$-homotopy sheaves of ${\mathbb A}^{2n+2} \setminus 0$ is $\aone$-nilpotent.

Following the observation of Remark~\ref{rem:trivialityofactionofpi1}, we will show that the scaling $\gm{}$-action on ${\mathbb A}^{2n+2} \setminus 0$ is null $\aone$-homotopic; we prove this in a fashion analogous to \cite[Lemma 4.8]{AsokFaselThreefolds}.  Since the homotopy sheaves of ${\mathbb A}^{2n+2} \setminus 0$ are unramified, to show that the action of $\gm{}$ is trivial, it suffices to show that, for every finitely generated separable extension $L$ of the base field $k$ the action of $\gm{}(L)$ is null $\aone$-homotopic.

Thus, it suffices to show that for every unit $u \in L^{\times}$, the endomorphism of ${\mathbb A}^{2n+2} \setminus 0$ induced by scaling via $u$ has trivial motivic Brouwer degree.  Indeed, the motivic Brouwer degree of the map $(x_1,\ldots,x_{2n+2}) \to (ux_1,ux_2,\ldots,ux_{2n+2})$ can be computed via the procedure described in \cite[Remark 2.6]{Fasel}.  Following this procedure, one sees that the Brouwer degree of the above map is $\langle u \rangle^{2n+2}$, which, as a square class, is trivial in $GW(L)$.
\end{proof}

\begin{rem}
One can also establish the preceding result as follows.  One shows that there is an $\aone$-fiber sequence of the form
\[
\pone_k \longrightarrow {\mathbb P}^{2n+1}_k \longrightarrow \mathrm{HP}^n_k
\]
where $\mathrm{HP}^n$ is the Panin--Walter model for quaternionic projective space \cite[\S 3]{PaninWalterPontryaginClasses}.  By construction, $\mathrm{HP}^n = Sp_{2n+2}/(Sp_{2n} \times Sp_2)$, and the map ${\mathbb P}^{2n+1} \to \mathrm{HP}^n$ is the map $Sp_{2n+2}/(Sp_{2n} \times \gm{}) \to Sp_{2n+2}/(Sp_{2n} \times Sp_2)$ induced by the inclusion of $\gm{} \hookrightarrow Sp_2$ as a maximal torus.  Here $Sp_{2n+2}/(Sp_{2n} \times \gm{})$ is a model for the standard Jouanolou device over ${\mathbb P}^{2n+1}$ for every $n \geq 0$.  One may show that $\mathrm{HP}^n$ is $\aone$-$1$-connected.  Then, since $\pone$ is $\aone$-nilpotent by Theorem \ref{thm:flagmanifoldaonenilpotent} below one may conclude by appealing to Proposition \ref{prop:totalspaceaonenilpotent}.
\end{rem}

The next result is an analog in $\aone$-algebraic topology of the classical fact that $BO(2n+1)$ is a simple space.

\begin{thm}
\label{thm:bglnoddaonesimple}
If $k$ is a field, then for every integer $n \geq 0$, the $k$-space $BGL_{2n+1}$ is $\aone$-simple.
\end{thm}

\begin{proof}
Again, we suppress $k$ from the notation.  For $n = 0$, the assertion is that $B\gm{}$ is $\aone$-simple, which follows from the fact that $B\gm{}$ admits an $\aone$-$h$-space structure.  Therefore, we assume $n > 0$.

We begin by considering the homomorphism $GL_m \to SL_{m+1}$ defined at the level of points by sending an invertible $m \times m$-matrix $X$ to the block matrix $(X,\det X^{-1})$.  This homomorphism identifies $GL_m$ as the Levi factor of a parabolic subgroup in $SL_{m+1}$.  The quotient morphism $SL_{m+1} \to SL_{m+1}/GL_{m}$ is Zariski and hence Nisnevich locally trivial, so Theorem~\ref{thm:buildingfibersequences} yields an $\aone$-fiber sequence of the form
\[
SL_{m+1}/GL_m \longrightarrow BGL_m \longrightarrow BSL_{m+1}.
\]
The fiber here is identified as the standard Jouanolou device over ${\mathbb P}^{m}$ (i.e., the complement of the incidence divisor in the product of ${\mathbb P}^m \times {\mathbb P}^m$, where the second copy of projective space is viewed as the dual of the first).

If $m = 2n+1$ is odd, then it follows from Proposition~\ref{prop:p2n+1aonesimple} and the identifications described in the preceding paragraph that $SL_{2n+2}/GL_{2n+1}$ is $\aone$-simple.  Since $BSL_{m+1}$ is always $\aone$-$1$-connected, the result follows from Proposition~\ref{prop:totalspaceaonenilpotent}.
\end{proof}

\subsubsection*{Examples of $\aone$-nilpotent spaces}
Examples of $\aone$-nilpotent spaces may be built from those in Example~\ref{ex:basicexamples} by taking homotopy fibers of maps.

\begin{ex}
\label{ex:gems}
Any finite product of ($\aone$-connected) spaces of the form $K(\mathbf{A},i)$ is also an $\aone$-simple space and thus $\aone$-nilpotent.  By an $\aone$-polyGEM, we will mean a space contained in the smallest class of spaces that i) contains finite products of Eilenberg--Mac Lane spaces, and ii) is closed under the operation of taking homotopy fibers of morphisms of $\aone$-polyGEMs.  A straightforward induction argument appealing to Corollary~\ref{cor:weaklyaonenilpotentfibersequences} guarantees that such spaces are automatically $\aone$-nilpotent.  For example, any finite Postnikov truncation of an $\aone$-nilpotent space is an $\aone$-polyGEM.
\end{ex}

More interesting examples of $\aone$-nilpotent spaces are constructed in the following result.

\begin{thm}
\label{thm:flagmanifoldaonenilpotent}
Suppose $k$ is a perfect field, and assume $G$ is a split, semi-simple, simply connected $k$-group scheme, and let $B$ be a split Borel subgroup.  For every integer $n \geq 0$, the space $\Omega_{\gm{\sma n}}G/B$ is $\aone$-nilpotent.
\end{thm}

\begin{proof}
We first treat the case $n = 0$.  Fix a split maximal torus $T \subset G$ and observe that there are induced maps
\[
G/T \longrightarrow G/B.
\]
The above morphism is Zariski locally trivial with affine space fibers and thus an $\aone$-weak equivalence.  Moreover the $T$-torsor $G \to G/T$ is Zariski and hence Nisnevich locally trivial (and thus Nisnevich locally split).  Since the map $G/T \to G/B$ is an $\aone$-weak equivalence, Theorem~\ref{thm:buildingfibersequences} allows us to conclude that there is an $\aone$-fiber sequence of the form
\[
G/B \longrightarrow BT \longrightarrow BG.
\]
Since $BT$ is $\aone$-nilpotent ($\aone$-simple in fact) and all spaces are $\aone$-connected, the result follows from Theorem \ref{thm:aonenilpotenthomotopyfiber}.  For $n \geq 0$, we conclude by appeal to Theorem \ref{thm:gmloopsaonenilpotent}.
\end{proof}

\begin{rem}
\label{rem:nonabelianaonenilpotent}
In the simplest case $G = SL_2$ the above result shows that $\pone$ is $\aone$-nilpotent.  Indeed, in that case $\bpi_1^{\aone}(\pone)$ is non-abelian by Example~\ref{ex:pi1poneaonenilpotent}.  More generally, the fiber sequence in the statement together with the facts that $\bpi_1^{\aone}(BT) = T$ and $BT$ is $\aone$-$1$-truncated yield a central extension of the form
\[
1 \longrightarrow \bpi_2^{\aone}(BG) \longrightarrow \bpi_1^{\aone}(G/B) \longrightarrow T \longrightarrow 1.
\]
For simplicity, assume $G$ is a split, simple, simply-connected $k$-group scheme.  In that case, the description of $\bpi_2^{\aone}(BG)$ depends on the Dynkin classification of $G$.  For example, if $G$ is not of type $C_n$, then $\bpi_2^{\aone}(BG)$ is $\K^M_2$, while if $G$ is of type $C_n$, then $\bpi_2^{\aone}(BG)$ is $\K^{MW}_2$.  For $G = SL_n$, this result appears in \cite[Theorem 7.20]{MField}.  For $k$ infinite and general $G$, this appears in \cite{VoelkelWendt}.  A uniform proof, without assumption on cardinality of $k$ appears in \cite[Theorem 1]{MorelSawant}.  Moreover, it is known that the resulting extension is non-trivial in general and $\bpi_1^{\aone}(G/B)$ is a non-abelian sheaf of groups.
\end{rem}

\subsubsection*{Examples of locally $\aone$-nilpotent spaces}
In contrast to the situation for $BGL_{2n+1}$ discussed in Theorem~\ref{thm:bglnoddaonesimple}, one knows that $\bpi_1^{\aone}(BGL_{2n}) = \gm{}$ can act non-trivially on $\bpi_i^{\aone}(BGL_{2n})$ (e.g., for $n = 1$, see \cite[Corollary 4.9]{AsokFaselThreefolds}).  Thus, $BGL_{2n}$ is not $\aone$-simple.  Nevertheless, we will see that under appropriate assumptions on the base field, $BGL_{2n}$ is locally $\aone$-nilpotent.  To show this, we first establish a local $\aone$-nilpotence result for even-dimensional projective spaces.  As discussed in Remark~\ref{rem:p2nnotaonenilpotent}, ${\mathbb P}^{2n}_k$ is {\em not} $\aone$-nilpotent for general fields $k$.  We will furthermore see that $BGL_{2n}$ is also {\em not} $\aone$-nilpotent for general fields $k$.

\begin{prop}
\label{prop:p2npointwiseaonenilpotent}
Suppose $n \geq 1$ is an integer.  If $k$ is not formally real, then ${\mathbb P}^{2n}_k$ is weakly $\aone$-nilpotent.
\end{prop}

\begin{proof}
The $\aone$-fiber sequence
\[
{\mathbb A}^{2n+1} \setminus 0 \longrightarrow {\mathbb P}^{2n} \longrightarrow B\gm{}
\]
shows that $\bpi_1^{\aone}({\mathbb P}^{2n}) = \gm{}$, which is, in particular, $\aone$-nilpotent.  To establish local $\aone$-nilpotence, we again appeal to Corollary~\ref{cor:checkingaonenilpotence}.  As in the proof of Proposition \ref{prop:p2n+1aonesimple}, we reduce to showing the action of $\gm{}$ on ${\mathbb A}^{2n+1} \setminus 0$ is homotopically nilpotent.  If $L/k$ is an extension field, and $u \in L^{\times}$ is a unit, then the scaling action by $u$ has motivic degree $\langle u^{2n+1} \rangle = \langle u \rangle$.  On the other hand, $1 - \langle u \rangle$ lies in $I(L)$.  Since $k$ is not formally real, so is $L$ and a well-known result of Pfister \cite[Proposition 31.4]{EKM} asserts that $I(L)$ is the nilradical of $GW(L)$.  Therefore, some power of $1 - \langle u \rangle$ is zero and we conclude. 
\end{proof}

\begin{rem}
\label{rem:p2nnotaonenilpotent}
If $k = {\mathbb R}$ or, more generally, any formally real field, then ${\mathbb P}^{2n}_k$ is not locally nilpotent.  Indeed, under this hypothesis, the filtration discussed in Proposition~\ref{prop:p2npointwiseaonenilpotent} remains infinite stalkwise.  In fact, the lower central series for the $\bpi_1^{\aone}({\mathbb P}^{2n}_k)$-action on the homotopy sheaf $\bpi_{2n}^{\aone}({\mathbb P}^{2n}) = \K^{MW}_{2n+1}$ is not stalkwise finite.  One may write down the $\aone$-$\gm{}$-lower central series for this sheaf as follows.  The kernel of the canonical epimorphism $\K^{MW}_{2n+1} \to \K^M_{2n+1}$ is $\mathbf{I}^{2n+2}$; the induced $\gm{}$-action on $\K^M_{2n+1}$ is trivial.  The induced action of $\gm{}$ on $\mathbf{I}^{2n+2}$ is, upon taking sections over fields, the action by multiplication by units.  As a consequence, the powers of $\mathbf{I}^j \subset \mathbf{I}^{2n+2}$, $j \geq 2n+2$ then yield a filtration for which $\gm{}$-acts trivially on the successive subquotients.
\end{rem}

\begin{thm}
\label{thm:bglevenlocallyaonenilpotent}
If $k$ is not formally real, then the $k$-space $BGL_{2n}$ is weakly $\aone$-nilpotent.
\end{thm}

\begin{proof}
As in the proof of Theorem~\ref{thm:bglnoddaonesimple}, there is an $\aone$-fiber sequence of the form
\[
{\mathbb P}^{2n} \longrightarrow BGL_{2n} \longrightarrow BSL_{2n+1}
\]
where the base is $\aone$-simply connected.  Since the fiber is locally $\aone$-nilpotent by appeal to Proposition~\ref{prop:p2npointwiseaonenilpotent}, the result then follows by appeal to Proposition~\ref{prop:totalspaceaonenilpotent}.
\end{proof}

\begin{rem}
\label{rem:bgl2nnotlocallyaonenilpotent}
Again, if $k = {\mathbb R}$ (or, more generally, any formally real field), then the space $BGL_{2n}$ fails to be locally $\aone$-nilpotent.  Indeed, along the lines of Remark~\ref{rem:p2nnotaonenilpotent} one can explicitly show that $\bpi_{2n}^{\aone}(BGL_{2n})$ does not have a stalkwise finite $\gm{}$-central series.  To see this, begin by observing that $\bpi_{2n}^{\aone}(BGL_{2n})$ is described in \cite[Theorem 1.1]{AsokFaselSpheres}: it is an extension of $\K^Q_{2n}$ by a sheaf $\mathbf{T}_{2n+1}$ that is the cokernel of a certain morphism $\K^{Q}_{2n+1} \to \K^{MW}_{2n+1}$.  The composite morphism $\K^Q_{2n+1} \to \mathbf{I}^{2n+1}$ is trivial by \cite[Lemma 3.13]{AsokFaselSpheres}, and $\mathbf{T}_{2n+1}$ is a fiber product of $\mathbf{I}^{2n+1}$ and another sheaf $\mathbf{S}_{2n+1}$ whose precise description is unimportant.  Appealing to the long exact sequence in $\aone$-homotopy sheaves attached to the $\aone$-fiber sequence appearing in the proof of Theorem~\ref{thm:bglevenlocallyaonenilpotent}, one sees that $\mathbf{T}_{2n+1}$ is precisely the image of $\bpi_{2n}^{\aone}({\mathbb P}^{2n})$ in $\bpi_{2n}^{\aone}(BGL_n)$.  The $\gm{}$-action on the copy of $\mathbf{I}^{2n+1}$ in $\mathbf{T}^{2n+1}$ thus does not have a stalkwise finite $\aone$-$\gm{}$-central series by the conclusion of Remark~\ref{rem:p2nnotaonenilpotent}.
\end{rem}

\section{Nilpotence, $R$-localization and $\aone$-homology}
\label{s:nilpotenceaonehomology}
In this section, we study $R$-localization in $\aone$-homotopy theory.  We introduce a notion of $R$-homology equivalence in $\aone$-homotopy theory, and show that $R$-localization is well-behaved on locally $\aone$-nilpotent spaces.  Section~\ref{ss:sheafcohomology} introduces $\aone$-homology (following F. Morel) and studies its various properties.  Section~\ref{ss:relativehurewicz} considers various refinements of Hurewicz theorems in $\aone$-homotopy theory; in particular, we finish the proof that $\aone$-nilpotent spaces have $\aone$-Postnikov towers that admit principal refinements (see Theorem \ref{thm:moorepostnikovcharacterization}).  Finally, Section~\ref{ss:rlocalizationnilpotentspaces} studies $R$-localization for $\aone$-nilpotent spaces; Theorems~\ref{thm:aonelocalizationofnilpotentspaces} and \ref{thm:Raonelocalizationoffibersequences} show that $R$-localization has good properties for (locally) $\aone$-nilpotent spaces.

\subsection{Sheaf cohomology and $\aone$-homology}
\label{ss:sheafcohomology}
In this section, we study $\aone$-homology theory and establish analogs of classical results about the $\aone$-homology of nilpotent spaces.  We establish generalizations of the known relative Hurewicz theorems in $\aone$-homotopy theory (keeping track of actions of the $\aone$-fundamental group).

\subsubsection*{Sheaf cohomology}
We briefly recall some notation related to sheaf cohomology in the context of local homotopy theory; we refer the reader to \cite[Chapter 8]{Jardine} for a detailed treatment of this circle of ideas, which goes back to the foundational work of Brown--Gersten.  Suppose $(\mathbf{C},\tau)$ is a site with enough points.  We write $\mathrm{Ch}_{\mathbf{C}}$ for the abelian category of chain complexes (i.e., differential of degree $-1$) of presheaves of abelian groups on $\mathbf{C}$ situated in degrees $\geq 0$. The category $\mathrm{Ch}_{\mathbf{C}}$ can be equipped with an injective local model structure: cofibrations are monomorphisms, weak equivalences are Nisnevich local quasi-isomorphisms (i.e., morphisms that induce isomorphisms on homology sheaves), and fibrations are defined by the lifting property.  We write $\D(\Ab(\mathbf{C}))$ for the homotopy category of $\mathrm{Ch}_{\mathbf{C}}$, i.e., the derived category of (bounded below complexes of) presheaves of abelian groups on ${\mathbf{C}}$.

If $\mathscr{X} \in \sPre(\mathbf{C})$ is a space, the presheaf $\Z[\mathscr{X}]$ is a simplicial abelian group.  We abuse terminology and write $\Z[\mathscr{X}]$ for the associated normalized chain complex which is an object of $\mathrm{Ch}_{\mathbf{C}}$ (the Dold--Kan correspondence).  Likewise, if $A$ is a chain complex in non-negative degree, then we write $K(A) \in \sPre(\mathbf{C})$ for the associated Eilenberg--Mac Lane space \cite[p. 212]{Jardine}.

We write $\H_i(\mathscr{X})$ for the associated homology sheaves of $\Z[\mathscr{X}]$, and we write $\tilde{\H}_i(\mathscr{X})$ for the associated reduced homology sheaves of $\mathscr{X}$ (i.e., the homology sheaves of the kernel $\tilde{\Z}[\mathscr{X}]$ of $\Z[\mathscr{X}] \to \Z[\ast]$).  Note that taking (reduced) homology sheaves commutes with taking stalks, i.e., if $s$ is a point of $(\mathbf{C},\tau)$, then $s^*\H_i(\mathscr{X}) = H_i(s^*\mathscr{X})$ (the latter being homology with integral coefficients of a simplicial set); we use this fact freely in the sequel.

\subsubsection*{$\aone$-homology}
We now recall $\aone$-homology and Morel's (effective) $\aone$-derived category.  The $\aone$-derived category $\Daone{k}$ is obtained by Nisnevich and $\aone$-localizing the derived category of presheaves of abelian groups on $\Sm_k$ (this is the construction described in \cite[\S 6.2]{MField}, but our notation for the category contains the superscript $\eff$ to distinguish from a corresponding ``$\gm{}$-stabilized" version).  We write $\Laone^{ab}$ for the endo-functor that effects this localization; this is an exact endo-functor on the category of chain complexes of presheaves of abelian groups on $\Sm_k$.

If $f: \mathscr{X} \to \mathscr{Y}$ is an $\aone$-weak equivalence in $\Spc_k$, then one knows the induced map $\Z[\mathscr{X}] \to \Z[\mathscr{Y}]$ is an isomorphism in $\Daone{k}$.  Thus, the assignment $\mathscr{X} \to \Z[\mathscr{X}]$ passes to a functor $\ho{k} \to \Daone{k}$.  We set
\[
\Z_{\aone}[\mathscr{X}] := \Laone^{ab}\Z[\mathscr{X}]
\]
and define $\aone$-homology sheaves by the formula
\[
\H_i^{\aone}(\mathscr{X}) := \H_i(\Z_{\aone}[\mathscr{X}]).
\]
As usual, one may also define analogs of reduced homology, and following the standard notation, we write $\tilde{H}_i^{\aone}(\mathscr{X})$ for the reduced $\aone$-homology sheaves of $\mathscr{X}$.

Morel's stable $\aone$-connectivity theorem \cite[Theorem 3]{MStable} implies that if $\mathscr{X}$ is a space, then $\Z_{\aone}[\mathscr{X}]$ is a $(-1)$-connected object of the derived category of Nisnevich sheaves of abelian groups.  In particular, $\H_i^{\aone}(\mathscr{X}) = 0$ for $i < 0$.

If $(\mathscr{X},x)$ is a pointed space, then the Hurewicz homomorphism is induced by a morphism
\[
\mathscr{X} \longrightarrow K(\Z_{\aone}[\mathscr{X}]).
\]
It follows that there are induced actions of $\bpi_1^{\aone}(\mathscr{X})$ on all the $\aone$-homology sheaves.

\begin{rem}
\label{rem:motives}
Voevodsky's derived category of motives is constructed by starting with the ``presheaves with transfers" and then Nisnevich and $\aone$-localizing.  Write $\Sm\Cor_k$ for the category whose objects are smooth schemes and where morphisms are finite correspondences.  One writes $\mathrm{PST}_k$ for the category of contravariant additive functors on $\Sm\Cor_k$; this is called the category of presheaves with transfers.  There is a functor $\Sm_k \to \Sm\Cor_k$ that induces a restriction functor from presheaves with transfers to presheaves of abelian groups.  Voevodsky's derived category of motives $\DM^{\eff}_{k,\Z}$ is constructed by localizing the derived category of chain complexes of presheaves with transfers with respect to Nisnevich and $\aone$-local weak equivalences.  Write $\Z_{tr}[X]$ for the representable functor on $\Sm\Cor_k$ determined by a smooth scheme.

The left Kan extension of the assignment $\Z[X] \to \Z_{tr}[X]$ yields a functor $\Daone{k} \to \DM^{\eff}_{k,\Z}$ (usually called the functor of ``adding transfers").  It follows that any morphism of spaces $f: \mathscr{X} \to \mathscr{Y}$ that induces an isomorphism in the $\aone$-derived category also induces an isomorphism in $\DM^{\eff}_{k,\Z}$.  Nevertheless, the two categories are typically quite far from each other.
\end{rem}

\subsubsection*{$\aone$-homology equivalences and cohomology equivalences}
Suppose $R \subset \Q$ is a subring.  We write $R_{\aone}[\mathscr{X}]$ for $R \tensor \Laone^{ab} \Z[\mathscr{X}]$.  Since tensoring with $R$ is exact and $\Laone^{ab}$ is exact, the order in which we localize and tensor is irrelevant.  We set $\H_i^{\aone}(\mathscr{X},R) := \H_i^{\aone}(R_{\aone}[\mathscr{X}])$.  We now give a number of equivalent characterizations of morphisms inducing an isomorphism on $\aone$-homology.

\begin{prop}
\label{prop:homologicalandcohomologicalequivalences}
The following conditions on a morphism $f: \mathscr{X} \to \mathscr{Y}$ of spaces are equivalent.
\begin{enumerate}[noitemsep,topsep=1pt]
\item The morphism $R_{\aone}[f]: R_{\aone}[\mathscr{X}] \to R_{\aone}[\mathscr{Y}]$ is a quasi-isomorphism.
\item For every integer $i \geq 0$, the morphism $f_*: \H_i^{\aone}(\mathscr{X},R) \to \H_i^{\aone}(\mathscr{Y},R)$ is an isomorphism.
\item For any strictly $\aone$-invariant sheaf $\mathbf{M}$ of $R$-modules and any integer $i \geq 0$, the map $f^*: H^i(\mathscr{Y},\mathbf{M}) \to H^i(\mathscr{X},\mathbf{M})$ is an isomorphism.
\end{enumerate}
\end{prop}

\begin{proof}
That (1) $\Leftrightarrow$ (2) is immediate.  That (2) $\Rightarrow$ (3) follows from the universal coefficient spectral sequence.  Indeed, there is a, functorial, strongly convergent spectral sequence of the form
\[
E^{p,q}_2 = Ext^p_{\Ab_k}(\H_q(R_{\aone}[\mathscr{X}]),\mathbf{M}) \Longrightarrow \hom_{D(\Ab_k)}(R_{\aone}[\mathscr{X}],\mathbf{M}[p+q]),
\]
see \cite[Lemma 8.30]{Jardine} for the spectral sequence; the strong convergence follows from the fact that the homology sheaves of $R_{\aone}[\mathscr{X}]$ vanish in degrees $< 0$, which is the stable $\aone$-connectivity theorem.

For the implication (3) $\Rightarrow$ (2), we proceed as follows.  If $\mathbf{C}$ is a $i-1$-connected and $\aone$-local chain complex of $R$-modules, then for any strictly $\aone$-invariant sheaf of $R$-modules $\mathbf{M}$, there is a canonical isomorphism of the form
\[
\hom(H_i^{\aone}(\mathbf{C},R),\mathbf{M}) \cong \mathrm{RHom}(\mathbf{C},\mathbf{M}[i]).
\]  
In particular, we conclude that if $\mathrm{RHom}(\mathbf{C},\mathbf{M}[i])$ vanishes for all $\mathbf{M}$, then $H_i^{\aone}(\mathbf{C})$ vanishes by the Yoneda lemma.  Applying this observation inductively to the cone of the map $f$, the hypothesis in (3) implies inductively that the cone of of $f$ has vanishing $\aone$-homology sheaves in all degrees, and the result follows. 
%The smallest full subcategory of $\Daone{k}$ containing $\mathbf{M}[n]$ for every strictly $\aone$-invariant sheaf of $R$-modules $\mathbf{M}$ and every integer $n \geq 0$ and closed under extensions is the positive ($\geq 0$) part of the homotopy $t$-structure on $\Daone{k}$.  In particular, by Morel's stable connectivity theorem, $R_{\aone}[\mathscr{X}]$ lies in the positive part of the homotopy $t$-structure, and the result follows.
\end{proof}

\begin{rem}
Extending Remark~\ref{rem:motives}, if $f$ is a morphism of spaces such that $R_{\aone}[f]$ is a quasi-isomorphism, then the induced map of motives with $R$-coefficients is an isomorphism as well.
\end{rem}

\subsubsection*{A homological connectivity lemma}
The next result analyzes how homological connectivity behaves under $\aone$-localization.

\begin{lem}
\label{lem:homologicalconnectivity}
Assume $\mathscr{X}$ is a simplicially connected space and $n \geq 1$ is an integer.  If $\mathscr{X}$ is homologically $(n-1)$-connected, i.e., $\tilde{\H}_i(\mathscr{X})$ vanishes for $i < n$, then the following statements hold:
\begin{enumerate}[noitemsep,topsep=1pt]
\item the sheaves $\tilde{\H}_i^{\aone}(\mathscr{X})$ vanish for $i < n$, and
\item the morphism $\tilde{\H}_n(\mathscr{X}) \to \tilde{\H}_n^{\aone}(\mathscr{X})$ is the initial morphism from $\tilde{\H}_n(\mathscr{X})$ to a strictly $\aone$-invariant sheaf.
\end{enumerate}
\end{lem}

\begin{proof}
While the first statement is an immediate consequence of the stable $\aone$-connectivity theorem, we give a different proof that will be useful in setting up the notation to establish the second statement as well.  By assumption $\mathscr{X}$ is simplicially connected.  Therefore, the reduced singular chain complex $\tilde{\Z}[\mathscr{X}]$ is a chain complex that has vanishing homology sheaves in negative degrees and is thus a $0$-connected chain complex.  By Morel's stable $\aone$-connectivity theorem, the complex $\tilde{\Z}_{\aone}[\mathscr{X}] := \Laone^{ab} \tilde{\Z}[\mathscr{X}]$ is thus also $0$-connected.

Since $\tilde{\Z}[\mathscr{X}]$ is $0$-connected, by \cite[Proposition 6.25]{MField}, the space $K(\tilde{\Z}_{\aone}[\mathscr{X}])$ is $\aone$-local.  As a consequence, the map $K(\tilde{\Z}[\mathscr{X}]) \to K(\tilde{\Z}_{\aone}[\mathscr{X}])$ arising by functoriality factors through a morphism
\[
\Laone K(\tilde{\Z}[\mathscr{X}]) \longrightarrow K(\tilde{\Z}_{\aone}[\mathscr{X}]).
\]
By \cite[Corollary 6.27]{MField} the morphism of the previous display is a simplicial weak equivalence.

By construction, the homotopy sheaves of $K(\tilde{\Z}[\mathscr{X}])$ are the (reduced) homology sheaves of $\mathscr{X}$.  Our assumption on the homological connectivity of $\mathscr{X}$ thus implies that the space $K(\tilde{\Z}[\mathscr{X}])$ is simplicially $(n-1)$-connected.  Therefore, Morel's unstable connectivity theorem guarantees that $\Laone K(\tilde{\Z}[\mathscr{X}])$ is again simplicially $(n-1)$-connected.  The simplicial weak equivalence of the previous paragraph thus implies $K(\tilde{\Z}_{\aone}[\mathscr{X}])$ is also simplicially $(n-1)$-connected.  However, the homotopy sheaves of $K(\tilde{\Z}_{\aone}[\mathscr{X}])$ are precisely the reduced $\aone$-homology sheaves of $\mathscr{X}$ and thus vanish in degrees $< n$ as well.

For the final statement, observe that \cite[Corollary 6.60]{MField} applied to the simplicially $(n-1)$-connected space $K(\tilde{\Z}[\mathscr{X}])$ shows that the morphism
\[
\tilde{\H}_n(\mathscr{X}) \cong \bpi_n(K(\tilde{\Z}[\mathscr{X}])) \longrightarrow \bpi_n(\Laone K(\tilde{\Z}[\mathscr{X}])) \cong \bpi_n(K(\tilde{\Z}_{\aone}[\mathscr{X}])) \cong \tilde{\H}_n^{\aone}(\mathscr{X})
\]
has the desired universal property.
\end{proof}

\subsection{Relative Hurewicz theorems revisited}
\label{ss:relativehurewicz}
In this section, we refine the Hurewicz theorems in $\aone$-homotopy theory by keeping track of the action of the $\aone$-fundamental sheaf of groups.  In particular, we refine the strong form of the relative Hurewicz theorem, which follows from the Blakers--Massey theorem established in \cite[Theorem 4.1]{AFComparison} (and independently in \cite[Theorem 2.3.8]{Strunk}).

\subsubsection*{The Relative Hurewicz theorem and coinvariants}
Suppose $f: \mathscr{E} \to \mathscr{B}$ is a morphism of pointed, connected simplicial presheaves with homotopy fiber $\mathscr{F}$ and homotopy cofiber $\mathscr{C}$.  In that case, there is a morphism of simplicial fiber sequences
\[
\xymatrix{
\mathscr{F} \ar[r]\ar[d] & \mathscr{E} \ar[r]\ar[d] & \mathscr{B} \ar[d] \\
\Omega \mathscr{C} \ar[r] & \ast \ar[r]& \mathscr{C}.
}
\]
The sheaf $\bpi_1(\mathscr{E})$ acts on the higher homotopy sheaves $\bpi_i(\mathscr{F})$.  By functoriality there is also an action of $\bpi_1(\mathscr{E})$ on the higher homotopy of sheaves in the bottom row, though this action is necessarily trivial since $\ast$ is certainly simplicially $1$-connected.

The relative Hurewicz theorem analyzes the connectivity of the morphism $\mathscr{F} \to \Omega \mathscr{C}$ under suitable hypotheses on the connectivity of $f$.  More precisely, there is an evident morphism $\bpi_i(\mathscr{F}) \to \bpi_i(\Omega \mathscr{C})$.  The counit of the loop suspension adjunction yields a morphism $\Sigma \Omega \mathscr{C} \to \mathscr{C}$.  This morphism yields, in turn the composite:
\[
\bpi_i(\Omega\mathscr{C}) \longrightarrow \H_i(\Omega \mathscr{C}) \cong \H_{i+1}(\Sigma \Omega \mathscr{C}) \longrightarrow \H_{i+1}(\mathscr{C}),
\]
where the isomorphism is the suspension isomorphism.  Thus, we obtain a morphism
\[
\bpi_i(\mathscr{F}) \longrightarrow \H_{i+1}(\mathscr{C}).
\]
that we will call the relative Hurewicz homomorphism.

Note that, even if $\mathscr{E}$ and $\mathscr{B}$ are $\aone$-local, $\mathscr{C}$ need not be.  Nevertheless, the morphism $\mathscr{C} \to \Laone \mathscr{C}$ is an isomorphism on $\aone$-homology sheaves, and models the $\aone$-homotopy cofiber of $\mathscr{E} \to \mathscr{B}$. The canonical map $\mathscr{C} \to \Laone \mathscr{C}$ yields a map $\Omega \mathscr{C} \to \Omega \Laone \mathscr{C}$ and the latter is already $\aone$-local.  Composing with the map $\mathscr{F} \to \Omega \mathscr{C}$ described above, we obtain a map
\[
\mathscr{F} \longrightarrow \Omega \Laone \mathscr{C}.
\]
The counit of the loop-suspension adjunction yields a morphism $\Sigma \Omega \Laone \mathscr{C} \to \Laone \mathscr{C}$.  Since the target of this morphism is $\aone$-local (even though the source need not be), there is a factorization
\[
\Sigma \Omega \Laone \mathscr{C} \longrightarrow \Laone \Sigma \Omega \Laone \mathscr{C} \longrightarrow \Laone \mathscr{C}
\]
In a fashion analogous to that described above, we obtain morphisms $\bpi_i(\Omega \Laone \mathscr{C}) \to \H_{i+1}^{\aone}(\mathscr{C})$ and thus a relative Hurewicz homomorphism
\[
\bpi_i^{\aone}(\mathscr{F}) \longrightarrow \H_{i+1}^{\aone}(\mathscr{C})
\]
that we now analyze.

\begin{thm}
\label{thm:relativeHurewicz}
Suppose $f: \mathscr{E} \to \mathscr{B}$ is a morphism of pointed, $\aone$-local and simplicially connected spaces.  Write $\mathscr{F}$ for the simplicial homotopy fiber of $f$ and $\mathscr{C}$ for the simplicial homotopy cofiber of $f$.  If $\mathscr{F}$ is $\aone$-$(n-1)$-connected for some integer $n \geq 1$, then the following statements hold.
\begin{enumerate}[noitemsep,topsep=1pt]
    \item The sheaves $\H_i^{\aone}(\mathscr{C})$ vanish for $i \leq n$.
    \item The relative Hurewicz homomorphism
        \[
        \bpi_n^{\aone}(\mathscr{F}) \longrightarrow \H_{n+1}^{\aone}(\mathscr{C})
        \]
        is the initial morphism from $\bpi_n^{\aone}(\mathscr{F})$ to a strictly $\aone$-invariant sheaf on which $\bpi_1^{\aone}(\mathscr{E})$ acts trivially.
    \item If $n = 1$ and $\bpi_1(\mathscr{F})$ is strictly $\aone$-invariant, or if $n \geq 2$, then the morphism of the previous point induces an isomorphism
        \[
        \bpi_n^{\aone}(\mathscr{F})_{\bpi_1^{\aone}(\mathscr{E})} \isomto \H_{n+1}^{\aone}(\mathscr{C}),
        \]
        where the source is the sheaf of coinvariants.
\end{enumerate}
\end{thm}

\begin{proof}
Since $\mathscr{E}$ and $\mathscr{B}$ are simplicially connected and $\aone$-local by assumption, the simplicial homotopy fiber $\mathscr{F}$ is then $\aone$-local by Lemma \ref{lem:homotopyfiberaonelocal} and thus $\bpi_i^{\aone}(\mathscr{F}) = \bpi_i(\mathscr{F})$ so $\mathscr{F}$ is simplicially $(n-1)$-connected.

The simplicial homotopy cofiber $\operatorname{hocofib} f$ need not be $\aone$-local, so we consider $\Laone \operatorname{hocofib} f$ and we have a commutative diagram of simplicial fiber sequences
\[
\xymatrix{
\mathscr{F} \ar[r]\ar[d] & \mathscr{E} \ar[r]\ar[d] & \mathscr{B} \ar[d] \\
\Omega \Laone \operatorname{hocofib} f \ar[r] & \ast \ar[r] & \Laone \operatorname{hocofib} f.
}
\]
where $\Omega \Laone \operatorname{hocofib} f$ is already $\aone$-local.

Consider the action of $\bpi_1(\mathscr{E})$ on all of the spaces in the diagram.  Since we saw above that $\mathscr{F}$ is simplicially $(n-1)$-connected, the classical relative Hurewicz theorem \cite[Theorem IV.7.2]{Whitehead} applied stalkwise implies that the map $\bpi_i(\mathscr{F}) \to \H_{i+1}(\operatorname{hocofib} f)$ is the map ``factoring out the action of $\bpi_1(\mathscr{E})$."  More precisely, we conclude that $\H_{i+1}(\operatorname{hocofib} f)$ vanishes in degrees $\leq n$ and that $\H_{n+1}(\operatorname{hocofib} f)$ is the $\bpi_1(\mathscr{E})$ coinvariants of $\bpi_n(\mathscr{F})$.  Since $\operatorname{hocofib} f$ is homologically $n$-connected, it follows by appeal to Lemma~\ref{lem:homologicalconnectivity}(1) that $\tilde{\H}_{i+1}^{\aone}(\operatorname{hocofib} f)$ also vanishes for $i \leq n$, which establishes the first point of the theorem.

To establish the second point of the theorem, observe that Lemma~\ref{lem:homologicalconnectivity}(2) implies that $\H_{n+1}^{\aone}(\operatorname{hocofib} f)$ is the initial strictly $\aone$-invariant sheaf admitting a morphism from $\H_{n+1}(\operatorname{hocofib} f)$.  We observed in the preceding paragraph that $\bpi_1(\mathscr{E})$-action on the latter sheaf is already trivial, so the induced action on $\H_{n+1}^{\aone}(\operatorname{hocofib} f)$ is also trivial.

To see that the morphism in question is an isomorphism if either $n = 1$ and $\bpi_1^{\aone}(\mathscr{F})$ is strictly $\aone$-invariant or for $n \geq 2$, we proceed as follows.  Let $\mathbf{K}$ be the kernel of the morphism $\bpi_n^{\aone}(\mathscr{F}) \to \H_{n+1}^{\aone}(\operatorname{hocofib} f)$, then $\mathbf{K}$ and the quotient $\bpi_n^{\aone}(\mathscr{F})/\mathbf{K}$ are strictly $\aone$-invariant by appeal to Theorem~\ref{thm:stableconnectivity}.  By universality, it follows that the identity map on $\bpi_n^{\aone}(\mathscr{F})/\mathbf{K}$ factors through an epimorphism $\bpi_n^{\aone}(\mathscr{F})/\mathbf{K} \to \H_{n+1}^{\aone}(\operatorname{hocofib} f)$.  Then, as observed above, the ordinary Hurewicz theorem identifies $\bpi_n^{\aone}(\mathscr{F})/\mathbf{K}$ with $\H_{n+1}(\operatorname{hocofib} f)$ and the latter is precisely the $\bpi_1(\mathscr{E})$-coinvariants of $\bpi_n(\mathscr{F})$.
\end{proof}

\subsubsection*{The $\aone$-Whitehead theorem for $\aone$-simple spaces}
The relative Hurewicz theorem above has a number of consequences, including the following result that generalizes other ``homological" forms of the $\aone$-Whitehead theorem in the literature (e.g., \cite[Corollary 2.22]{WickelgrenWilliams}).

\begin{thm}
\label{thm:whiteheadtheoremforsimplespaces}
If $f: \mathscr{X} \to \mathscr{Y}$ is a morphism of $\aone$-simple spaces (see \textup{Definition~\ref{defn:aonenilpotentspace}}) and $f_*$ is an isomorphism $\Z[\mathscr{X}] \to \Z[\mathscr{Y}]$ in $\Daone{k}$, then $f$ is an $\aone$-weak equivalence.
\end{thm}

\begin{proof}
Since $\mathscr{X}$ and $\mathscr{Y}$ are $\aone$-simple, we know their $\aone$-fundamental sheaves of groups are strictly $\aone$-invariant.  For any $\aone$-simple space, the canonical map $\bpi_1^{\aone}(\mathscr{X}) \longrightarrow \H_1^{\aone}(\mathscr{X})$ is an isomorphism.  Indeed, the latter is the initial strictly $\aone$-invariant sheaf of (abelian) groups admitting a map from $\bpi_1(\mathscr{X})$, but since the former is strictly $\aone$-invariant already, this map must be an isomorphism.  Let $\mathscr{F}$ be the $\aone$-homotopy fiber of $f$ and consider the associated long exact sequence of homotopy sheaves:
\[
\bpi_2^{\aone}(\mathscr{X}) \longrightarrow
\bpi_2^{\aone}(\mathscr{Y}) \longrightarrow \bpi_1^{\aone}(\mathscr{F}) \longrightarrow \bpi_1^{\aone}(\mathscr{X}) \stackrel{f_*}{\longrightarrow} \bpi_1^{\aone}(\mathscr{Y}) \longrightarrow 1;
\]
here the map $f_*$ coincides with the map on $\aone$-homology sheaves induced by $f$ by the discussion above.  Therefore, our assumption on $f$ guarantees that the map $f_*$ is an isomorphism.  Thus, we conclude that $\bpi_1^{\aone}(\mathscr{F})$ is a cokernel of a morphism of strictly $\aone$-invariant sheaves and thus itself strictly $\aone$-invariant.  Granted that observation, the result follows from Theorem~\ref{thm:relativeHurewicz}(3) and a straightforward induction.
\end{proof}

\begin{rem}
Theorem~\ref{thm:whiteheadtheoremforsimplespaces} applies to show that a map of pointed, $\aone$-connected, $\aone$-$h$-spaces is an $\aone$-weak equivalence if and only if the map is an $\aone$-homology equivalence.
\end{rem}

\subsubsection*{Nilpotence and Moore--Postnikov factorizations}
With the relative Hurewicz theorem in hand, we may establish the reverse implication in Theorem \ref{thm:moorepostnikovcharacterization}.

\begin{cor}
\label{cor:MoorePostnikovfactorizationconclusion}
Assume $k$ is a perfect field.  If $f: \mathscr{E} \to \mathscr{B}$ an $\aone$-nilpotent morphism of pointed $\aone$-connected spaces, then the $\aone$-Moore--Postnikov tower admits a principal refinement (see \textup{Definition~\ref{defn:principalrefinement}}).
\end{cor}

\begin{proof}
Consider the morphism $\tau_{\leq i}f \to \tau_{\leq i-1}f$.  We treat two cases: $i = 1$ and $i \geq 2$.  For $i = 1$, we want to show that $\bpi_1^{\aone}(\mathscr{F})$ admits a $\bpi_1^{\aone}(\tau_{\leq 1}f) = \bpi_1^{\aone}(\mathscr{E})$-invariant central series.  Since $k$ is perfect, by Proposition~\ref{prop:characterizationofaonenilpotent} we know the $\aone$-upper central series witnesses the $\aone$-nilpotence of $\bpi_1^{\aone}(\mathscr{F})$; it is furthermore $\bpi_1^{\aone}(\mathscr{E})$-invariant by definition of the action and the fact that the terms of the upper central series are characteristic subgroup sheaves (i.e., stable by automorphisms of $\mathbf{G}$).  By induction, one sees that the subquotients of the $\aone$-upper central series provide the necessary factorization.

For $i \geq 2$, the result follows directly from the relative $\aone$-Hurewicz theorem as in the classical argument \cite[Theorem II.2.9 and II.2.14]{HMR}.  Arguing inductively, the $\aone$-relative Hurewicz theorem above yields a canonical strictly $\aone$-invariant quotient of the $i$-th homotopy sheaf of the $\aone$-homotopy fiber equipped with a trivial action of $\bpi_1^{\aone}(\tau_{\leq i}f) = \bpi_1^{\aone}(\mathscr{E})$; in fact, as the sheaf of coinvariants this quotient is the maximal strictly $\aone$-invariant quotient with trivial action of $\bpi_1^{\aone}(\mathscr{E})$.  We define the first stage of the refinement to be the pullback of $\tau_{\leq i-1}f \to K(\H_i^{\aone}(f),i+1)$ and iterate this procedure.  By assumption $\bpi_1^{\aone}(\tau_{\leq i}f)$ is an $\aone$-$\bpi_1^{\aone}(\mathscr{E})$-nilpotent sheaf, so has a finite $\bpi_1^{\aone}(\mathscr{E})$-central series.  The $\bpi_1^{\aone}(\mathscr{E})$-central series obtained by iteratively taking coinvariants is thus finite and the iterative procedure just sketched necessarily terminates. 
\end{proof}

\begin{rem}
In contrast to the standard proof of this result, for $n = 1$, the $\aone$-principal refinement constructed above need not be ``maximal" because of the form of the relative $\aone$-Hurewicz theorem \ref{thm:relativeHurewicz}.  Indeed, the first quotient of $\bpi_1^{\aone}(\mathscr{F})$ described above is just {\em some} strictly $\aone$-invariant sheaf on which $\bpi_1^{\aone}(\mathscr{E})$ acts trivially: {\em a priori} there is no reason this quotient has to coincide with the coinvariant subsheaf in general since we do not know that the latter sheaf is actually strictly $\aone$-invariant.  Moreover, because we use the $\aone$-upper central series, the above principal refinement fails to be a {\em functorial} principal refinement.
\end{rem}

\subsection{$R$-localization for $\aone$-nilpotent spaces}
\label{ss:rlocalizationnilpotentspaces}
Finally, in this section, we discuss $R$-localization of motivic spaces.  We show that locally $\aone$-nilpotent spaces have well-behaved localizations and that localization preserves various fiber sequences.

\subsubsection*{The basic definitions}
\begin{defn}
\label{defn:plocalweakequivalence}
Suppose $f: \mathscr{X} \to \mathscr{Y}$ is a morphism in $\Spc_k$.  If $R \subset \Q$ is a subring, then we will say that $f$ is an {\em $R$-local $\aone$-weak equivalence} or {\em $R$-$\aone$-weak equivalence} if the induced morphism $R_{\aone}[\mathscr{X}] \to R_{\aone}[\mathscr{Y}]$ is an isomorphism in $\Daone{k}$.  When $R = \Z_{(P)}$ for a set of primes $P$, we will call $f$ a $P$-local $\aone$-weak equivalence and a rational $\aone$-weak equivalence if $P$ is empty; when $R = \Z[\frac{1}{n}]$, we will say that $f$ is an $\aone$-weak equivalence after inverting $n$.
\end{defn}

\begin{rem}
We have defined the notion $R$-local $\aone$-weak equivalence above using homology in contrast to our earlier definition of $R$-local equivalence using self-maps of the circle Section~\ref{ss:rlocalizationofsimplicialpresheaves}.  One reason for this disjunction is simply convenience of referencing.  In Theorem~\ref{thm:aonelocalizationofnilpotentspaces} we will see that, for weakly $\aone$-nilpotent spaces, in essence, $R$-$\aone$-localization can be modeled by first $\aone$-localizing and then applying the $R$-localization functor studied previously.  
\end{rem}

Granted this definition, we may construct the $R$-$\aone$-local homotopy category by techniques of Bousfield localization.

\begin{defn}
\label{defn:plocalhomotopycategory}
We write ${\ho{k}}_{R}$ for the $R$-$\aone$-local homotopy category, i.e., the category obtained by left Bousfield localizing $\Spc_k$ at the set of $R$-$\aone$-weak equivalences.  Likewise, if $n \neq 0$, we write $\ho{k}[\frac{1}{n}]$ for the category ${\ho{k}}_{\Z[\frac{1}{n}]}$.
\end{defn}

\begin{defn}
\label{defn:Raonelocalspaces}
A space $\mathscr{Z}$ is {\em $R$-$\aone$-local} if for every $R$-$\aone$-weak equivalence $f: \mathscr{X} \to \mathscr{Y}$, the map $\operatorname{Map}(\mathscr{Y},\mathscr{Z}) \to \operatorname{Map}(\mathscr{X},\mathscr{Z})$ of derived mapping spaces is a weak equivalence.
\end{defn}

The next lemma summarizes some basic facts about $R$-$\aone$-local spaces.

\begin{lem}
\label{lem:propertiesofraonelocalspaces}
The following statements hold.
\begin{enumerate}[noitemsep,topsep=1pt]
\item If $\mathbf{A}$ is an $R$-local strictly $\aone$-invariant sheaf of abelian groups, then for any integer $n \geq 0$, $K(\mathbf{A},n)$ is $R$-$\aone$-local.
\item If $\mathscr{F} \to \mathscr{E} \to \mathscr{B}$ is a simplicial fiber sequence of pointed, $\aone$-connected spaces and both $\mathscr{E}$ and $\mathscr{B}$ are $R$-$\aone$-local, then $\mathscr{F}$ is $R$-$\aone$-local as well; the $R$-$\aone$-model structure is right proper.
\item If $\mathbf{G}$ is an $R$-local $\aone$-nilpotent sheaf of groups, then $B\mathbf{G}$ is $R$-$\aone$-local.
\item If $\mathscr{X} \in \Spc_k$ is a connected $\aone$-nilpotent space such that $\bpi_i^{\aone}(\mathscr{X})$ is $R$-local for every $i > 0$, then $\mathscr{X}$ is $R$-$\aone$-local.
\end{enumerate}
\end{lem}

\begin{proof}
Since $K(\mathbf{A},n)$ is $\aone$-local, for any space $\mathscr{X}$ there are identifications of the form:
\[
[\mathscr{X},K(\mathbf{A},n)]_s  = [\mathscr{X},K(\mathbf{A},n)]_{\aone} = H^n_{\Nis}(\mathscr{X},\mathbf{A}).
\]
Granted these identifications, the fact that $K(\mathbf{A},n)$ is $R$-$\aone$-local is an immediate consequence of the equivalent conditions listed in Proposition~\ref{prop:homologicalandcohomologicalequivalences}.

For the second statement, observe that applying $\mathrm{Map}(\mathscr{X},-)$ preserves fiber sequences. If $\mathscr{Y} \to \mathscr{X}$ is an $R$-$\aone$-equivalence, then there is a morphism of simplicial fiber sequences
\[
\xymatrix{
\mathrm{Map}(\mathscr{X},\mathscr{F}) \ar[r]\ar[d] & \mathrm{Map}(\mathscr{X},\mathscr{E}) \ar[r]\ar[d] & \mathrm{Map}(\mathscr{X},\mathscr{B}) \ar[d]\\
\mathrm{Map}(\mathscr{Y},\mathscr{F}) \ar[r] & \mathrm{Map}(\mathscr{Y},\mathscr{F}) \ar[r] & \mathrm{Map}(\mathscr{Y},\mathscr{B})
}
\]
By, e.g., considering the associated long exact sequence in homotopy, one checks that the map $\mathrm{Map}(\mathscr{X},\mathscr{F}) \to \mathrm{Map}(\mathscr{Y},\mathscr{F})$ in the above diagram is again a weak equivalence.  The right properness of the $R$-$\aone$-local model structure is now immediate from this observation.

For the third statement, consider $B\mathbf{G}$ where $\mathbf{G}$ is $R$-local and $\aone$-nilpotent.  In that case, $\mathbf{G}$ may be written as an iterated central extension by strictly $\aone$-invariant sheaves of $R$-modules.  Given such a central extension, there is an associated $\aone$-fiber sequence of the form
\[
B\mathbf{G} \longrightarrow B\mathbf{G}' \longrightarrow K(\mathbf{A},2).
\]
Point (1) guarantees that $K(\mathbf{A},2)$ is $R$-$\aone$-local, and we may inductively assume that $B\mathbf{G}'$ is $R$-$\aone$-local, so we conclude that $B\mathbf{G}$ is $R$-$\aone$-local as well.

The final statement is immediate from Theorem \ref{thm:moorepostnikovcharacterization} and the preceding points.
\end{proof}

\subsubsection*{$R$-$\aone$-localization of $\aone$-nilpotent groups}
\begin{lem}
\label{lem:Raonelocalizationabeliancase}
If $\mathbf{A}$ is a strictly $\aone$-invariant sheaf of groups, then for every integer $n > 0$,
\[
\LR \mathrm{R}_{\Nis} K(\mathbf{A},n) \cong K(R \tensor \mathbf{A},n)
\]
is a functorial $R$-$\aone$-localization of $K(\mathbf{A},n)$.
\end{lem}

\begin{proof}
The space $K(R \tensor \mathbf{A},n)$ is $R$-$\aone$-local by Lemma \ref{lem:propertiesofraonelocalspaces}(1).  The canonical morphism $\mathbf{A} \to R \tensor \mathbf{A}$ yields a morphism $K(\mathbf{A},n) \to K(R \tensor \mathbf{A},n)$.  We want to show that this morphism is an $R$-$\aone$-weak equivalence.  In fact, the morphism $\mathrm{R}_{\Nis} K(\mathbf{A},n) \to \mathrm{R}_{\Nis} K(R \tensor \mathbf{A},n)$ is an $R$-local simplicial weak equivalence.  Indeed, this can be checked stalkwise, in which case it follows from \cite[Chapter V.3.2]{BK}.
\end{proof}

We may now define a functorial $R$-$\aone$-localization for $\aone$-nilpotent sheaves of groups.

\begin{thm}
\label{thm:localizationofaonenilpotentgroups}
Assume $k$ is a perfect field, and suppose $\mathbf{G} \in \Grp^{\aone}_k$ is an $\aone$-nilpotent sheaf of groups.
\begin{enumerate}[noitemsep,topsep=1pt]
\item The space $\LR B\mathbf{G}$ is $R$-$\aone$-local.
\item The functor $\Ab^{\aone}_k \to \Mod^{\aone}_R$ sending a strictly $\aone$-invariant sheaf $\mathbf{A}$ to $R \tensor \mathbf{A}$ extends to a functor
    \[
    \mathbf{G} \to \bpi_1(\LR B_{\Nis}\mathbf{G}) =: \mathbf{G}_R
    \]
    from the category of $\aone$-nilpotent sheaves of groups to the category of $R$-local $\aone$-nilpotent sheaves of groups.
\item The functor of (2) preserves short exact sequences of $\aone$-nilpotent sheaves of groups.
\item For any $\aone$-nilpotent sheaf of groups $\mathbf{G}$, there is a canonical isomorphism $(\mathbf{G}_{-1})_{R} \cong (\mathbf{G}_R)_{-1}$ (see \textup{Definition~\ref{defn:contraction}}).
\end{enumerate}
\end{thm}

\begin{proof}
If $\mathbf{A}$ is strictly $\aone$-invariant, then we know that $\LR B\mathbf{A}$ is $R$-$\aone$-local by appeal to Lemma~\ref{lem:Raonelocalizationabeliancase} and $\bpi_1(\LR B\mathbf{A}) \cong \mathbf{A} \tensor R =: \mathbf{A}_R$.

To treat the general case, suppose we have a central extension of the form
\[
1 \longrightarrow \mathbf{A} \longrightarrow \mathbf{G} \longrightarrow \mathbf{G}' \longrightarrow 1
\]
where $\mathbf{A}$ is strictly $\aone$-invariant and $\mathbf{G}$ and $\mathbf{G}'$ are $\aone$-nilpotent.  In that case, there is a fiber sequence of the form
\[
B\mathbf{G} \longrightarrow B\mathbf{G}' \longrightarrow K(\mathbf{A},2)
\]
in which each term is $\aone$-local (i.e., this is an $\aone$-fiber sequence).  We claim that the sequence
\begin{equation}
\label{eqn:inductionstep}
\LR B\mathbf{G} \longrightarrow \LR B\mathbf{G}' \longrightarrow \LR K(\mathbf{A},2)
\end{equation}
is a simplicial fiber sequence of $\aone$-local spaces.  To this end, observe that $\LR K(\mathbf{A},2)$ is simplicially $1$-connected and $\aone$-local by appeal to Lemma~\ref{lem:Raonelocalizationabeliancase}.  Next, observe that $B\mathbf{G}'$ and $B\mathbf{G}$ are $\aone$-nilpotent and therefore nilpotent in the classical sense.  It follows from Lemma~\ref{lem:Rlocalizationnilpotentfiberlemma} that the canonical map from the $R$-localization of $B\mathbf{G}$ to the homotopy fiber of $\LR B\mathbf{G}' \to \LR K(\mathbf{A},2)$ is a simplicial weak equivalence.

We claim that by induction $\LR B\mathbf{G}$ is $R$-$\aone$-local.  The base case is Lemma~\ref{lem:Raonelocalizationabeliancase}: in particular, we may assume inductively that $\LR B\mathbf{G}'$ is simplicially connected, $\aone$-local, has $\bpi_1(\LR B\mathbf{G}')$ an $R$-local strongly $\aone$-invariant sheaf of groups and is $1$-truncated, i.e., all higher homotopy sheaves vanish.  In that case, Lemma~\ref{lem:propertiesofraonelocalspaces} guarantees that $\LR B\mathbf{G}'$ is $R$-$\aone$-local.  Then, combining Lemma~\ref{lem:Raonelocalizationabeliancase} and Lemma~\ref{lem:propertiesofraonelocalspaces} with the fiber sequence of \eqref{eqn:inductionstep} implies that $\LR B\mathbf{G}$ is $R$-$\aone$-local as well.

Point (2) is immediate from (1).  Point (3) follows immediately from Lemma~\ref{prop:Rlocalizationofgroups}.  For point (4), consider the map $\mathbf{G} \to \mathbf{G}_R$.  Taking contractions yields a morphism $\mathbf{G}_{-1} \to (\mathbf{G}_R)_{-1}$.  The construction of $R$-localization, together with the fact that $(-)_{-1}$ preserves exact sequences implies that $(\mathbf{G}_R)_{-1}$ is an $R$-local $\aone$-nilpotent sheaf of groups and therefore, there is an induced morphism $(\mathbf{G}_{-1})_R \to (\mathbf{G}_R)_{-1}$.  By definition of contraction, this map is an isomorphism if $\mathbf{G}$ is abelian, and one deduces it is an isomorphism in general by induction on the $\aone$-nilpotence class of $\mathbf{G}$.
\end{proof}

\subsubsection*{A technical result}
Before moving on to analyze $R$-$\aone$-localization of $\aone$-nilpotent spaces, we establish a useful technical result here about the interaction between $\aone$-localization and $R$-localization in the sense of Section~\ref{ss:rlocalizationofsimplicialpresheaves}.

\begin{prop}
\label{prop:lrxisaonelocal}
Suppose $(\mathscr{X},x)$ is a pointed, connected space.  If $\mathscr{X}$ is a weakly $\aone$-nilpotent and $\aone$-local space (see \textup{Definition \ref{defn:aonenilpotentspace}} for the former notion), then $\LR \mathscr{X}$ is again $\aone$-local.  Moreover, for each integer $i \geq 1$, there are identifications of the form:
\[
\bpi_i(\LR \mathscr{X}) \cong \bpi_i(\mathscr{X})_R
\]
and $\LR \mathscr{X}$ is again a weakly $\aone$-nilpotent space.
\end{prop}

\begin{proof}
By \cite[Lemma 2.2.11(2)]{AWW}, to check a connected space $\mathscr{Y}$ is $\aone$-local, it suffices to show that $\bpi_1(\mathscr{Y})$ is strongly $\aone$-invariant and $\bpi_i(\mathscr{Y})$ is strictly $\aone$-invariant for $i \geq 2$.

Replacing $\mathscr{X}$ by $\Laone \mathscr{X}$ if necessary, we may assume that $\mathscr{X}$ is simplicially fibrant and $\aone$-local.  Set $\bpi := \bpi_1(\mathscr{X},x)$ and consider the first stage of the $\aone$-Postnikov tower for $\mathscr{X}$, i.e., the morphism $\mathscr{X} \to B\bpi$.  Write $\widetilde{\mathscr{X}}$ for the simplicial homotopy fiber of $\mathscr{X} \to B\bpi$.  Since $\bpi$ is strongly $\aone$-invariant, $B\bpi$ is $\aone$-local as well, so $\widetilde{\mathscr{X}}$ is $\aone$-local by appeal to Lemma~\ref{lem:homotopyfiberaonelocal}.  Furthermore $\widetilde{\mathscr{X}}$ is simplicially $1$-connected by assumption.  Of course, $\bpi_i(\widetilde{\mathscr{X}}) \cong \bpi_i(\mathscr{X})$ for $i \geq 2$ as well.

The morphism $\mathscr{X} \to B\bpi$ is a locally $\aone$-nilpotent morphism of simplicially connected spaces by the identifications of the preceding paragraph.  If we apply $\LR$ to this simplicial fiber sequence, then Lemma~\ref{lem:Rlocalizationnilpotentfiberlemma} guarantees that there is again a simplicial fiber sequence of the form
\[
\LR \widetilde{\mathscr{X}} \longrightarrow \LR \mathscr{X} \longrightarrow \LR B\bpi.
\]
Since the simplicial connectivity of $\LR \widetilde{\mathscr{X}}$ may be checked stalkwise, and $\LR$ commutes with formation of stalks, since stalkwise $R$-localization preserves connectivity, we conclude that $\LR \widetilde{\mathscr{X}}$ is again simplicially $1$-connected.

Corollary~\ref{cor:Rlocalizationofnilpotentisnilpotent} allows us to conclude that $\bpi_i(\LR \widetilde{\mathscr{X}}) \cong \bpi_i(\mathscr{X})_R$.  Since $\bpi_i(\mathscr{X})$ is strictly $\aone$-invariant as $\mathscr{X}$ is $\aone$-local, and we know that tensoring with $R$ preserves strict $\aone$-invariance, we conclude that $\bpi_i(\LR \mathscr{X}) \cong \bpi_i(\LR \widetilde{\mathscr{X}})$ is strictly $\aone$-invariant for all $i \geq 2$.  Likewise, from Theorem~\ref{thm:localizationofaonenilpotentgroups}, we conclude that $\bpi_1(\LR \mathscr{X}) \cong \bpi_1(\LR B\bpi)$ is an $R$-local strongly $\aone$-invariant sheaf of groups.

We already know that $\bpi_1(\mathscr{X})_R$ is an $R$-local $\aone$-nilpotent sheaf of groups. To establish the local $\aone$-nilpotence of the higher homotopy sheaves, simply observe that because $R$-localization is given by tensoring with $R$ and $R$ is flat as a $\Z$-module, if we take any locally finite $\bpi$-central series for $\bpi_i(\mathscr{X})$, then tensoring that series with $R$ provides a locally finite $\bpi_R$-central series for $\bpi_i(\mathscr{X})_R$.
\end{proof}

\subsubsection*{$R$-$\aone$-localizations of $\aone$-nilpotent spaces}
We may now construct a convenient, functorial model for the $R$-$\aone$-localization of a weakly $\aone$-nilpotent spaces.

\begin{thm}
\label{thm:aonelocalizationofnilpotentspaces}
Suppose $k$ is a perfect field and $\mathscr{X} \in \Spc_k$ is a pointed, connected, $\aone$-local space.  If $\mathscr{X}$ is weakly $\aone$-nilpotent (see \textup{Definition~\ref{defn:aonenilpotentspace}}), then the following statements hold.
\begin{enumerate}[noitemsep,topsep=1pt]
\item The canonical map $\bpi_i^{\aone}(\mathscr{X})_R \to \bpi_i^{\aone}(\LR \mathscr{X})$ is an isomorphism for every $i \geq 1$.  Moreover, $\LR \mathscr{X}$ is connected, $\bpi_1^{\aone}(\LR \mathscr{X})$ is an $R$-local $\aone$-nilpotent sheaf of groups, and $\LR \mathscr{X}$ is again locally $\aone$-nilpotent.
\item The canonical map $\mathscr{X} \to \LR \mathscr{X}$ is an $R$-$\aone$-weak equivalence.
\item The $R$-$\aone$-localization functor commutes with formation of finite products for weakly $\aone$-nilpotent spaces.
\end{enumerate}
\end{thm}

\begin{proof}
The first point is precisely Proposition~\ref{prop:lrxisaonelocal}.
To check that $\mathscr{X} \to \LR\mathscr{X}$ is an $R$-$\aone$-weak equivalence, we proceed as follows.  Observe that $\mathscr{X} \to \LR\mathscr{X}$ is an $R$-local simplicial weak equivalence since it is so stalkwise.  Therefore, appeal to \cite[Proposition V.3.2]{BK} implies that the map $R[\mathscr{X}] \to R[\LR \mathscr{X}]$ is stalkwise a weak-equivalence, which means it is an isomorphism in the derived category of Nisnevich sheaves of $R$-modules.  It follows that the map remains an isomorphism after applying $\Laone^{ab}$, so the map $\mathscr{X} \to \LR \mathscr{X}$ is an $R$-$\aone$-equivalence as well.

For the final point, note that both $\LR$ and $\Laone$ commute with formation of finite products by construction (see Proposition~\ref{prop:LRproperties} for the former).  Since products of weakly $\aone$-nilpotent spaces are weakly $\aone$-nilpotent, the conclusion of third point follows from the first two.
\end{proof}

The above theorem has the following very useful consequence.

\begin{cor}
\label{cor:raoneweakequivalencesaredetectedonhomotopysheaves}
Suppose $k$ is a perfect field.  A morphism $f: \mathscr{X} \to \mathscr{Y}$ of pointed, weakly $\aone$-nilpotent spaces is an $R$-$\aone$-weak equivalence if and only if the induced maps $f_*: \bpi_i^{\aone}(\mathscr{X})_R \to \bpi_i^{\aone}(\mathscr{Y})_R$ are isomorphisms for every $i \geq 1$.
\end{cor}

We conclude this section by observing that our model for $R$-$\aone$-localization behaves well in fiber sequences.

\begin{thm}
\label{thm:Raonelocalizationoffibersequences}
Suppose
\[
\mathscr{F} \longrightarrow \mathscr{E} \longrightarrow \mathscr{B}
\]
is an $\aone$-fiber sequence of pointed $\aone$-connected spaces.  If $\mathscr{E}$ and $\mathscr{B}$ are weakly $\aone$-nilpotent, then
\[
\LR \mathscr{F} \longrightarrow \LR \mathscr{E} \longrightarrow \LR \mathscr{B}
\]
is again an $\aone$-fiber sequence.
\end{thm}

\begin{proof}
By applying $\Laone$ to the fiber sequence, we may assume that the sequence in question is a simplicial fiber sequence of pointed simplicially connected and $\aone$-local spaces.  Moreover, Corollary~\ref{cor:weaklyaonenilpotentfibersequences} implies that $\mathscr{F}$ is weakly $\aone$-nilpotent.

The canonical map $\LR \mathscr{F} \to \operatorname{hofib}(\LR \mathscr{E} \to \LR \mathscr{B})$ is a simplicial weak equivalence by appeal to Lemma~\ref{lem:Rlocalizationnilpotentfiberlemma}. On the other hand, Theorem~\ref{thm:aonelocalizationofnilpotentspaces} guarantees that the canonical maps $\mathscr{F} \to \LR \mathscr{F}$, $\mathscr{E} \to \LR \mathscr{E}$ and $\mathscr{B} \to \LR \mathscr{B}$ are $R$-$\aone$-weak equivalences.  In total, we conclude that the map
\[
\mathscr{F} \longrightarrow \operatorname{hofib}(\LR \mathscr{E} \to \LR \mathscr{B})
\]
is an $R$-$\aone$-weak equivalence.
\end{proof}

\begin{rem}
\label{rem:homotopyfiberofmapofsimplespacescanbenilpotent}
As in the classical situation, even if $f: \mathscr{E} \to \mathscr{B}$ is a morphism of $\aone$-{\em simple} spaces, the homotopy fiber of $f$ may be $\aone$-nilpotent.  Here is a rather geometric example: over any field $k$, the proof of Theorem~\ref{thm:bglnoddaonesimple} in the case where $n = 0$ yields an $\aone$-fiber sequence of the form
\[
\pone \longrightarrow B\gm{} \longrightarrow BSL_2,
\]
where we have identified $\pone$ with $SL_2/\gm{}$.  The middle term is $\aone$-simple and $BSL_2$ is even $\aone$-simply connected.  As mentioned in Remark~\ref{rem:nonabelianaonenilpotent}, $\bpi_1^{\aone}(\pone)$ is non-abelian and $\aone$-nilpotent.
\end{rem}

\section{Applications}
\label{s:applications}
H. Hopf observed that compact Lie groups have the same real cohomology as products of spheres \cite{Hopf} and J.-P. Serre \cite[V.4]{Serre} established the same results after inverting smaller sets of primes. More precisely, suppose one is given a compact Lie group $G$.  Serre called a prime $p$ {\em regular for $G$} if there exists a product of spheres $X = \prod_{i = 1}^{\ell} S^{n_i}$ and a map $f: X \to G$ such that $f_*: H_*(X,\Z/p) \to H_*(G,\Z/p)$ is an isomorphism.  If $p$ is regular for $G$, then $G$ is $p$-locally a product of spheres.  Later, B. Harris \cite{Harris} observed that similar $p$-local decompositions exist for various homogeneous spaces of compact Lie groups.  In this section, we study motivic analogs of these and related $p$-local decompositions.

\subsection{Self-equivalences of motivic spheres}
\label{ss:Suslin}
We study self-maps of motivic spheres that are weak equivalences after suitable localization.  To construct such maps, we review some constructions of Suslin and then analyze computations of motivic Brouwer degree for the relevant self-maps.  We begin by recalling some facts about Grothendieck-Witt rings and motivic Brouwer degrees.

Fix an arbitrary field $k$.  We use the notation $Q_{2n-1}$ for the odd-dimensional smooth affine quadric defined by the equation $\sum_{i=1}^n x_iy_i = 1$ in ${\mathbb A}^{2n}$ with coordinates $x_1,\ldots,x_n,y_1,\ldots,y_n$.  Projection onto the $x$-variables determines a morphism $Q_{2n-1} \to {\mathbb A}^n \setminus 0$ that is a torsor under a vector bundle, in particular an $\aone$-weak equivalence; we also refer to this morphism as the ``universal unimodular row".  For any integer $n \geq 1$, the quadric $Q_{2n-1}$ is $\aone$-weakly equivalent to $\Sigma^{n-1}\gm{\sma n}$.  We also use the notation $Q_{2n}$ for the smooth affine quadric defined by the equation $\sum_{i=1}^n x_iy_i = z(1-z)$ in ${\mathbb A}^{2n+1}$ with the evident coordinates.  By \cite[Theorem 2]{ADF}, this quadric is $\aone$-weakly equivalent to ${\pone}^{\sma n}$ (in fact, all of these assertions hold over $\Spec \Z$).

\subsubsection*{Grothendieck-Witt groups and motivic Brouwer degree}
Morel described the ring of homotopy endomorphisms of motivic spheres by establishing the following result.

\begin{thm}[Morel]
Suppose $k$ is a field.  For any integer $n \geq 2$, there is an isomorphism of rings:
\[
\operatorname{deg}: [Q_{2n-1},Q_{2n-1}]_{\aone} \longrightarrow GW(k).
\]
\end{thm}

\begin{proof}
For a casual exposition, see \cite[\S 1]{MICM} and see \cite{Cazanave} for a very explicit ``classical" description of the above degree map.  The result as stated follows by combining a number of statements from \cite{MField}.  We observed above that for any integer $n \geq 1$, $Q_{2n-1}$ is $\aone$-weakly equivalent to $\Sigma^{n-1}\gm{\sma n}$.  Appealing to \cite[Corollary 6.43]{MField}, one knows that there are isomorphisms of abelian groups of the form $\bpi_{n-1,n}^{\aone}(Q_{2n-1}) \cong \K^{MW}_0(k)$ for $n \geq 3$; the corresponding isomorphism for $n = 2$ follows from \cite[Theorem 7.20]{MField}.  By \cite[Lemma 3.10]{MField}, one knows that $\K^{MW}_0(k) \cong GW(k)$ as abelian groups.  The fact that the ring structure on homotopy endomorphism induced by composition corresponds to the ring structure in the Grothendieck--Witt group follows from \cite[Theorem 7.35]{MField}.
\end{proof}

We will refer to the isomorphism $\operatorname{deg}$ as the {\em motivic Brouwer degree}.  In order to apply our localization techniques, we will need to observe that $GW(k)$ simplifies after inverting primes, at least under suitable assumptions on the base field.

\begin{lem}
\label{lem:gwwhen2isinvertible}
Suppose $A$ is a ring which is $2$-divisible and $k$ is a field that is not formally real.  The rank map $GW(k) \to \Z$ induces an isomorphism $GW(k) \tensor A \isomt A$.
\end{lem}

\begin{proof}
There is an isomorphism of rings $GW(k) \cong \Z \times_{\Z/2} W(k)$.  If $k$ is a field that is not formally real, then $W(k)$ is a $2$-primary torsion group \cite[Proposition 31.4(6)]{EKM}.  It follows that $GW(k) \tensor A \cong A$ as $A$ is $2$-divisible.
\end{proof}

\begin{rem}
If $k$ is formally real the situation is rather different. For example, when $k = \real$, then $GW(k) \cong \Z \oplus \Z$.
\end{rem}

\subsubsection*{Suslin matrices and motivic degree}
Recall that there is a quotient morphism
\[
q: SL_n \to SL_n/SL_{n-1}
\]
where the latter is isomorphic to $Q_{2n-1}$ by the map sending a matrix to the first row times the first column of its inverse.  This morphism factors the ``first row" morphism $SL_n \to {\mathbb A}^n \setminus 0$, where the induced map $SL_n/SL_{n-1} \to {\mathbb A}^n \setminus 0$ is Zariski locally trivial with affine space fibers and thus an $\aone$-weak equivalence.

Let $R$ be any commutative ring, and let $a = (a_1,\ldots,a_n)$ be a unimodular row in $R$.  Such a row corresponds to a morphism of schemes $\Spec R \to {\mathbb A}^n \setminus 0$.  In this case, we can always find $b = (b_1,\ldots,b_n)$ such that $a b^t = 1$.  Suslin showed \cite[Theorem 2]{Suslinstablyfree} (see also \cite[p. 489 Point (b)]{Suslinstablyfree}), that there exists $S_n(a,b) \in SL_n(R)$ whose first row is equal to $(a_1,a_2,a_3^2,\ldots,a_n^{n-1})$.  The ``universal" $S_n$ corresponds to a morphism
\[
S_n: Q_{2n-1} \to SL_n;
\]
we refer to such morphisms as {\em Suslin matrices}.  The composite map
\[
\phi_n := q \circ S_n : Q_{2n-1} \longrightarrow Q_{2n-1}
\]
has a motivic Brouwer degree, which we now compute.

\begin{lem}
\label{lem:degreeofsuslin}
Suppose $k$ is a field.  For any integer $n \geq 2$, the following formula holds in $GW(k)$:
\[
\operatorname{deg}(\phi_n) = \begin{cases} \langle 1 \rangle & \text{ if } n = 2, \\
\frac{(n-1)!}{2}h & \text{ if } n > 2; \end{cases}
\]
here $h$ is the class of the hyperbolic form.
\end{lem}

\begin{proof}
The degree of the element in question coincides with the degree of the map ${\mathbb A}^n \setminus 0 \to {\mathbb A}^n \setminus 0$ given in coordinates by $(a_1,\ldots,a_n) \to (a_1,a_2,a_3^2\ldots,a_n^{n-1})$.  As explained in the proof of \cite[Theorem 4.10]{AsokFaselSpheres}, the degree of this map can be computed as follows.  By elementary manipulations, we can view the element $[a_1,a_2,a_3^2\ldots,a_n^{n-1}]$ as a symbol in $\K^{MW}_n(k)$.  If $n = 2$, the map in question is the identity map.

If $n > 2$, then we proceed as follows.  Following \cite[p. 51]{MField} write $\epsilon = - \langle -1 \rangle$, then $h = 1 - \epsilon$.  If $r$ is any positive integer, set $r_{\epsilon} = \sum_{i=1}^r \langle (-1)^{i-1} \rangle$.  A straightforward computation shows that $r_{\epsilon}s_{\epsilon} = (rs)_{\epsilon}$.  Appealing to \cite[Lemma 3.14]{MField}, one knows that $[a^m]=m_\epsilon [a]$.  Therefore, for any integer $n \geq 3$.
\[
[a_1,a_2,a_3^2\ldots,a_n^{n-1}] = ((n-1)!)_{\epsilon}[a_1,a_2,\ldots,a_n].
\]
Since $n \geq 3$, $(n-1)!$ is even, and we conclude that $((n-1)!)_{\epsilon} = \frac{(n-1)!}{2}h$ as asserted.
\end{proof}

\subsubsection*{Self equivalences of motivic spheres}
\begin{prop}
\label{prop:suslinisalocalweakequivalence}
If $k$ is a field that is not formally real, then the map $\phi_n$ is an $\aone$-weak equivalence after inverting $(n-1)!$.
\end{prop}

\begin{proof}
The maps $\phi_n$ are all defined over $\Spec \Z$.  By standard base-change results, it suffices to prove that $\phi_n$ is an $\aone$-weak equivalence over a non-formally real perfect subfield (e.g., the prime field if $k$ has positive characteristic or $\Q[i]$ if $k$ has characteristic $0$).

Set $R = \Z[\frac{1}{(n-1)!}]$.  By the equivalent characterizations of what it means for a map to be an $R$-$\aone$-equivalence (see Definition~\ref{defn:Raonelocalspaces} and apply the variant of Lemma~\ref{lem:equivalentcharacterizationsofbeinganRequivalence} in that context), it suffices to show that if $\mathscr{W}$ is an arbitrary $R$-$\aone$-local space, then the map induced by $\phi_n$
\[
[{\mathbb A}^n \setminus 0,\mathscr{W}]_{\aone} \longrightarrow [{\mathbb A}^n \setminus 0,\mathscr{W}]_{\aone}
\]
is a bijection.  Because ${\mathbb A}^n \setminus 0$ is $\aone$-connected, we can assume without loss of generality that $\mathscr{W}$ is pointed and connected.  In that case, the canonical map from pointed to free homotopy classes of maps is a surjection.  Thus, it suffices to prove that the map of pointed sets is a bijection.

However, the sets in question are, by definition $\bpi_{n-1,n}^{\aone}(\mathscr{W})(k)$.  For any integer $n \geq 2$, the group $\bpi_{n-1,n}^{\aone}(\mathscr{W})(k)$ is a $GW(k)$-module with module structure induced by precomposition.  However, we know that $\bpi_{n-1,n}^{\aone}(\mathscr{W})(k)$ is an $R$-module, which means that the $GW(k)$-module structure factors through an action of $GW(k) \tensor R$.  Using Lemma \ref{lem:gwwhen2isinvertible}, we conclude that $GW(k) \tensor R \cong R$.

By appeal to Lemma~\ref{lem:degreeofsuslin} the self-map $\phi_n$ has motivic degree $(n-1)!$ in $GW(k) \tensor R \cong R$.  Since $(n-1)!$ is invertible in $R$ by assumption, the result follows.
\end{proof}

\subsection{Unstable splittings of special groups and associated homogeneous spaces}
\label{ss:unstablesplittings}
We now analyze splittings of classical split reductive groups using the results of the preceding section.  In particular, we construct unstable splittings of $SL_n$ and $Sp_n$ (Theorem~\ref{thm:suslinsplitting} and Corollary~\ref{cor:symplecticsuslinsplitting}) as well as splittings of associated symmetric spaces (see Corollary~\ref{cor:symmetricspacesplitting}).  We then use these results to deduce splittings for all the stable classical groups (see Proposition~\ref{prop:splittingofstableclassicalgroupsI} and Corollary~\ref{cor:splittingofstableclassicalgroupsII}).

\subsubsection*{Splittings for the special linear group}
If $m < n$, then we may consider the composite morphism:
\[
\Phi_m: Q_{2m-1} \stackrel{S_m}{\longrightarrow} SL_{m} \longrightarrow SL_n
\]
where the second morphism is given as the inclusion $X \mapsto diag(X,Id_{n-m})$.  Then, we obtain a morphism
\[
\Phi: Q_{3} \times \cdots \times Q_{2n-1} \longrightarrow SL_n
\]
by taking the composite of $\prod_{m \leq n} \Phi_m$ and the product map $SL_n^{\times n-1} \longrightarrow SL_n$.  The following result about the morphism $\Phi$ corresponds to Theorem~\ref{thmintro:suslinsplitting} from the introduction.

\begin{thm}
\label{thm:suslinsplitting}
Assume $k$ is a field that is not formally real.  The morphism
\[
\Phi: Q_3 \times \cdots \times Q_{2n-1} \longrightarrow SL_n
\]
becomes an $\aone$-weak equivalence after inverting $(n-1)!$.
\end{thm}

\begin{proof}
The map sending $X \in SL_n$ to its first row and the first column of its inverse defines a morphism $SL_n \to Q_{2n-1}$.  The quotient $SL_n/SL_{n-1}$ exists as a smooth scheme, and the morphism just described defines an isomorphism $SL_n/SL_{n-1} \isomt Q_{2n-1}$.  By Theorem~\ref{thm:buildingfibersequences}, there is an $\aone$-fiber sequence of the form $Q_{2n-1} \to BSL_{n-1} \to BSL_{n-1}$, which yields an $\aone$-fiber sequence of the form:
\[
SL_{n-1} \longrightarrow SL_n \longrightarrow Q_{2n-1}
\]
by simplicially looping.  By appeal to Theorem~\ref{thm:Raonelocalizationoffibersequences}, this fiber sequence remains a fiber sequence after $R$-$\aone$-localization for any $R \subset \Q$.

Under the hypothesis on $k$, the composite map $S_n: Q_{2n-1} \to SL_n \to Q_{2n-1}$ becomes an $\aone$-weak equivalence after inverting $(n-1)!$ by appeal to Proposition~\ref{prop:suslinisalocalweakequivalence}.  Therefore, taking $R = \Z[\frac{1}{(n-1)!}]$, the fiber sequence is split after $R$-$\aone$-localization, and the product map $Q_{2n-1} \times SL_{n-1} \to SL_n$ is an $\aone$-weak equivalence after inverting $(n-1)!$.  The result then follows by a straightforward induction.
\end{proof}

\begin{rem}
Since $R$-localization commutes with the formation of finite products, and since $Q_{2n-1}$ is a retract of $SL_n$ after inverting $(n-1)!$, we conclude that $Q_{2n-1}$ is an $\aone$-$h$-space after inverting $(n-1)!$.
\end{rem}

\subsubsection*{Splittings for the symplectic group}
In \cite[Proposition 3.3.3]{AsokFaselKO}, we observed that a slight modification of the Suslin matrix can be viewed as a morphism $P_n: Q_{4n-1} \to Sp_{2^{n-1}}$.

\begin{lem}
\label{lem:ishomotopic}
The morphism $P_n: Q_{4n-1} \to Sp_{2^{n-1}}$ lifts uniquely up to $\aone$-homotopy to a morphism $Q_{4n-1} \to Sp_{2n}$ in such a way that the composite map $Q_{4n-1} \to Sp_{2n} \hookrightarrow SL_{2n}$ is $\aone$-homotopic to the morphism $S_{2n}$ considered above.
\end{lem}

\begin{proof}
The modified Suslin matrices are obtained from usual Suslin matrices by left and right multiplying by matrices that are elementary and thus naively $\aone$-homotopic to the identity.  It follows from this observation that the morphism $P_n$ is $\aone$-homotopic to a morphism $Q_{4n-1} \to SL_{2^{n-1}}$ defined by the appropriate Suslin matrix.  Therefore, if we can lift $P_n$ through $Sp_{2n}$, it follows immediately that the composite with the inclusion $Sp_{2n} \to SL_{2n}$ is $\aone$-homotopic to $S_{2n}$.

To that end, recall that there is a Cartesian square of closed-immersion group homomorphisms of the form:
\begin{equation}\label{eqn:cartesian}
\xymatrix{
Sp_{2n-2} \ar[r]\ar[d] &SL_{2n-1} \ar[d] \\
Sp_{2n} \ar[r] &SL_{2n};
}
\end{equation}
here the vertical maps are the usual stabilization maps for symplectic and special linear groups respectively.  This Cartesian square defines an isomorphism of quotients $Sp_{2n}/Sp_{2n-2} \isomt SL_{2n}/SL_{2n-1}$ and thus an isomorphism $Sp_{2n}/Sp_{2n-2} \cong Q_{4n-1}$.

Now, there is an $\aone$-weak equivalence $Q_{4n-1} \to {\mathbb A}^{2n} \setminus 0$, and we know that ${\mathbb A}^{2n} \setminus 0$ is at least $\aone$-$(2n-2)$-connected.  Appealing to Theorem~\ref{thm:buildingfibersequences} and looping, we deduce that there is an $\aone$-fiber sequence of the form
\[
Sp_{2n-2} \longrightarrow Sp_{2n} \longrightarrow Q_{4n-1},
\]
i.e., the $\aone$-homotopy fiber of $Sp_{2n-2} \to Sp_{2n}$ is $\Omega Q_{4n-1}$, which is at least $\aone$-$(2n-3)$-connected.

We now proceed by induction.  Given a morphism $Q_{4n-1} \to Sp_{2N}$, we can analyze inductively the obstructions to lifting along the inclusion $Sp_{2N-2}$.  These obstructions live in $H^i(Q_{4n-1},\bpi_i^{\aone}({\mathbb A}^{2N} \setminus 0))$.  The group $H^i(Q_{4n-1},\bpi_i^{\aone}({\mathbb A}^{2N} \setminus 0))$ vanishes unless $i = 0$ or $2n-1$ \cite[Lemma 4.5]{AsokFaselSpheres}, and the sheaf $\bpi_i^{\aone}({\mathbb A}^{2N} \setminus 0))$ is trivial if $i \leq 2N-2$ by the connectivity observation of the previous paragraph.  If all of these obstructions vanish, then a lift necessarily exists.  The result then follows by a straightforward induction.
\end{proof}

As was the case with the special linear group, these modified Suslin matrices yield a morphism
\[
\Phi: Q_{3} \times Q_{7} \times \cdots \times Q_{4n-1} \longrightarrow Sp_{2n}.
\]
We then obtain a splitting analogous to that of Theorem~\ref{thm:suslinsplitting}.

\begin{cor}
\label{cor:symplecticsuslinsplitting}
Assume $k$ is a field that is not formally real.  The morphism:
\[
\Phi: Q_{3} \times Q_{7} \times \cdots \times Q_{4n-1} \longrightarrow Sp_{2n}
\]
becomes an isomorphism in $\ho{k}$ after inverting $(2n-1)!$.
\end{cor}

\begin{proof}
We repeat the proof of Theorem~\ref{thm:suslinsplitting} appealing to the $\aone$-fiber sequence
\[
Sp_{2n-2} \longrightarrow Sp_{2n} \longrightarrow Q_{4n-1}.
\]
Then, we simply observe that the morphism $Q_{4n-1} \to Sp_{2n}$ uniquely lifting $P_n$ when composed with the map $Sp_{2n} \to Q_{4n-1}$ gives an $\aone$-homotopy endomorphism of ${\mathbb A}^{2n} \setminus 0$ with degree $(n-1)!$ after Lemma \ref{lem:ishomotopic}.
\end{proof}

\subsubsection*{Splittings for some homogeneous spaces}
B. Harris generalized some of Serre's splittings for compact Lie groups to some classes of homogeneous spaces \cite{Harris}.  We now establish some analogs of the splitting results he established.  Let $X_n = SL_{2n}/Sp_{2n}$.  Using the isomorphism $X_{n+1} \cong SL_{2n+1}/Sp_{2n}$ obtained from Diagram \eqref{eqn:cartesian}, one sees that there is an $\aone$-fiber sequence of the form
\[
X_n \longrightarrow X_{n+1} \longrightarrow Q_{4n+1}.
\]
The Suslin morphism $S_{2n+1}: Q_{4n+1} \to SL_{2n+1}$ thus yields a morphism $Q_{2n+1} \to X_{n+1}$.

\begin{cor}
\label{cor:symmetricspacesplitting}
Assume $k$ is a field that is not formally real.  There is an $\aone$-weak equivalence
\[
X_{n+1} \cong Q_5 \times Q_9 \times \cdots \times Q_{4n+1}
\]
after inverting $(2n-1)!$.
\end{cor}

\begin{proof}
There is an $\aone$-fiber sequence of the form
\[
Sp_{2n} \longrightarrow SL_{2n} \longrightarrow X_{n}.
\]
Appealing to the proofs of Theorem~\ref{thm:suslinsplitting} and Corollary~\ref{cor:symplecticsuslinsplitting} and recalling Lemma~\ref{lem:ishomotopic}, we conclude that after inverting $(2n-1)!$ the left-hand morphism is split and corresponds to the inclusion $Q_3 \times Q_7 \times \cdots \times Q_{4n-1} \to \prod_{2 \leq m \leq n} Q_{2m-1}$.
\end{proof}

\subsubsection*{Infinite products}
We need to make some comments about infinite products of spaces.

\begin{lem}
\label{lem:infiniteproducts}
Suppose $\mathscr{X}_n \in \Spc_k$, $n \in {\mathbb N}$, is a sequence of pointed fibrant, simplicially connected and $\aone$-local spaces.  Suppose for each positive integer $m$ there exists an integer $c(m)$ such that $\mathscr{X}_N$ is at least $\aone$-$m$-connected for every $N > c(m)$.  If we consider the system of finite products $\prod_{i=0}^n \mathscr{X}_n$ as  a directed system with respect to the evident inclusion maps arising from the base-point, then the induced map
\[
\operatorname{colim}_n \prod_{i=0}^n \mathscr{X}_i \longrightarrow \prod_{i \in {\mathbb N}} \mathscr{X}_i
\]
is an isomorphism on $\aone$-homotopy sheaves and thus is an $\aone$-weak equivalence.
\end{lem}

\begin{proof}
Under the stated connectivity hypotheses, for any given integer $j$, the $j$-th $\aone$-homotopy sheaf of the left hand side stabilizes at a finite stage: more precisely, the map $\bpi_j(\prod_{i=1}^N \mathscr{X_i}) \to \bpi_j(\prod_{i=1}^{c(j)} \mathscr{X}_i)$ induced by the projection is an isomorphism for all $N > c(j)$.  It follows from \cite[Theorem IX.3.1]{BK} that the map in the statement induces an isomorphism on homotopy sheaves and is thus an $\aone$-weak equivalence.
\end{proof}

\begin{ex}
\label{ex:increasingconnectivity}
If $i < n$, then the spaces $K(\Z(n),2n-i)$ are at least $\aone$-$(n-i)$-connected.  Likewise, the spaces ${\mathbb A}^n \setminus 0$ are at least $\aone$-$(n-2)$-connected.  In the sequel, we will use Lemma~\ref{lem:infiniteproducts} with these sequences of spaces.
\end{ex}

\subsubsection*{Rational splittings of the stable linear and symplectic groups}
Using the facts about infinite products from the preceding section, we now establish splitting results for stable linear and symplectic groups.  Before Theorem~\ref{thm:suslinsplitting}, we constructed morphisms
\[
\prod_{i=2}^n Q_{2i-1} \longrightarrow SL_n
\]
for each integer $n$.  If we take the colimit of the maps on the left hand side with respect to inclusions of base-points, and the maps on the right hand side with respect to the inclusions $SL_n \hookrightarrow SL_{n+1}$ described before Theorem~\ref{thm:suslinsplitting}, we obtain a morphism
\[
\colim_n \prod_{i=2}^n Q_{2i-1} \longrightarrow SL.
\]
The same construction can be performed for the symplectic groups (see the discussion preceding Corollary~\ref{cor:symplecticsuslinsplitting}) to obtain a morphism
\[
\colim_n \prod_{i=2}^n Q_{4i-1} \longrightarrow Sp.
\]
We use these morphisms in the next statement.

\begin{prop}
\label{prop:splittingofstableclassicalgroupsI}
There are diagrams of the form:
\[
\begin{split}
\prod_{i\geq 2} Q_{2i-1} \longleftarrow &\colim_n \prod_{2 \leq i \leq n} Q_{2i-1} \longrightarrow SL, \text{ and } \\
\prod_{i\geq 2} Q_{4i-1} \longleftarrow &\colim_n \prod_{2 \leq i \leq n} Q_{4i-1} \longrightarrow Sp,
\end{split}
\]
where all morphisms are $\Q$-$\aone$-weak equivalences.
\end{prop}

\begin{proof}
Theorem~\ref{thm:suslinsplitting} implies that the maps $\prod_{2 \leq i \leq n} Q_{2i-1} \longrightarrow SL_n$ become $\aone$-weak equivalences after inverting $(n-1)!$ and thus are all $\Q$-$\aone$-weak equivalences.  By definition, $SL = \colim_n SL_n$ where the colimit is taken with respect to the standard inclusions as block diagonal matrices.  Since filtered colimits of $\aone$-weak equivalences are $\aone$-weak equivalences \cite[\S Lemma 2.2.12]{MV}, it follows that the induced morphism
\[
\colim_n \prod_{2 \leq i \leq n} Q_{2i-1} \longrightarrow SL
\]
is a $\Q$-$\aone$-weak equivalence as well.  That the map
\[
\colim_n \prod_{2 \leq i \leq n} Q_{2i-1} \longrightarrow \prod_{i\geq 2} Q_{2i-1}
\]
is an $\aone$-weak equivalence is an immediate consequence of Lemma~\ref{lem:infiniteproducts}.  The argument for $Sp$ is formally identical to that for $SL$: one makes appeal to Corollary \ref{cor:symplecticsuslinsplitting} instead of Theorem~\ref{thm:suslinsplitting}.
\end{proof}

\subsubsection*{Rational splittings of the stable orthogonal group}
Assume through the remainder of this paragraph (i.e., until the beginning of Section \ref{ss:rationalmotivicspheres}) that one works over a base field $k$ having characteristic not equal to $2$ (this assumption will be made explicit in theorem statements below, and it simplifies the discussion).  We may use the above results to give a corresponding decomposition for the stable orthogonal group as well.  Let $O_n$ (resp. $SO_n$) be the usual split (special) orthogonal group, i.e., the (special) orthogonal group of the split quadratic form in $n$ variables.

There are closed immersion group homomorphisms $O_n \to GL_n$ and $H_n': GL_n \to O_{2n}$ ({mapping a matrix $M$ to the block-diagonal matrix $\mathrm{diag}(M,(M^{-1})^t)$); we will also use the hyperbolic morphism $H_n: GL_n \to Sp_{2n}$ defined in an analogous way.  Explicitly, if $X$ is a smooth affine scheme, then the composite $X \to BGL_n \to B_{et}O_{2n}$ sends a vector bundle $\mathscr{V}$ on $X$ to the bundle $\mathscr{V} \oplus \mathscr{V}^{\vee}$ equipped with its canonical symmetry automorphism: if $c: \mathscr{V} \to \mathscr{V}^{\vee\vee}$ is the canonical isomorphism, then the symmetric space structure is given by the matrix \begin{tiny}$\begin{pmatrix} 0 & 1 \\ c & 0 \end{pmatrix}$\end{tiny}.  Likewise, the composite $X \to BGL_n \to BSp_{2n}$ sends a vector bundle $\mathscr{V}$ on $X$ to $\mathscr{V} \oplus \mathscr{V}^{\vee}$ equipped with its canonical antisymmetry automorphism.

There are stabilization homomorphisms $O_n \to O_{n+1}$ for every integer $n \geq 2$.  The composite homomorphisms
\[
\begin{split}
O_n \longrightarrow &\; GL_n \longrightarrow Sp_{2n}, \text{ and } \\
Sp_{2n} \longrightarrow &\; GL_{2n} \longrightarrow O_{4n}
\end{split}
\]
are compatible with the aforementioned stabilization maps in the sense that the diagrams:
\[
\xymatrix{
O_n \ar[r]\ar[d] & Sp_{2n} \ar[d] \\
O_{n+1}\ar[r] & Sp_{2n+2}
} \text{ and }
\xymatrix{
Sp_{2n} \ar[r]\ar[d] & O_{4n} \ar[d] \\
Sp_{2n+2}\ar[r] & O_{4n+4}
}
\]
are $\aone$-homotopy commutative.  One therefore obtains stable homomorphisms
\[
O \longrightarrow Sp \;\;\;\text{ and } Sp \longrightarrow O
\]
that are well-defined in the $\aone$-homotopy category.  We would like to show that these homomorphisms are $\aone$-weak equivalences after inverting suitable primes, but our statements are slightly complicated because of connectivity issues: the stable symplectic group $Sp$ is $\aone$-connected, but as the next remark explains, the stable orthogonal group is not.

\begin{rem}
Since $2$ is invertible in the base field, the stable orthogonal group $O$ has $2$ connected components corresponding to elements of determinant $\pm 1$.  As a consequence we do not know that $O$ has a ``well-behaved" $\aone$-localization because Theorem~\ref{thm:aonelocalizationofnilpotentspaces} does not necessarily apply.  In the analogous situation in topology, this problem can be rectified by passing to the special orthogonal group, which is connected as a topological space.  However, in the algebro-geometric setting, while the special orthogonal group $SO$ is the component of $O$ corresponding to elements of determinant $1$, it is {\em not} $\aone$-connected.  Indeed, there is an $\aone$-fiber sequence of the form
\[
Spin \longrightarrow SO \longrightarrow B_{\et}\mu_2
\]
classifying the spin double cover.  As a consequence of Lemma~\ref{lem:spinisaoneconnected}, it follows that there is an injective morphism $\bpi_0^{\aone}(SO) \to \bpi_0^{\aone}(B_{\et}\mu_2)$.  The composite morphism $SO \to \bpi_0^{\aone}(B_{\et}\mu_2)$ is a sheafification of the spinor norm homomorphism (see, e.g., \cite{Bass}; it coincides with this homomorphism on sections over essentially smooth local rings), which is non-trivial in general.
\end{rem}

\begin{lem}
\label{lem:spinisaoneconnected}
Assume $k$ is a field. The stable group $Spin$ is $\aone$-connected and $\aone$-simple.  Moreover, the map $Spin \to SO \to O$ induces an isomorphism on $\aone$-homotopy sheaves in degrees $\geq 1$ (pointed by the identity element).
\end{lem}

\begin{proof}
Granted $\aone$-connectedness, $\aone$-simplicity is immediate since all spaces in question are $\aone$-$h$-groups.  To see that $Spin$ is $\aone$-connected, we proceed as follows.  To check that a space $\mathscr{X}$ is $\aone$-connected, it suffices by \cite[Lemma 6.1.3]{MStable} to check that $\pi_0(\Singaone \mathscr{X}(L))$ consists of a single element for every finitely generated separable extension field $L/k$.  Since $Spin$ is, by construction, a filtered colimit, every element of $\Singaone Spin(L)$ lies in $\Singaone Spin_n(L)$ for some integer $n$.  Thus, it suffices to show that $\pi_0(\Singaone Spin_n (L))$ is trivial for $n$ sufficiently large.  To to check the latter statement, it suffices to know that every element of $Spin_n(L)$ is $\aone$-homotopic to the base-point.  Recall from that if $G$ is a split, reductive $k$-group scheme, and $R$ is a $k$-algebra, then the elementary subgroup $E(R)$ is the subgroup of $G(R)$ generated by root subgroups (see \cite[Definition 1.5]{Abe} for a more precise definition).  Any element of the elementary subgroup is $\aone$-homotopic to the identity.  It follows from \cite[Corollary 2]{Matsumoto}, that $Spin_n(L)$ coincides with its elementary subgroup if $n \geq 5$ (which guarantees it has rank $\geq 2$) and we conclude that $\pi_0(\Singaone Spin_n(L)) = \ast$ in that range as well.
\end{proof}

\begin{rem}
While the blanket assumption that $k$ has characteristic not equal to $2$ is in place in the preceding statement, the assertion that $Spin_n$ is $\aone$-connected for any $n \geq 5$ holds even without that assumption.  Thus, the assertion that the stable group $Spin$ is $\aone$-connected also holds without that assumption; for this reason, we have not included any hypothesis on the field explicitly in the statement.
\end{rem}

We consider the composite map $Spin \to SO \to Sp$.  The homomorphism $Sp_{2n} \to O_{4n}$ discussed above lifts uniquely through a morphism $Sp_{2n} \to Spin_{4n}$ (use \cite[Exercise 6.5.2(iii)]{Conrad}) and these lifts are compatible with stabilization maps.  Therefore, we obtain a morphism $Sp \to Spin$ factoring the map $Sp \to O$ described above.

\begin{thm}
\label{thm:orthogonalsymplectic}
Suppose $k$ is a field that is not formally real and has characteristic not equal to $2$, and $R \subset \Q$ is a subring in which $2$ is invertible.  The morphisms $Spin \to Sp$ and $Sp \to Spin$ just defined are mutually inverse $\aone$-weak equivalences after inverting $2$.
\end{thm}

\begin{proof}
The composite $O_n \to O_{4n}$ induces a morphism $B_{\et}O_n \to B_{et}O_{4n}$ that stabilizes to a morphism $\psi: B_{\et}O \to B_{\et}O$. For any smooth affine $k$-scheme $X$, if $(\mathscr{V},\varphi)$ is a symmetric bilinear space, then the formulas for the hyperbolic morphisms described above show that the composite $X \to B_{\et}O \to B_{\et}O$ of the classifying map of $(\mathscr{V},\varphi)$ with $\psi$ sends $[(\mathscr{V},\varphi)] \mapsto \langle 1, 1, -1,-1 \rangle [(\mathscr{V},\varphi)]$ and therefore induces multiplication by this class at the level of $\aone$-homotopy sheaves.  Likewise, one shows that the other composite $BSp \to BSp$ induced by $Sp_{2n} \to Sp_{8n}$ also coincides with multiplication by $\langle 1,1,-1,-1\rangle$ and thus induces multiplication by this class in homotopy sheaves as well.

We now use the observation above to deduce corresponding statements about the induced self-maps of the stable group $Spin$ by passing to suitable connected covers.  The proof of \cite[Theorem 5]{SchlichtingTripathi} shows that $B_{\Nis}O$ is the $\aone$-connected component of the trivial torsor in $B_{\et}O$.  Likewise, $B_{\Nis}Spin$ is $\aone$-$1$-connected by appeal to Lemma~\ref{lem:spinisaoneconnected}.  The map $Spin \to O$ induces a morphism $B_{\Nis}Spin \to B_{\Nis}O \to B_{\et}O$ which makes $B_{\Nis}Spin$ into the $\aone$-$1$-connected cover of $B_{\et}O$.  Therefore, the composite of $B_{\Nis}Spin \to B_{\et}O$ and the map $B_{et}O \to B_{\et}O$ described in the previous paragraph lifts uniquely up to $\aone$-homotopy through a morphism $B_{\Nis}Spin \to B_{\Nis}Spin$.  By the uniqueness of this lift, this map coincides up to $\aone$-homotopy with the self-map $Spin \to Spin$ defined as the composite $Spin \to Sp \to Spin$ after taking loop spaces.

It follows from the discussion above that the composite map $Spin \to Spin$ induces multiplication by the class of $\langle 1,1,-1,-1\rangle$ in $GW(k)$ after passing to $\aone$-homotopy sheaves.  If $k$ is not formally real, then Lemma~\ref{lem:gwwhen2isinvertible} shows that $GW(k) \tensor R \cong R$ via the rank map.  Since $\langle 1,1,-1,-1\rangle$ has rank $4$ we conclude that the induced map on homotopy sheaves is an isomorphism after inverting $2$.  Likewise, the map $Sp \to Sp$ is an isomorphism on $\aone$-homotopy sheaves after inverting $2$.
\end{proof}

\begin{rem}
The assumption that $k$ has characteristic not equal to $2$ in Theorem~\ref{thm:orthogonalsymplectic} arises implicitly in our appeal to \cite{SchlichtingTripathi} where this assumption is used to analyze the higher Grothendieck--Witt space.
\end{rem}

The next result, which follows by combining Theorem~\ref{thm:orthogonalsymplectic} and Proposition~\ref{prop:splittingofstableclassicalgroupsI}, is an unstable variant of a result originally announced by F. Morel \cite{MRational}; it completes the description of rational $\aone$-homotopy groups of stable classical groups (away from fields of characteristic $2$).

\begin{cor}
\label{cor:splittingofstableclassicalgroupsII}
If $k$ is a field that is not formally real and has characteristic not equal to $2$, then the composite morphism:
\[
\prod_{i \geq 1} Q_{4i-1} \longrightarrow Sp \longrightarrow Spin,
\]
where the first morphism is that from \textup{Proposition~\ref{prop:splittingofstableclassicalgroupsI}} and the second is the morphism constructed just prior to \textup{Theorem~\ref{thm:orthogonalsymplectic}}, is a $\Q$-$\aone$-weak equivalence.
\end{cor}

\subsection{On rational homotopy sheaves of motivic spheres}
\label{ss:rationalmotivicspheres}
In this section, we combine the results of the previous section with some results about rationalized K-theory to deduce information about the structure of unstable $\aone$-homotopy sheaves of spheres.  We will freely use facts about Voevodsky's motivic Eilenberg--Mac Lane spaces; we refer the reader to \cite[\S 3]{VMEM} and the references therein for more discussion about these objects.  For any integer $n \geq 0$, we write $\Z(n)$ for a model of Voevodsky motivic complex (see, e.g., \cite[Definition 3.1]{MVW}).  Viewing $\Z(n)$ as a chain complex of Nisnevich sheaves of abelian groups, its homology sheaves are concentrated in degrees $\geq -n$.  If $m \geq n$, we use the notation
\[
K(\Z(n),m)
\]
to denote the Eilenberg--Mac Lane space attached to the complex $\Z(n)[m]$; this space represents motivic cohomology in $\ho{k}$ (see \cite[\S 2]{VRed} for a summary and \cite{Deligne} for details).

\subsubsection*{Rationalized $K$-theory}
The motivic cohomology of $BGL_n$ is generated as a module over motivic cohomology of a point by Chern classes $c_1,\ldots,c_n$ in bi-degrees $(2i,i)$ \cite{Pushin}.  As a consequence, for every integer $n$, and every integer $i < n$, the $i$-th Chern class corresponds to an $\aone$-homotopy class of maps
\[
\begin{split}
c_i: &BGL_n \longrightarrow K(\Z(i),2i), \\
c_i: &BGL \longrightarrow K(\Z(i),2i)
\end{split}
\]
Since $BGL_1$ is already $\aone$-local \cite[\S 4 Proposition 3.8]{MV}, and $\Z(1) = \gm{}[-1]$, the map $c_1: BGL_1 \to K(\Z(1),2)$ is an $\aone$-weak equivalence essentially by definition.

Taking products of Chern classes, there are induced morphisms $BGL_n \to \prod_{i=1}^n K(\Z(i),2i)$ and a morphism
\[
c: BGL \longrightarrow \prod_{n \geq 1} K(\Z(n),2n).
\]
Likewise, Chern classes induce maps $BSL_n \to \prod_{i=2}^n K(\Z(i),2i)$ and a morphism
\[
c': BSL \longrightarrow \prod_{n \geq 2} K(\Z(n),2n).
\]
Direct sum of vector bundles equips the classifying spaces $BGL$ and $BSL$ of the stable general and special linear groups with the structure of $\aone$-$h$-spaces; these spaces are thus $\aone$-simple.  We know that $SL_n$ and $SL$ are all $\aone$-connected spaces.  Taking simplicial loops, one obtains a morphism
\[
\Omega c': SL \longrightarrow \prod_{n \geq 2} K(\Z(n),2n-1).
\]
The next result analyzes these morphisms.

\begin{prop}
\label{prop:stablerationalequivalences}
The following maps are $\Q$-$\aone$-weak equivalences:
\begin{enumerate}[noitemsep,topsep=1pt]
\item the map $c: BGL \to \prod_{n \geq 1} K(\Z(n),2n)$,
\item the map $c': BSL \to \prod_{n \geq 2} K(\Z(n),2n)$, and
\item the map $\Omega c': SL \to \prod_{n \geq 2} K(\Z(n),2n-1)$.
\end{enumerate}
\end{prop}

\begin{proof}
For the first statement, it suffices to show that the induced map on homotopy sheaves after rationalization is an isomorphism.  The map of the statement induces a morphism of presheaves on $\Sm_k$ of the form:
\[
[\Sigma^i U_+,BGL]_{\aone} \longrightarrow [\Sigma^i U_+, \prod_{n \geq 1} K(\Z(n),2n)]_{\aone}
\]
The left hand side is precisely $K_i(U)$ by appeal to \cite[\S 4 Theorem 3.13]{MV} (see \cite[\S 8 Remark 2 p. 1162]{SchlichtingTripathi} for some corrections).  Since $[\Sigma^i U_+, K(\Z(n),2n)]_{\aone} = H^{2n-i,n}(U)$ the universal property of a product allows us to identify the right hand side with $\prod_{n \geq 1} [\Sigma^i U_+, K(\Z(n),2n)]_{\aone}$ and thus $\prod_{n \geq 1} H^{2n-i,n}(U)$.  Now, one knows that, after tensoring the groups on the both sides with $\Q$, this map is an isomorphism; this is the degeneration of the motivic spectral sequence, which identifies the graded pieces for the $\gamma$-filtration in K-theory with Bloch's higher Chow groups (this is originally due to Bloch \cite{Bloch}, but see \cite{LevineQ} for a corrected version).  In the context of motivic homotopy theory, see \cite[Theorem 5.3.10]{Riou} and \cite[Remark 6.2.3.10]{Riou}.

For the second statement, observe that there is a commutative diagram
\[
\xymatrix{
BGL \ar[d]\ar[r]^-c & \prod_{i \geq 1} K(\Z(i),2i) \ar[d] \\
B\gm{} \ar[r]^-{c_1}& K(\Z(1),2),
}
\]
where the left hand vertical map is induced by the determinant and the right hand vertical map is projection onto the first factor.  As remarked above, the bottom morphism here is an $\aone$-weak equivalence.  The left hand vertical map fits into an $\aone$-fiber sequence where the $\aone$-homotopy fiber is $BSL$.  The $\aone$-homotopy fiber of the right hand vertical map is $\prod_{i \geq 2} K(\Z(i),2i)$.  The induced map on homotopy fibers is the morphism $c'$ by construction.  Since the morphisms $c$ and $c_1$ are $\Q$-$\aone$-weak equivalences, the latter by the preceding point, one concludes that $c'$ is a $\Q$-$\aone$-weak equivalence as well.  The final statement follows by looping the previous one.
\end{proof}

\begin{rem}
\label{rem:finitechernvsinfinitechern}
Looping the map $BSL_n \to \prod_{i=2}^n K(\Z(i),2i)$ one obtains a morphism $SL_n \to \prod_{i=2}^n K(\Z(i),2i-1)$.  Taking colimits on both sides, one obtains a morphism
\[
\colim_n SL_n \longrightarrow \colim_n \prod_{i=2}^n K(\Z(i),2i-1).
\]
The colimit on the left hand side is $SL$, and there is thus a morphism
\[
SL \longrightarrow \colim_n \prod_{n \geq 2} K(\Z(n),2n-1) \longrightarrow \prod_{n \geq 2} K(\Z(n),2n-1).
\]
By appeal to Lemma~\ref{lem:infiniteproducts} the right hand map is an $\aone$-weak equivalence. Since $S^1$ is a compact object, taking simplicial loops commutes with formation of filtered (homotopy) colimits.  It follows that the morphism described in the previous display coincides with $\Omega c'$.
\end{rem}

\subsubsection*{Rational homotopy of motivic spheres: odd-dimensional quadrics}
Another classical result of Serre is that the morphism $S^{2n-1} \to K(\Z,2n-1)$ corresponding to the fundamental class (alternatively, the first non-trivial stage of the Postnikov tower) is a rational weak equivalence.  Because of the weak equivalence ${\mathbb A}^n \setminus 0 \cong \Sigma^{n-1}\gm{\sma n}$, we conclude that there is an isomorphism $\mathbf{M}({\mathbb A}^n \setminus 0) \cong \Z \oplus \Z(n)[2n-1]$ in Voevodsky triangulated category of motives \cite[Corollary 15.3]{MVW}.  Thus, for any integer $n \geq 2$, $H^{2n-1,n}({\mathbb A}^n \setminus 0,\Z) \cong \Z$.  A choice $\xi$ of generator can thus be realized by an $\aone$-homotopy class of maps ${\mathbb A}^n \setminus 0 \to K(\Z(n),2n-1)$.  Above, we described an $\aone$-weak equivalence $Q_{2n-1} \to {\mathbb A}^n \setminus 0$.  So the preceding map also defines a morphism
\begin{equation}
\label{eqn:generatoranminus0}
Q_{2n-1} \longrightarrow K(\Z(n),2n-1).
\end{equation}
We now establish an analog of this result in $\aone$-homotopy theory, which corresponds to the first part of Theorem~\ref{thmintro:rationalizedspheres}.

\begin{thm}
\label{thm:rationalizedmotivicspheres}
If $k$ is a field that is not formally real, then for every integer $n > 1$, the map
\[
Q_{2n-1} \longrightarrow K(\Z(n),2n-1)
\]
of \eqref{eqn:generatoranminus0} is a $\Q$-$\aone$-weak equivalence.
\end{thm}

\begin{proof}
After Proposition~\ref{prop:splittingofstableclassicalgroupsI} there is a $\Q$-$\aone$-weak equivalence of the form $\prod_{n \geq 2} Q_{2n-1} \isomt SL$.  Composing this $\Q$-$\aone$-weak with $\Omega c'$, we conclude that there is an induced $\Q$-$\aone$-weak equivalence of the form
\[
\prod_{n \geq 2} Q_{2n-1} \longrightarrow \prod_{n \geq 2} K(\Z(n),2n-1)
\]
and we now give an alternative identification of the component maps.

The map $Q_{2n-1} \to SL_n$ is given by a Suslin morphism, while the map $SL_n \to K(\Z(n),2n-1)$ is given by the loops on the $n$-th Chern class.  Since $Q_{2n-1} \cong \Sigma^{n-1} \gm{\sma n}$, it follows from \cite[\S 4 Theorem 3.13]{MV} that $SK_1(Q_{2n-1}) \cong K_0(k) \cong \Z$.  Suslin showed that the class of the composite map $Q_{2n-1} \to SL_n \to SL$ in $SK_1(Q_{2n-1})$ is a generator \cite[Theorem 2.3 and Corollary 2.7]{SuslinMennicke}.

Since rationally the motivic cohomology of $Q_{2n-1}$ coincides with the rationalized K-theory, we conclude that the Suslin matrix is a rational multiple of the class $\xi$.  In fact, we can be slightly more precise.  The motivic cohomology of $SL_n$ is computed in \cite{Williams}.  In particular, one knows $H^{2n-1,n}(SL_n,\Z) \cong \Z$ and the relevant generator is the image via pullback along the map $SL_n \to {\mathbb A}^n \setminus 0$ of the corresponding generator of $H^{2n-1,n}({\mathbb A}^n \setminus 0)$.  In particular, the composite with the Suslin morphism corresponds to $(n-1)!$ times the generator.  Thus, after inverting $(n-1)!$ this morphism coincides with the map $Q_{2n-1} \to K(\Z(n),2n-1)$ from \eqref{eqn:generatoranminus0}.  Therefore, we conclude that the map
\[
\prod_{n \geq 2} Q_{2n-1} \longrightarrow \prod_{n \geq 2} K(\Z(n),2n-1)
\]
defined by taking colimits of the finite products of the morphisms in \eqref{eqn:generatoranminus0} is also a $\Q$-$\aone$-weak equivalence (see Remark~\ref{rem:finitechernvsinfinitechern} for related discussion).

Finally, for every integer $n \geq 2$, the map of \eqref{eqn:generatoranminus0} is, by construction, a retract of the map displayed in the previous paragraph.  Therefore, we conclude that \eqref{eqn:generatoranminus0} is a $\Q$-$\aone$-weak equivalence as well, which is what we wanted to show.
\end{proof}

\begin{rem}
If $k$ is a subfield of $\cplx$, then the conclusion of Theorem~\ref{thm:rationalizedmotivicspheres} is compatible with complex realization.  Indeed, the complex realization of $K(\Z(n),2n-1)$ is homotopy equivalent to $K(\Z,2n-1)$ \cite[Corollary 3.48]{VMEM} while the complex realization of ${\mathbb A}^n \setminus 0$ or $Q_{2n-1}$ is homotopy equivalent to $S^{2n-1}$.  The induced map on realizations yields the map $S^{2n-1} \to K(\Z,2n-1)$.
\end{rem}

\begin{rem}
Theorem \ref{thm:rationalizedmotivicspheres} is an unstable variant of a result due originally to F. Morel in the stable case \cite[Theorem 5.2.2]{Mpi0} or \cite{MRational}: if $-1$ is a sum of squares in $k$, then the rationalized motivic sphere spectrum is simply the rational motivic Eilenberg--Mac Lane spectrum (see \cite[Theorem 16.2.13]{CisinskiDeglise} for a proof).
\end{rem}

Theorem~\ref{thm:rationalizedmotivicspheres} may be used to reformulate the Beilinson--Soul\'e vanishing conjecture (see, e.g., \cite[II.4.3.4]{Kahn} for a statement of the conjecture).  Before stating the reformulation, we include the following result explained to us by F. D\'eglise.

\begin{lem}
\label{lem:BSsubfields}
If $E/k$ is an arbitrary field extension, then the map
\[
H^{p,q}(k,\Q) \longrightarrow H^{p,q}(E,\Q)
\]
is injective.  In particular, if the Beilinson--Soul\'e vanishing conjecture holds for $E$, then it holds for $k$ as well.
\end{lem}

\begin{proof}
One may reduce to the case where $E/k$ is a finitely generated extension by a continuity argument; this follows from, e.g., \cite[Proposition 1.24]{DegliseFinCor}.  To treat the case of finitely generated extensions, it suffices to treat two special cases: $E/k$ is finite and $E = k(t)$.

If $E/k$ is finite, the result follows by existence of transfers in motivic cohomology: indeed, there is a pushforward map $H^{p,q}(E,\Q) \to H^{p,q}(k,\Q)$ such that the composite of pushforward and pullback is multiplication by the degree of the extension; injectivity follows immediately.  Finally, we treat the case where $E = k(t)$.  In that case, by \cite[Proposition 6.1.1 and Corollaire 6.1.3]{Deglise} we know that
\[
H^{p,q}(k(t),\Q) = H^{p,q}(k) \oplus_{x \in ({\mathbb A}^1_k)^{(1)}} H^{p-1,q-1}(\kappa_x,\Q);
\]
the sum is taken over the codimension $1$ points of ${\mathbb A}^1_k$, and the pullback map corresponding to the inclusion $k \hookrightarrow k(t)$ is the inclusion of the first summand on the right.

The injectivity assertion in the case $E = k(t)$ is actually ``softer" than the more precise assertions above, as observed by the referees.  Indeed, if $k$ is infinite, which we may assume without loss of generality, injectivity can be deduced as follows.  The group $H^{p,q}(k(t),\Q)$ is a colimit of $H^{p,q}(U,\Q)$ where $U$ ranges over the non-empty open subschemes of the affine line.  Since non-empty open subchemes over an infinite field have a rational point, the pullback along the inclusion of that rational point provides the splitting of the statement.
\end{proof}

As a consequence, we obtain the following reformulation of the Beilinson--Soul\'e vanishing conjecture.

\begin{cor}
\label{cor:beilinsonsoule}
Suppose $k$ is a field that is not formally real, and assume $n > 1$ is an integer.  If $L/k$ is an extension, then the following statements are equivalent.
\begin{enumerate}[noitemsep,topsep=1pt]
\item the Beilinson--Soul\'e vanishing conjecture holds for every subfield of $L$ in weight $n$;
\item the group $\bpi_i^{\aone}({\mathbb A}^n \setminus 0)_{\Q}(L) = 0$ for $i \geq 2n-1$.
\item the group $\H_i^{\aone}({\mathbb A}^n \setminus 0,\Q)(L) = 0$ for $i \geq 2n-1$.
\end{enumerate}
\end{cor}

\begin{proof}
Note that
\[
\bpi_i^{\aone}(K(\Z(n),2n-1)(L) = [\Sigma^i \Spec L_+,K(\Z(n),2n-1)]_{\aone} = H^{2n-1-i,n}(\Spec L,\Z).
\]
The Beilinson--Soul\'e vanishing conjecture for fields in weight $n$ is equivalent to:
\[
H^{2n-1-i,n}(\Spec L,\Z) = 0 \text{ if } i \geq 2n-1.
\]
The equivalence of the first two statements is then immediate from Theorem~\ref{thm:rationalizedmotivicspheres} and Lemma~\ref{lem:BSsubfields}.

The equivalence of the second and third statements is an immediate consequence of the Whitehead theorem.  Indeed, consider the map ${\mathbb A}^n \setminus 0 \longrightarrow \tau_{\leq 2n-2}{\mathbb A}^n \setminus 0$ from punctured affine $n$-space to its $(2n-2)$nd $\aone$--Postnikov truncation.  The vanishing of rational homotopy in (2) is equivalent to this map being a $\Q$-$\aone$-equivalence.  Since both spaces in question are simple, this map is a $\Q$-$\aone$-equivalence if and only if it is a $\Q$-$\aone$-homology equivalence by the Whitehead theorem (see Theorem~\ref{thm:whiteheadtheoremforsimplespaces}).
\end{proof}

\begin{rem}
\label{rem:otherreformulations}
Since the Beilinson-Soul\'e vanishing conjecture is known to hold for finite fields, totally imaginary number fields or function fields of curves over finite fields, Corollary~\ref{cor:beilinsonsoule} can also be viewed as a rational computation of the sections of motivic homotopy sheaves of ${\mathbb A}^n \setminus 0$ over such fields.

On the other hand, the vanishing of the integral $\aone$-homology of ${\mathbb A}^n \setminus 0$ in degrees $> 2n-1$ is a special case of a conjecture of F. Morel \cite[Conjecture 6.34]{MField}.  At least with rational coefficients, and over fields in which $-1$ is a sum of squares, Morel's vanishing conjecture in the case of ${\mathbb A}^n \setminus 0$ is thus {\em equivalent} to the Beilinson--Soul\'e vanishing conjecture over all finitely generated separable extensions of the base-field.
\end{rem}

\begin{cor}
\label{cor:rationalcohomologyofmotivicemspaces}
Assume $k$ is a field that is not formally real.  If $\mathbf{M}$ is any strictly $\aone$-invariant sheaf of $\Q$-vector spaces, then for any integer $n \geq 2$, the following statement holds:
\[
H^i(K(\Z(n),2n-1),\mathbf{M}) = \begin{cases}\mathbf{M}(k) & \text{ if } i = 0, \\ 0 & \text{ if } 1 \leq i \leq n-2, \\ \mathbf{M}_{-n}(k) & \text{ if } i = n-1, \text{ and } \\ 0 & \text{ if } i > n-1.\end{cases}
\]
\end{cor}

\begin{proof}
This follows by combining \cite[Lemma 4.5]{AsokFaselSpheres} and Proposition~\ref{prop:homologicalandcohomologicalequivalences}.
\end{proof}

\subsubsection*{Rational homotopy of motivic spheres: even dimensional quadrics}
The motivic analog of even-dimensional spheres are the quadrics $Q_{2n}$, which are $\aone$-homotopy equivalent to ${\pone}^{\sma n}$ (see the beginning of Section~\ref{ss:Suslin}).  Indeed, the complex realization of such spheres is homotopy equivalent to $S^{2n}$.  We also obtain the following motivic analog of a result of a classical result of Serre \cite[IV.4 Corollaire 2]{Serre}.  In order to formulate the statement, we use the $\aone$-EHP sequence of \cite{WickelgrenWilliams}.

\begin{prop}
\label{prop:evenspheres}
Suppose $k$ is a field that is not formally real and $R$ is a subring of $\Q$ in which $2$ is invertible.
\begin{enumerate}[noitemsep,topsep=1pt]
\item There is an $\aone$-fiber sequence of the form:
\[
Q_{2n-1} \longrightarrow \Omega Q_{2n} \longrightarrow \Omega Q_{4n-1}.
\]
\item The sequence from (1) splits to yield an $R$-$\aone$-equivalence of the form
\[
\Omega Q_{2n} \isomto Q_{2n-1} \times \Omega Q_{4n-1}.
\]
\end{enumerate}
\end{prop}

\begin{proof}
The first statement is a special case of: \cite[Theorem 8.3]{WickelgrenWilliams} applied to $\mathscr{X} = Q_{2n-1} \cong \Sigma^{n-1}\gm{\sma n}$ (see \cite[Corollary 1.4]{WickelgrenWilliams}: by hypothesis, $k$ is not formally real, and the parity condition from their statement is satisfied by construction).  Thus, we obtain an $\aone$-fiber sequence of the form:
\[
Q_{2n-1} \longrightarrow \Omega Q_{2n} \longrightarrow \Omega Q_{4n-1}.
\]

For the second point, the $\aone$-fiber sequence in question remains an $R$-$\aone$-fiber sequence, e.g., by appeal to Theorem~\ref{thm:Raonelocalizationoffibersequences}.  The long exact sequence in $\aone$-homotopy sheaves of the resulting fiber sequence takes the form:
\[
\cdots \longrightarrow \bpi_i^{\aone}(Q_{2n-1})_{R} \stackrel{\mathrm{E}}{\longrightarrow} \bpi_i^{\aone}(\Omega Q_{2n})_{R} \stackrel{\mathrm{H}}{\longrightarrow} \bpi_i^{\aone}(\Omega Q_{4n-1})_{R} \stackrel{\mathrm{P}}{\longrightarrow} \cdots
\]
Since $k$ is not formally real, $GW(k) \tensor R \cong R$ by Lemma~\ref{lem:gwwhen2isinvertible}.  Consider the Whitehead square $[\iota,\iota]$ of the identity map $\iota$ on $Q_{2n}$; this corresponds to an $\aone$-homotopy class of maps $Q_{4n-1} \to Q_{2n}$ (see \cite[\S 4.1]{AWW} for further discussion of Whitehead products in this motivic setting).  The computation of $\mathrm{H}$ on the Whitehead square of the identity is contained in \cite[Theorem 4.4.1]{AWW}; in this case, the relevant class lies in $GW(k) \tensor R$.  Indeed, it follows that this class is a unit $GW(k) \tensor R$ under the hypotheses on $k$ and $R$.  In other words, the simplicial loops of the Whitehead square of the identity yields a splitting of the $\aone$-EHP sequence and thus a decomposition of the form:
\[
\Omega Q_{2n} \cong Q_{2n-1} \times \Omega Q_{4n-1}.
\]
In essence, the argument we have given is a straightforward translation of Serre's proof to motivic homotopy theory; the results of \cite{AWW} and \cite{WickelgrenWilliams} guarantee that the necessary technology to make this translation sensible is in place.
\end{proof}

In classical topology there is a $\Q$-fiber sequence $S^{2n} \to K(\Z,2n) \to K(\Z,4n)$, where the first map is given by the fundamental class and the second map is given by the squaring map.  In motivic homotopy theory, the $\aone$-weak equivalence $Q_{2n} \cong {\pone}^{\sma n}$ shows that $\mathbf{M}(Q_{2n}) \cong \Z \oplus \Z(n)[2n]$ in Voevodsky's derived category of motives.  As a consequence $H^{2n,n}(Q_{2n},Z) = \Z$, and a choice of generator corresponds to an $\aone$-homotopy class of maps $Q_{2n} \to K(\Z(n),2n)$.  The following result, which refines and extends Proposition~\ref{prop:evenspheres} is the second statement of Theorem~\ref{thmintro:rationalizedspheres} from the introduction.

\begin{thm}
\label{thm:fibersequenceevenspheres}
If $k$ is a field that is not formally real, then for any integer $n \geq 1$, there is a $\Q$-$\aone$-fiber sequence of the form
\[
Q_{2n} \longrightarrow K(\Z(n),2n) \longrightarrow K(\Z(2n),4n),
\]
where the first morphism is the fundamental class in motivic cohomology, and the second morphism is the squaring map.
\end{thm}

\begin{rem}
The standard proof of the corresponding statement in classical topology goes as follows.  One considers the squaring map $K(\Z,2n) \to K(\Z,4n)$, and then computes the rational cohomology of the homotopy fiber by a Serre spectral sequence argument.  The map from $S^{2n}$ to the homotopy fiber is easily seen to be a rational cohomology isomorphism and the result follows.  Just as with Theorem~\ref{thm:rationalizedmotivicspheres}, our proof of the corresponding motivic result is not a direct translation of this standard proof from topology.  In addition to our lack of a useable Serre spectral sequence, there are ``weight" issues that would arise.  In conjunction with Corollary~\ref{cor:beilinsonsoule}, Theorem~\ref{thm:fibersequenceevenspheres} gives another geometric interpretation of the Beilinson--Soul\'e vanishing conjecture; the resulting statement is our candidate for the natural motivic analog of the vanishing statement for rational homotopy groups of even-dimensional spheres.
\end{rem}

\begin{proof}
The cup square map corresponds to a morphism
\[
K(\Z(n),2n) \longrightarrow K(\Z(2n),4n).
\]
As discussed before the statement, we may choose a map $Q_{2n} \to K(\Z(n),2n)$ corresponding to a generator of $H^{2n,n}(Q_{2n},\Z)$.  Since $\mathbf{M}(Q_{2n}) = \Z \oplus \Z(n)[2n]$, the group $H^{4n,2n}(Q_{2n},\Z)$ vanishes and thus the composite map $Q_{2n} \to K(\Z(n),2n) \to K(\Z(2n),4n)$ is null $\aone$-homotopic.  A choice of a null $\aone$-homotopy yields a morphism fitting into the following diagram:
\[
\xymatrix{
\mathscr{F} \ar[r] &  K(\Z(n),2n) \ar[r] & K(\Z(2n),4n); \\
& Q_{2n} \ar[u]\ar@{-->}[ul]^{\exists} &
}
\]
Fix such a choice once and for all for the rest of this proof.

Since the spaces $K(\Z(n),2n)$ and $K(\Z(2n),4n)$ are $\aone$-simple, it follows from Theorem~\ref{cor:weaklyaonenilpotentfibersequences} that the $\aone$-homotopy fiber is $\aone$-nilpotent.  In that case, Theorem~\ref{thm:Raonelocalizationoffibersequences} implies that the $\Q$-$\aone$-localization of the induced fiber sequence
\[
\mathscr{F} \longrightarrow K(\Z(n),2n) \longrightarrow K(\Z(2n),4n)
\]
is again an $\aone$-fiber sequence.  We want to show that the map $Q_{2n} \to \mathscr{F}$ corresponding to our choice of null-homotopy of the composite is a $\Q$-$\aone$-weak equivalence.  Corollary~\ref{cor:raoneweakequivalencesaredetectedonhomotopysheaves} states that to check this morphism is a $\Q$-$\aone$-weak equivalence, it suffices to show that the induced map on $\Q$-$\aone$-homotopy sheaves is an isomorphism.

We first compute the $\Q$-$\aone$-homotopy sheaves of $\mathscr{F}$.  By appeal to Theorem~\ref{thm:aonelocalizationofnilpotentspaces}, $\bpi_i^{\aone}(\mathrm{L}_{\mathbb{Q}} \mathscr{F}) = \bpi_i^{\aone}(\mathscr{F})_{\Q}$.  The $\Q$-$\aone$-homotopy sheaves of $K(\Z(n),2n)$ and $K(\Z(2n),4n)$ may be computed by using the path-loop fibration
\[
K(\Z(n),2n-1) \longrightarrow P K(\Z(n),2n) \longrightarrow K(\Z(n),2n),
\]
and Theorem~\ref{thm:rationalizedmotivicspheres}.  In particular, we conclude that
\[
\bpi_i^{\aone}(K(\Z(n),2n))_{\Q} \cong \bpi_{i-1}^{\aone}(K(\Z(n),2n-1)_{\Q} \cong \bpi_{i-1}^{\aone}(Q_{2n-1})_{\Q}.
\]
Likewise, we conclude that
\[
\bpi_i^{\aone}(K(\Z(2n),4n))_{\Q} \cong \bpi_{i-1}^{\aone}(Q_{4n-1})_{\Q}.
\]
Putting these facts together, rationalizing the long exact sequence in $\aone$-homotopy sheaves yields an exact sequence of the form:
\[
\cdots \longrightarrow \bpi_{i}^{\aone}(Q_{2n-1})_{\Q} \longrightarrow \bpi_{i}^{\aone}(Q_{4n-1})_{\Q} \longrightarrow \bpi_i^{\aone}(\mathscr{F})_{\Q} \longrightarrow \bpi_{i-1}^{\aone}(Q_{2n-1})_{\Q} \longrightarrow \bpi_{i-1}^{\aone}(Q_{4n-1})_{\Q} \longrightarrow \cdots.
\]
We claim that this sequence is canonically split.

Consider the cup-squaring map $K(\Z(n),2n) \to K(\Z(2n),4n)$.  We claim that this map is null-$\aone$-homotopic after taking simplicial loops.  Indeed, the map obtained by simplicial looping is adjoint to a map
\[
\Sigma \Omega K(\Z(n),2n) \longrightarrow K(\Z(2n),4n).
\]
Since cup-products are null-homotopic in a suspension, it follows that the above map is null-$\aone$-homotopic; this observation yields the claimed splitting.

We use this observation now to show that the map $Q_{2n} \to \mathscr{F}$ described before the statement induces an isomorphism on homotopy sheaves, and is therefore a $\Q$-$\aone$-weak equivalence.  We may directly compute the $\Q$-$\aone$-homotopy sheaves of $Q_{2n}$ by appeal to Proposition~\ref{prop:evenspheres}.  Indeed, we observe that $\bpi_i^{\aone}(Q_{2n})_{\Q} \cong \bpi_{i-1}^{\aone}(Q_{2n-1})_{\Q} \oplus \bpi_i^{\aone}(Q_{4n-1})_{\Q}$.  To prove that the induced maps are isomorphisms, we just have to show that the induced maps on each summand are isomorphisms; this involves unwinding the definitions of the various maps.

To see that the map $Q_{2n} \to \mathscr{F}$ induces an isomorphism on the summand $\bpi_{i-1}^{\aone}(Q_{2n-1})_{\Q}$ we proceed as follows.  The summand of $\bpi_{i-1}^{\aone}(Q_{2n-1})_{\Q}$ in the $\bpi_i^{\aone}(Q_{2n})_{\Q}$ arises via the map $Q_{2n-1} \to \Omega \Sigma Q_{2n-1} \cong \Omega Q_{2n}$ in the $\aone$-EHP sequence.  Likewise, the summand of $\bpi_{i-1}^{\aone}(Q_{2n-1})_{\Q}$ in $\bpi_i^{\aone}(\mathscr{F})_{\Q}$ arises from the map $\mathscr{F} \to K(\Z(n),2n)$.  We therefore have a diagram of the form:
\[
\xymatrix{
Q_{2n-1} \ar[r]\ar[d] & \Omega Q_{2n} \ar[r]\ar[d] & \Omega \mathscr{F} \ar[d] \\
K(\Z(n),2n-1) \ar[r] &  \Omega \Sigma K(\Z(n),2n-1) \ar[r] & \Omega K(\Z(n),2n), \\
}
\]
where the top left horizontal morphism arises from the EHP sequence, the leftmost vertical morphism is the fundamental class, the middle vertical morphism is obtained by applying $\Omega \Sigma$ to the leftmost vertical morphism, and the right vertical map is the simplicial loops of the map $\mathscr{F} \to K(\Z(n),2n)$.  By construction the composite of the top right horizontal morphism and the rightmost vertical morphism is the loops on the fundamental class $Q_{2n} \to K(\Z(n),2n)$.  The bottom right horizontal morphism exists by appeal to the suspension isomorphism in motivic cohomology: there is an isomorphism in motivic cohomology $H^{2n-1,n}(Q_{2n-1},\Z) \cong H^{2n,n}(Q_{2n},\Z)$, and the resulting square commutes.  The claim about the summand $\bpi_{i-1}^{\aone}(Q_{2n-1})_{\Q}$ then follows.

The statement about the summand $\bpi_i^{\aone}(Q_{4n-1})_{\Q}$ is obtained similarly.  Indeed, as in the proof of Proposition~\ref{prop:evenspheres}, the summand in question in $\bpi_i^{\aone}(Q_{2n})$ arises from the choice of map $Q_{4n-1} \to Q_{2n}$; for concreteness we fix the one given by the $\aone$-homotopy class of the Whitehead square $[\iota,\iota]$.  This morphism fits into a square of the form:
\[
\xymatrix{
Q_{4n-1} \ar[d]\ar[r]^{[\iota,\iota]} & Q_{2n} \ar[d] \\
K(\Z(2n),4n-1) \ar[r] & \mathscr{F},
}
\]
where the left vertical map is the fundamental class, the bottom horizontal map is the map $\Omega K(\Z(2n),4n) \to \mathscr{F}$ from the fiber sequence structure.  Under the assumptions on $k$, the left vertical map is a $\Q$-$\aone$-weak equivalence, and the bottom horizontal map induces the summand of $\bpi_i^{\aone}(Q_{4n-1})_{\Q}$ in $\bpi_i^{\aone}(\mathscr{F})_{\Q}$.  This square commutes after rationalization.  Indeed, the composite map $Q_{4n-1} \to Q_{2n} \to K(\Z(n),2n)$ is null, and therefore lifts (rationally) to a map $Q_{4n-1} \to K(\Z(2n),4n-1)$.  The group of such homomorphisms is isomorphic to $\Q$, so it suffices to observe that the lifted map is non-zero, in which case it is a multiple of the fundamental class.  Since $[\iota,\iota]$ splits the $\aone$-EHP sequence, it induces the identity $\Omega Q_{4n-1} \to \Omega Q_{4n-1}$ and thus is not null-$\aone$-homotopic.
\end{proof}

\begin{rem}
If $k$ is a subfield of $\cplx$, then the complex realization of the $\Q$-$\aone$-fiber sequence of Theorem~\ref{thm:fibersequenceevenspheres} is the $\Q$-fiber sequence $S^{2n} \to K(\Z,2n) \to K(\Z,4n)$.  This fact follows immediately from the fact that the complex realization of $Q_{2n}$ is $S^{2n}$ and that the complex realizations of motivic Eilenberg-Mac Lane spaces of this form are corresponding ordinary Eilenberg-Mac Lane spaces \cite[Corollary 3.48]{VMEM}.  Indeed, the induced map $H^{2n,n}(Q_{2n},\Z) \to H^{2n}(S^{2n},\Z)$ is precisely the cycle class map, which sends the fundamental class to the fundamental class, and a similar statement holds for the squaring map.
\end{rem}

\begin{footnotesize}
\bibliographystyle{alpha}
\bibliography{a1nilpotent}
\end{footnotesize}
\Addresses
\end{document}